\documentclass[11pt,a4paper]{article}
\usepackage{macro}
\usepackage[a4paper,left=30mm,top=30mm,right=30mm,bottom=30mm]{geometry}

\allowdisplaybreaks[4]   
\title{Density measures and applications}
\author{Friedemann Schuricht \\ \small TU Dresden -
	Fakultät Mathematik\\ \small 01062
Dresden, Germany}

\date{}

\begin{document}

\maketitle

\vspace{10mm}  

\begin{abstract}
The paper, that continuous some previous work of Sch\"onherr \& Schuricht, 
treats density measures on ${\mathbb R}^n$ that concentrate in any
neighborhood of a Lebesgue null set. Such measures are typical for 
purely finitely additive measures. We study their basic properties 
and investigate related integrals. 
Measures taking only the values~0 and~1 are considered as special case.
The results are first applied to weak convergence in
$\cL^\infty(\dom)$. Then we derive integral representations by means of
such measures for several notions of differentiability for 
integrable functions and we show a kind of mean value theorem for
some class of Sobolev functions.  
Finally we provide a new approach to the generalized
Jacobians in the sense of Clarke. 
\end{abstract}

\vspace{10mm}

\section{Introduction}
\label{int}

We study a class of purely finitely additive measures and related integrals in
$\R^n$. Examples of such measures suitable for relevant
applications had hardly been available in the past. The density of a
set $A\subs\R^n$ at point $x$, a standard tool in
geometric measure theory for decades, is given by 
\begin{equation}\label{int-e1}
  \dens_x A := \lim_{\de\dto 0} \frac{\lem(A\cap B_\de(x))}{\lem(B_\de(x))}
\end{equation}
if the limit exists ($\lem$ - Lebesgue measure, $B_\de(x)$ -
ball of radius $\de$ around $x$). Apparently, for fixed $x$, this density 
has not been considered as a set function for a long time. 
But in Schuricht \& Schönherr \cite{trace} it is shown that $\dens_xA$ can be
extended to an additive set function for all Borel sets $A\subs\R^n$.  
It has some similarity to the Dirac measure $\de_x$ concentrated at $x$, but
it also essentially differs. We obviously have for any $A$
\begin{equation*}
  \dens_x(A\setminus B_\de(x)) = \de_x(A\setminus B_\de(x)) = 0
  \qmq{for all}  \de>0\,,
\end{equation*}
and
\begin{equation*}
  \dens_x(B_\de(x)) = \de_x(B_\de(x)) = 1  \qmq{for all} \de>0\,.
\end{equation*}
However, we have the fundamental difference that 
\begin{equation*}
  0 = \dens_x\{x\} \ne \de_x\{x\} = 1\,.
\end{equation*}
While the Dirac measure $\de_x$ is concentrated at $x$, we have that 
$\dens_x$ ``lives'' in any vicinity of $x$ that does not must
contain $x$ itself.  
Taking the limit $\de\dto 0$ we get
\begin{equation*}
  1 = \lim_{\de\dto 0} \dens_x(B_\de(x)) \ne \dens_x(\{x\}) = 0 
\end{equation*}
and, if $\dens_x A>0$, 
\begin{equation*}
  0 = \lim_{\de\dto 0} \dens_x(A\setminus B_\de(x)) \ne
  \dens_xA = \dens_x(A\setminus\{x\})\,. 
\end{equation*}
Any of these limits prevents $\dens_x$ from being $\si$-additive. 
It turns out that $\dens_x$ is a typical example of a purely finitely
additive measure. In particular, it can be identified with a functional 
$f^*$ in the dual of $\cL^\infty(\dom)$, the space of essentially bounded
functions on $\dom\subs\R^n$, 
and it can be considered as an extension of $\de_x$, taken as
functional on the continuous functions, to $\cL^\infty(\dom)$. 
Based on the integration theory for finitely additive measures, we have that
\begin{equation*}
  \df{f^*}{f} = \I{\R^n}{f}{\dens_x} \qmq{for all} f\in\cL^\infty(\dom)\,.
\end{equation*}
We can now ask what the integral gives for general possibly discontinuous 
$f$. According to our previous arguments, the integral uses only 
values of $f$ in any small neighborhood of~$x$, but it does not need a value
$f(x)$. It turns out that it agrees with the precise representative 
of $f$ at $x$ given by
\begin{equation*} \label{int-e4}
  f^\star(x)=\lim_{\de\dto 0} \tfrac{1}{\lem(B_\de(x))}\I{B_\de(x)}{f}{\lem}
\end{equation*}
if the limit exists. Thus the integral with $\dens_x$ gives somehow a
limit of averages near~$x$, even if the limit for $f^\star(x)$ does not exist.  
This is a special and also typical feature for integrals related to finitely
additive measures, that essentially differs from the properties of the usually
considered integrals. One can construct more general finitely additive
{\it density measures}, that are concentrated near an $\lem$-null set 
$\C\subset\cl\dom$, by replacing $B_\de(x)$ with the $\de$-neighborhood $\C_\de$
of $\C$ in \reff{int-e1}. Moreover it can be shown that, up to exceptional
cases, also functions in $\cL^1(\dom)$ (space of $\lem$-integrable functions)
are integrable with respect to such measures. 
In Schuricht \& Schönherr \cite{trace} that kind of measures play a
central role in a general theory of traces, that extends 
the Gauss-Green formula both for $\lem$-integrable vector fields, 
with divergence being a measure, and for $BV$ and Sobolev
functions. In particular it is shown that 
\begin{equation*}
  f^\star(x) = \I{\R^n}{f}{\dens_x}
\end{equation*}
is valid for all $\cL^1$-functions $f$ at $\lem$-a.e. $x\in\dom$. 
This provides an integral representation for the precise representative. 
Overall it turns out that these finitely additive measures 
are applicable in the theory of partial differential equations
and that they develop their potential by properties that are different from
those we are used to. 

A more comprehensive investigation of finitely additive density measures
concentrating near an $\lem$-null set $\C$ and of related integrals 
was started in Schönherr \& Schuricht~\cite{dens}. In particular the
existence of such measures, basic properties, and estimates for the related
integrals are shown.
Moreover, so called \zo measures $\mu$, where 
\begin{equation*}
  \mu(A)\in\{0,1\} \qmq{for all Borel sets $A$\,,} 
\end{equation*}
are studied. Similar as $\dens_x$ they concentrate near a single point $x$, 
but they also concentrate near a direction (i.e. near a ray emanating 
from~$x$). It turns out that the \zo measures are the extreme points of the 
considered set of density measures. 
As application the precise representative is analyzed in more detail 
for BV and Sobolev functions and a new approach to Clarke's generalized
gradients in finite dimension is given in \cite{dens}.  

The present paper continues the investigations from \cite{dens}. 
We extend several results to unbounded sets $\C$ and to unbounded functions. 
But we also derive new integral estimates and we are expanding the
investigation about \zo measures. This is used to extend some results of
Toland \cite{toland} about weak convergence in $\cL^\infty(\dom)$, to provide
integral conditions for several integrability notions for $\cL^1$-functions,
and to extend the new approach from \cite{dens} to generalized Jacobians in
the sense of Clarke. 

In more detail, the paper is organized as follows. 
Since finitely additive measures and related integrals are not widely used, 
in Section~\ref{pl} we provide some brief introduction as needed for our
analysis. Here we focus on differences from classical theory. 
This way we also fix notation and terminology. 
Section~\ref{dm} is devoted to density measures and related integrals. 
Here we extend several results from \cite{dens} for later applications. 
This concerns the expansion to unbounded sets $\C$ and to unbounded
functions $f\in\cL^1(\dom)$. Furthermore we study density measures that
concentrate in a prescribed set $\E\subs\dom$ near $\C$ and
we show the existence of measures that concentrate near infinity.
The integral estimates and related characterizations of the integral are
essentially extended. Some assertion about the existence of a weak$^*$
accumulation point of certain density measures turns out to be very useful.
Though the extension of some results from \cite{dens} is not
difficult, we have to state it for later application. 
In Section~\ref{ep} we consider \zo measures as density measures.
Beyond the results from \cite{dens}, we transfer the characterization 
of \zo measures by ultra filters from Toland \cite{toland} to our setting. 
This gives a precise description where such measures live and, in particular, 
one can verify the existence of a huge amount of \zo measures. For later use
we extend some integral conditions from \cite{toland} to $\cL^1$-functions. 
Section~\ref{app} provides several applications. In Section~\ref{app-wc} 
we start with some treatment of weak convergence in $\cL^\infty(\dom)$. Here
we supplement former results of Toland~\cite{toland} with statements containing
density measures. In several examples we demonstrate that density measures
similar to $\dens_x$ from above are very well suited for practical use. 
In Section~\ref{app-der} we treat several notions of differentiability for
$\cL^1$-functions, as e.g. approximate differentiability. For each case 
we derive integral characterizations by means of density measures
and discuss them for BV and Sobolev functions.
We also introduce some new notion of differentiability for 
(classes of) $\cL^1$-functions that is equivalent to the classical
differentiability of a representative. Some kind of mean value theorem for
certain Sobolev functions concludes this section. 
Finally, Section~\ref{gd} extends the new approach to
Clarke's generalized gradients from \cite{dens} to generalized Jacobians in
the sense of Clarke. More precisely, we characterize these set-valued
derivatives by means of density measures. We then demonstrate that the 
new characterization leads to essentially simplified proofs of the
standard calculus rules.

\bsk

{\it Notation.} 
We use $\cl\R=[-\infty,\infty]$, $\R_\ge=[0,\infty)$, and
$\cl\R_{\ge 0}=[0,\infty]$. $\ccl{\R^n}=\R^n\cup\{\infty\}$
is the usual one point compactification. For a set $A$ we write $\als^c$, 
$\cl A$, $\bd A$, $\op{conv}A$, and $\ch_A$
for its complement, closure, boundary, convex hull, and 
characteristic function, respectively.
For $A\subs\R^n$ we write $\ccl A$ for the closure within $\ccl{\R^n}$.
The $\de$-neighborhood of $A$ is denoted by $A_\de$ and 
$B_\delta(x)$ is the open $\de$-ball around $x$ while $B_\delta^A(x)$ is its
intersection with $A$. The open $\de$-ball $B_\delta(\infty)$
around $\infty$ is given by $\ccl{\R^n}\setminus \cl{B_{1/\de}(0)}$.
By $[x,y]$ we denote the closed line segment between points $x$ and $y$.
For $a\in\R^n$ we use $|a|$ for the Euclidean norm and $a\cdot b$ is the
scalar product of two vectors. $\bor{\dom}$ stands for the Borel sets and  
$\cP(\dom)$ for the power set of $\dom$. If $\mu$ is a measure, then 
$\op{supp}\mu$ is its support and $\reme{\me}{\als}$ its restriction to~$A$. 
We use $\lem$ for the Lebesgue measure on $\R^n$ and
$\ham^k$ for the $k$-dimensional Hausdorff measure. We write 
$\lem$-a.e. for $\lem$-almost everywhere. 
$L^p(\dom)$ is the Lebesgue space of $p$-integrable functions and 
$W^{1,p}(\dom)$ the Sobolev space of $L^p$-functions with $p$-integrable weak
derivative. $\cL^p(\dom)$ and $\cW^{1,p}(\dom)$ stand for the related
spaces of equivalence classes. $\cB\cV(\dom)$  denotes the space of  
(equivalence classes of) functions of bounded variation.
$X^*$ is the dual of the normed space $X$ and $\df{\cdot}{\cdot}$ the 
related duality form. The symbol $\mIsymb$ stands for the 
mean value integral.

\section{Preliminaries about measures and integration}
\label{pl}

The integration theory for finitely additive measures is not so well known. 
Since it is a central tool for our analysis, we provide a short introduction 
for the convenience of the reader, adapted to our needs. This way we also fix  
notation and terminology, which is not uniform in the literature. 
For more comprehensive presentations about finitely additive measures 
we refer to Bhaskara Rao \& Bhaskara Rao~\cite{rao} and 
Dunford \& Schwartz \cite{dunford} or the brief summary in 
Schuricht \& Schönherr \cite{trace}. Let us emphasize that 
{\it we call any additive set function a measure, in contrast to the usual 
standard.}

For a Borel set $\dom\subs\R^n$ we call $\mu:\cB(\dom)\to\ol\R$
a {\it measure} on $\dom$ if $\mu(\emptyset)=0$ and if $\mu$ is {\it
  additive}, i.e. for pairwise disjoint 
$A_k\in\cB(\dom)$ and a finite index set $I\subset\N$ we have
\begin{equation*}
  \mu\Big(\, {\textstyle \bigcup\limits_{k\in I}} A_k\Big) = 
  \sum_{k\in I}\mu(A_k) \,.
\end{equation*}
If $\mu$ is {\it $\si$-additive}, i.e. the previous equality is true with
$I=\N$, then $\mu$ is called {\it $\si$-measure}. 
That additivity is properly defined, a measure cannot attain both values
$\pm\infty$. We say that $\mu$ is {\it positive} if $\mu(A)\ge 0$ 
and it is {\it finite} if $\mu(A)\in\R$ for all $A\in\cB(\dom)$. 
The usual restriction $\reme{\mu}{A}$ of $\mu$ to $A\in\cB(\dom)$ is again a
measure. The {\it positive} and {\it negative part}
$\mu^\pm:\cB(\dom)\to\ol\R_{\ge 0}$ of measure $\mu$ given by
\begin{equation*}
  \mu^\pm(A) := \sup\{ \pm\mu(B) \mid B\subset A\,,\: B\in\cB(\dom)\}\,
\end{equation*}
and the {\it total variation} $|\mu|:\cB(\dom)\to\ol\R_{\ge 0}$ of $\mu$ 
given by
\begin{equation*}
  |\mu|(A) := \sup \Big\{ \sum_{k=1}^l |\mu(A_k)|\:\Big|\: 
  \bigcup_{k=1}^l A_k=A\,,\text{ $A_k\in\cB(\dom)$ 
  pairwise disjoint} \Big\} \,
\end{equation*}
are positive measures with
\begin{equation*}
  |\mu|=\mu^++\mu^-\,.
\end{equation*}
We say that $\mu$ is {\it bounded} if $|\mu(A)|<c$ for all $A$ and some $c>0$,
which is equivalent to $|\mu|(\dom)<\infty$. Measure $\mu$ is called 
{\it pure} if we have for any $\sigma$-measure $\sigma:\cB(\dom)\to\ol\R$ that 
\begin{equation*}
  0\le\sigma\le|\mu| \qmq{implies} \sigma =0\,.
\end{equation*}
This means that a pure measure 
does not contain a $\si$-additive part and, roughly speaking, it is 
purely finitely additive. If $\mu$ is pure, then
also $\mu^\pm$ and $|\mu|$. For a positive measure $\mu$ the 
{\it outer measure} $\mu^*:\cP(\dom)\to\ol\R_{\ge 0}$ given by
\begin{equation*}
  \mu^*(A):=\inf\limits_{\substack{B\in\cB(\dom)\\A\subset B}}\mu(B) \,
\end{equation*}
is {\it subadditive}, i.e. for $A_k\in\cP(\dom)$ and a finite index set
$I\subset\N$ one has 
\begin{equation*}
  \mu^*\Big(\, {\textstyle \bigcup\limits_{k\in I}} A_k\Big) \le 
  \sum_{k\in I}\mu^*(A_k) \,.
\end{equation*}
We say that $A\subset\dom$ is a {\it $\mu$-null} set if $|\mu|^*(A)=0$.  
The space 
\begin{equation*}
  \op{ba}(\dom,\cB(\dom)) := 
  \{\mu:\cB(\dom)\to\R\mid \text{$\mu$ is a bounded measure}\}\,
\end{equation*}
is a Banach space with norm $\|\mu\|=|\mu|(\dom)$. The $\si$-additive and pure
measures are closed subspaces and any 
$\mu\in\op{ba}(\dom,\cB(\dom))$ can be uniquely decomposed
into a $\si$-additive and a pure measure. A measure 
$\me$ is called {\it absolutely continuous} with respect to  
measure $\lme$ ($\me\ac\lme$), 
if for any $\eps>0$ there is some 
$\delta>0$ such that for $A\in\cB(\dom)$
\begin{equation*}
  |\lme(A)|<\delta \qmq{implies} |\me(A)|<\eps \,.
\end{equation*}
We call $\me$ {\it weakly absolutely continuous}
with respect to $\lme$ ($\me\wac\lme$), if for $A\in\cB(\dom)$ 
\begin{equation*}
  |\lme|(A)=0 \qmq{implies} \me(A)=0\,
\end{equation*}
and, then, also $|\mu|\wac\lme$. 
For a (not necessarily bounded) measure $\lme$ on $\dom$, the space
\begin{equation*}
  \op{ba}(\dom,\cB(\dom),\la) :=
  \{ \mu\in\op{ba}(\dom,\cB(\dom))\mid \mu\wac\lme\}\,
\end{equation*}
is a closed subspace of $\op{ba}(\dom,\cB(\dom))$ and, thus, also a Banach
space. It turns out that  measures $\mu\in\bawl{\dom}$, 
that do not charge sets of $\lem$-measure zero, are an appropriate tool for
the investigation of Lebesgue integrable functions. 

Many examples of pure measure in the literature are measures on $\N$.
It turns out that the density of a set $A$ at a point $x$, as used in geometric
measure theory, leads to some typical examples in $\R^n$. 

\begin{example}  \label{pl-s2}
For a Borel set $\dom\subs\R^n$ and for $x\in\cl\dom$ we assume that
\begin{equation*}
  \lem(\dom\cap B_\delta(x))>0 \qmq{for all} \delta>0\,.
\end{equation*}
Then the density of $A\in\cB(\dom)$ at $x$ within $\dom$ is defined by
\begin{equation} \label{pl-s2-1} 
\dens_x^\dom(A) := \lim_{\delta \downarrow 0} 
\frac{\lem(A\cap\dom\cap B_\delta(x))}{\lem(\dom\cap B_\delta(x))}
\end{equation}
as long as the limit exists. This quantity is usually considered for
$\dom=\R^n$ in geometric measure theory. As set function,
$\dens_x^\dom(\cdot)$ is disjointly additive. Using the Hahn-Banach theorem,
we can extend $\dens_x^\dom$ to a measure on $\cB(\dom)$ such that for all $A$
\begin{equation*} 
  \liminf_{\delta \downarrow 0} 
  \frac{\lem(A\cap\dom\cap B_\delta(x))}{\lem(\dom\cap B_\delta(x))}
  \le \dens_x^\dom(A) \le 
  \limsup_{\delta\downarrow 0} 
  \frac{\lem(A\cap\dom\cap B_\delta(x))}{\lem(\dom\cap B_\delta(x))} \,
\end{equation*} 
(cf. Corollary~\ref{dm-s4b} below with $\E=\dom$ and $C=\{x\}$ and notice that
the extension is not unique).
Clearly, $\dens_x^\dom(A)\in[0,1]$ and $\dens_x^\dom(A)=0$ if $\lem(A)=0$. Thus
\begin{equation*}
  \dens_x^\dom\in\bawl{\dom} \,.
\end{equation*}
Moreover, for all $\de>0$,
\begin{eqnarray*}     
  1 
&=& 
  \dens_x^\dom(B_\delta(x)\cap\dom) \: = \:
  \dens_x^\dom (\dom\setminus\{x\})  \,, \\ 
  0
&=&
  \dens_x^\dom(\dom\setminus B_\delta(x)) \hspace{4pt} = \: 
  \dens_x^\dom(\dom\cap\{x\}) \,.
\end{eqnarray*}
Though $\dens_x^\dom$ has some similarity to the Dirac measure $\de_x$ 
concentrated at $x$, it differs substantially. It ``lives'' in any
neighborhood of $x$ while it vanishes at $x$. This behavior   
easily implies that $\dens_x^\dom$ cannot be a $\sigma$-measure, since 
otherwise the limit $\delta\downarrow 0$ would lead to 
$\dens_x^\dom(\dom\setminus\{x\})=0$. Furthermore, for a $\si$-measure $\si$
with $0\le\si\le\dens_x^\dom$, we readily get
\begin{equation*}
  \si(\{x\}) = \si(\dom\setminus\{x\}) = 0 \,.
\end{equation*}
Hence $\si=0$ and $\dens_x^\dom$ is pure by definition.

For an extension of the previous construction we choose  
$\la,\ti\la\in\op{ba}(\dom,\cB(\dom))$ and some Borel set 
$\C\subs\cl\dom$ with 
\begin{equation} \label{pl-s2-4a}
  \la,\ti\la\ge 0\,, \quad \la\wac\ti\la\,, \quad 
  \ti\la(\dom\cap\C)=0\,, \qmq{and}
\end{equation}
\begin{equation}\label{pl-s2-4b}
  \la(\dom\cap\C_\de)>0 \qmq{for all} \de>0\,
\end{equation}
($\C_\de$ is the $\delta$-neighborhood of $C$).
Then, analogously as above, we can define a density $\la_\C^\dom(A)$
of $A\in\cB(\dom)$ at $\C$ within
$\dom$ related to $\la$ 
and we can extend it to a positive measure 
satisfying
\begin{equation*} 
  \liminf_{\delta \downarrow 0} 
  \frac{\la(A\cap\dom\cap\C_\de)}{\la(\dom\cap\C_\de)}
  \le \la_\C^\dom(A) \le 
  \limsup_{\delta\downarrow 0} 
  \frac{\la(A\cap\dom\cap\C_\de)}{\la(\dom\cap\C_\de)} \,.
\end{equation*} 
Then $\la_\C^\dom\in\op{\ba}(\dom,\cB(\dom),\la)$ and, for all $\de>0$,
\begin{equation*}
  1 = \la_\C^\dom(\C_\delta\cap\dom) = \la_\C^\dom (\dom\setminus\C)  \,, \quad
  0 = \la_\C^\dom(\dom\setminus\C_\de) = \la_\C^\dom(\dom\cap\C) \,.
\end{equation*}
Again we obtain that $\la_C^\dom$ is pure. For $0<k<n$ we can e.g. 
consider a $k$-dimensional set $S$ in $\dom=\R^n$ (i.e. $0<\cH^k(S)<\infty$)
and we choose some $\C\subset S$ such that \reff{pl-s2-4a},
\reff{pl-s2-4b} are satisfied with 
$\la=\ti\la=\reme{\cH^k}{S}$. This way we get a 
related pure density measure. In particular, we can 
take $\C=\{x\}$ for some $x\in S$ and obtain some $k$-dimensional density 
at $x$ within $S$. 

Below we study measures $\la_\C^\dom$ for the special case that
$\la\in\bawl{\dom}$ and $\ti\la=\lem$. 
Notice that the previous construction is also possible for $\C\subset\cl\dom$ 
with $\lem(\dom\cap\C)>0$. But then $\la_\C^\dom$ is not pure anymore (cf. 
\cite[Ex. 2.1]{dens}), i.e. it contains a $\si$-additive part. Since
$\si$-measures are well understood, we focus on cases where  
$\lem(\dom\cap\C)=0$.
\end{example}

The special behavior of measures that we have observed in the
previous example is typical for pure measures in $\bawl{\dom}$. 
More precisely, in the case $\lem(\dom)<\infty$ the measure 
$\mu\in\bawl{\dom}$ is pure if and only if there is a decreasing
sequence of sets $\{A_k\}\subset\cB(\dom)$ such that
\begin{equation} \label{pl-pure}
  \lem(A_k)\to 0 \qmq{and} |\mu|(A_k^c)=0 \zmz{for all} k\in\N\,
\end{equation}
(cf. \cite[p. 244]{rao}). If \reff{pl-pure} is satisfied, then 
$\lem(A)=0$ for $A=\bigcap_kA_k$ and, thus, $\mu(A)=0$. Hence, 
a pure measure is concentrated on any neighborhood of $A$. This leads to the
question where a pure measure lives. For a $\si$-measure $\si$ 
we know that it is concentrated on its {\it support} given by
\begin{equation*}
  \supp\sme := \big\{x\in\dom\:\big|\: 
  |\sme|(U\cap\dom)>0 \text{ for all open }U\subset X \text{ with }
  x\in U   \big\} \,.
\end{equation*}
However, \reff{pl-pure} indicates that pure measures $\mu\in\bawl{\dom}$ 
typically vanish on its support and, if we choose $x\in\bd\dom$ for an open
$\dom$ in the previous example, we even have an empty support. To account for
this behavior, we need a more careful description of where the measure lives. 
First we slightly modified define  
the {\it core} of measure $\mu\in\op{ba}(\dom,\cB(\dom))$ by
\begin{equation}\label{pl-core}
  \cor{\me} := \big\{ x\in\R^n\:\big|\: 
  |\me|(U\cap\dom)>0 \text{ for all open }U\subset X \text{ with }
  x\in U   \big\} \,.
\end{equation}
Then, both support and core are closed with
\begin{equation*}
   \supp\mu = \dom\cap \cor\mu \,.
\end{equation*}
Moreover $\dom\cap \cor\mu$ can be empty for a pure measure $\mu$ and 
$\cor\si\setminus\dom$ must not be empty for a $\si$-measure $\si$ 
(e.g. $\si=\lem$ on a bounded open $\dom$). 
For unbounded $\dom$ it is possible that a pure
measure ``lives near $\infty$'', which leads to an empty core 
by lack of compactness
of $\R^n$. For a proper treatment of such cases we (tacitly) use the
compactification
\begin{equation} \label{e-comp}
  \ccl{\R^n}:=\R^n\cup\{\infty\} \,
\end{equation}
instead of $\R^n$ in the definition \reff{pl-core} of core. Here the topology
has to be supplemented by the open balls 
\begin{equation*}
  B_\de(\infty) := \ccl{\R^n}\setminus\cl{B_{1/\de}(0)}\,, \quad \de>0
\end{equation*}
around $\infty$. For some arguments we have to distinguish between the
(usual) closure $\cl A$ of $A$ in $\R^n$ and 
\begin{equation*}
  \ccl A := \mbox{(compactified) closure of $A$ in $\ccl{\R^n}$}\,.
\end{equation*}
In this setting we obtain for $\mu\in\op{ba}(\dom,\cB(\dom))$ that
\begin{equation*} \label{pl-e0}
  \cor\mu\ne\emptyset \zmz{and} \cor\mu\subs\ccl\dom \qmq{if} \me\ne 0\,.
\end{equation*}
Moreover, for all open $U\subs\ccl{\R^n}$ with $\cor\mu\subs U$,
\begin{equation} \label{pl-e1}
  |\me|(U\cap\dom) = |\me|(\dom)\,, \quad
  |\me|(\dom\setminus U)=0
\end{equation}
(cf. \cite[Prop. 2.11]{trace}). Thus, each measure in 
$\op{ba}(\dom,\cB(\dom))$ has its full mass on any neighborhood of the core
while pure measures can vanish on its core. For an appropriate description of
that behavior we call $A\in\cB(\dom)$ an {\it aura} of 
$\mu\in\op{ba}(\dom,\cB(\dom))$ if 
\begin{equation*}
  |\mu|(A^c)=0\,.
\end{equation*}
Though an aura is not unique, combined with the core it allows a proper
description of where a pure measure is concentrated. Using the core, we 
have in addition to \reff{pl-pure}) that, for any $\dom\in \cB(\R^n)$,
\begin{equation*}    
  \mu\in\bawl{\dom} \zmz{is pure} \qmq{if}\z
  \lem(\dom\cap\cor\mu)=0\,.
\end{equation*}
(cf. \cite[p. 24]{trace}).

For an integration theory that covers integrals with pure measures, we need
convergence and equality in measure instead almost every.
For functions $f_k,f:\dom\to\R$ and a measure $\mu$ on $\dom$ 
we have that $f_k$ {\it converges in
measure} $\mu$ to $f$ ($f_k\to f$ i.m. $\mu$) if  
\begin{equation*}
\lim_{k \to \infty} 
|\mu|^*\big(\big\{x\in\dom \;\big|\; |f_k(x) - f(x)| > \eps \big\}\big) = 0 
\qmq{for all} \eps>0\,.
\end{equation*} 
The limit is unique if we identify $f,g:\dom\to\R$ that   
{\it agree in measure} $\mu$ ($f=g$ i.m.), i.e.
\begin{equation} \label{pl-e3}
  |\mu|^*\big(\big\{x\in\dom \:\big|\: |f(x)-g(x)|>\eps \big\}\big) = 0
  \qmq{for all} \eps>0\,.
\end{equation}
For estimates we define that $f\le g$ i.m. $\mu$ if
\begin{equation*}
  |\mu|^*\big(\big\{x\in\dom \:\big|\: g(x)-f(x)<-\eps \big\}\big) = 0
  \qmq{for all} \eps>0\,.
\end{equation*}
A function is called {\it step function} if its range has only finitely many
values.  
We say that $f$ is {\it $\mu$-measurable} if there are Borel measurable step 
functions $f_k$ such that $f_k\to f$ i.m. $\mu$. 

The integral $\I{\dom}{f}{\mu}$ of Borel measurable step functions is just the
usual sum and $\I{\dom}{f}{\mu}$ for a $\mu$-measurable function $f$ is
defined by the usual approximation with step functions if we take
convergence in measure instead of almost everywhere. 
Then integrability of $f$ coincides with that of $|f|$ and $\mu$-integrability
is the same as $|\mu|$-integrability. We have the standard calculus rules
and dominated convergence if we use convergence in measure instead of almost
everywhere. In particular, 
\begin{equation} \label{pl-e5}
  f=g \zmz{i.m. $\mu$} \qmq{if and only if}
  \I{A}{f}{\mu} = \I{A}{g}{\mu} \qmq{for all} A\in\cB(\dom)\,.
\end{equation}
and the integral vanishes if $f=0$ i.m. $\mu$.
Moreover, the integral agrees with the usual one for a $\si$-measure
$\mu$.  
 
While it is sufficient for an integral with a $\si$-measure to integrate over
the support, for general integrals we have to integrate over an aura. For a
more precise indication of the domain of integration we define
\begin{equation*}
  \sI{C}{f}{\mu} := \I{\dom\cap U}{f}{\mu} \qmq{for any open $U$ with}
  \cor\mu\subset C \subset U 
\end{equation*}
which is justified by \reff{pl-e1}.

For $1\le p\le\infty$ we write 
\begin{equation*}
  L^p(\dom) \z \big(L^p_{\rm loc}(\dom)\big)  \qmq{and} \cL^p(\dom) \z
  \big(\cL^p_{\rm loc}(\dom)\big)
\end{equation*}
for the usual space of (locally) Lebesgue $p$-integrable functions 
and the related space of equivalence classes of functions that agree
$\lem$-a.e., respectively. 
For vector-valued functions we use 
\begin{equation*}
  L^p(\dom,\R^m):=L^p(\dom)^m  \qmq{etc. }
\end{equation*}
For the dual of $\cL^\infty(\dom)$ we have that
\begin{equation*}
  \cL^\infty(\dom)^* = \bawl{\dom}
\end{equation*}
if we identify $f^*\in\cL^\infty(\dom)^*$ with $\mu\in\bawl{\dom}$ such that
\begin{equation} \label{pl-dual}
  \df{f^*}{f} = \I{\dom}{f}{\mu} \qmq{for all} f\in\cL^\infty(\dom)\,. 
\end{equation}
Then $\|f^*\|=\|\mu\|=|\mu|(\dom)$. Consequently, for $\mu\in\bawl{\dom}$
all $f\in\cL^\infty(\dom)$ are $\mu$-integrable. But also
unbounded $f\in\cL^1(\dom)$ can be $\mu$-integrable.

\section{Density measures}
\label{dm}

We consider a certain class of pure measures in $\bawl\dom$ as density
measures, which includes the density of a set at a point and some  
extensions as discussed in Example~\ref{pl-s2}. Such measures 
are particularly suited for the treatment of Lebesgue integrable functions,
since they do not charge $\lem$-null sets. Here we continue the investigations
that started in Schönherr \& Schuricht \cite{dens}. In particular we extend
several results to unbounded sets and we provide integral estimates for 
$\cL^1$-functions instead for $\cL^\infty$-functions. Moreover, for later
application, we consider densities as in Example~\ref{pl-s2} but with
respect to a subset $\E\subs\dom$ and we formulate conditions for the 
existence of a weak$^*$ accumulation point of density measures.
Our definition of density measures requires them to be normalized. 
Alternatively, one could drop this condition. Then the set of density measures
would be a convex cone rather than a bounded convex set and it would lack
compactness, which is however important for some of our arguments. 
We partially use the compactification of $\R^n$ and the related
notation  
\begin{equation*}
  \ccl{\R^n}=\R^n\cup\{\infty\}\,, \quad 
  \ccl B = \tx{closure of $B$ in $\ccl{\R^n}$}\,, \quad
  \{\infty\}_\delta = B_\delta(\infty)=\ccl{\R^n}\setminus \cl{B_{1/\delta}(0)} \,
\end{equation*}
in the case of an unbounded $\dom$. 

We always assume that $\dom\in\bor{\R^n}$ and let 
\begin{equation} \label{dm-e0a}
\begin{split}  
  \tx{either} & \hspace{1.2em} \C \subset \cl{\dom} 
  \qmq{be closed with}  \lem(\C\cap\dom) = 0 
  \qmq{and} \C\ne\emptyset\, \\
  \tx{or}  & \hspace{1.2em}   C=\{\infty\} \,. 
\end{split}
\end{equation}
We call $\me\in\bawl{\dom}$  
{\it density measure} (short {\it density}) {\it at} $\C$ (on $\dom$) if 
\begin{equation} \label{dm-e0}
  \me\ge 0 \qmq{and}  \me(\Cd\cap\dom) = \me(\dom) = 1
  \qmq{for all} \delta>0\,.
\end{equation}
The set of all density measures at $\C$ (on $\dom$) is denoted by
\begin{equation*}
\Dens_\C \qquad (\text{briefly also } \Dens_x \text{ if } C={\{x\}})\,.		
\end{equation*}
Notice that $\Dens_\C=\emptyset$ if $\lem(\dom \cap \dnhd{\C}{\delta}) = 0$ 
for some $\delta >0$. Thus $\Dens_\C\ne\emptyset$ implies that 
\begin{equation} \label{dm-e00}
  \lem(\Cd\cap\dom)>0 \zmz{for all} \delta>0\,.
\end{equation}
We call $\C$ a {\it density set} (of $\dom$) if it satisfies \reff{dm-e0a} and
\reff{dm-e00}. 
In Theorem~\ref{dm-s4d} below we see that \reff{dm-e00} is also sufficient
for $\Dens_\C\ne\emptyset$.

\subsection{Basic properties and existence}
\label{dm-bp}

Let us start with the formulation of basic properties of density measures.

\begin{theorem} \label{dm-s1}   
Let $\dom \in \bor{\R^n}$ and let $\C$ be a density set.
\bgl
\item
If $\me\in\Dens_\C$, then
\begin{equation} \label{dm-s1-1}
  \cor\mu\ne\emptyset\,, \quad \cor{\me} \subset \ccl\C\,, \quad
  \|\me\|=|\me|(\dom)=\me(\dom)=1\,,
\end{equation}
and $\me$ is pure. Moreover
\begin{equation} \label{dm-e1}
  \I{\dom}{f}{\me} = \sI{\ccl\C}{f}{\me} = \lim_{\delta \downarrow 0}
  \I{\Cd\cap\dom}{f}{\me} = \lim_{\delta\downarrow 0} 
  \mI{\Cd\cap\dom}{f}{\me} \,
\end{equation}
for all $f\in\cL^\infty(\dom)$.

\item
$\Dens_\C$ is a weak$^*$ compact convex subset of $\bawl{\dom}$. 
\el
\end{theorem}
\noi
Notice that we have $\ccl\C=\C$ if $\C\subset\cl\dom$ is bounded
and $\ccl\C=\C\cup\{\infty\}$ otherwise. More generally, one could consider 
closed density sets $\C\subset\ccl\dom$. Then the closure $\ccl\dom$ of an
unbounded set always contains $\infty$. But, for applications it is more
convenient to take $\cl\dom$. Thus one would have to treat the two
different closures, which increases technicalities. We refrain from doing so
and just include $C=\{\infty\}$ to account for measures that concentrate near
infinity. This approach is sufficient for the applications we have in mind. 

\begin{proof}
For $C\subset\cl\dom$ the statement can be found in \cite{dens} and  
the remaining case $\C=\{\infty\}$ follows analogously. 
\end{proof}

Extending the idea of Example~\ref{pl-s2}, we now 
consider special measures that are important for our subsequent
treatment. 

\begin{proposition} \label{dm-s2}  
Let $\dom\in\bor{\R^n}$, let $\C$   
satisfy \reff{dm-e0a},
let $\E\in\cB(\dom)$, and let the measure 
$\lme\in\bawl{\dom}$ satisfy
\begin{equation} \label{dm-s2-1} 
  \lme\ge 0 \qmq{and}  \lme(\Cd\cap\E) > 0 \qmq{for all} \delta>0\,.
\end{equation}
If $\me\in\bawl{\dom}$ satisfies
\begin{equation} \label{dm-s2-2}
   \liminf_{\delta \downarrow 0} 
  \frac{\lme(A\cap\E\cap\C_\delta)}{\lme(\E\cap\C_\delta)}
  \le \me(A) \le 
  \limsup_{\delta\downarrow 0} 
  \frac{\lme(A\cap\E\cap\C_\delta)}{\lme(\E\cap\C_\delta)} \qmq{for all}
  A\in\cB(\dom)\,,
\end{equation}
then $\mu\in\Dens_\C$, $\E$ is an aura of $\me$, and one has
\begin{equation} \label{dm-s2-3}
   \liminf_{\delta \downarrow 0}
  \mI{\Cd\cap\E}{f}{\lme} \le \sI{\ccl\C}{f}{\me} 
\le \limsup_{\delta\downarrow 0} 
  \mI{\Cd\cap\E}{f}{\lme} 
 \qmq{for all} f\in\cL^\infty(\dom)\,.
\end{equation}

\end{proposition} 
\noi
A measure $\me\in\bawl{\dom}$ that meets \reff{dm-s2-2} 
for some $\lme\in\bawl{\dom}$ satisfying \reff{dm-s2-1} is called 
{\it $\lme$-density (on $\dom$) at $\C$ within $\E$}.
The set of all such measures is denoted by 
\begin{equation*}
  \lDens_\C^\E\,.
\end{equation*}
Notice that always
\begin{equation*}
   \lDens_\C^\E \subset \Dens_\C\,.
\end{equation*}
In particular we have that $\dens_x^\dom$ from Example~\ref{pl-s2} belongs to 
$\LDens_x^\dom$ while $\la_\C^\dom$ belongs to $\lDens_\C^\dom$ if
$\ti\la=\lem$. Subsequently, $\dens_\C^\E$ usually stands for a measure in 
$\LDens_\C^\E$. For $\E=\dom$ the previous result is already contained in
\cite{dens}. 

\begin{proof}
Obviously $\mu\ge 0$ and $\mu(\C_\de\cap\dom)=\mu(\E)=\mu(\dom)=1$. Thus 
$\mu\in\Dens_\C$ and $\E$ is an aura for $\mu$. Now, \reff{dm-s2-2} directly
implies \reff{dm-s2-3} for $f=\ch_A$ and, by multiplication with $-1$, also
for $f=-\ch_A$. Hence, by linearity, \reff{dm-s2-3} is valid for all 
step functions. Uniform approximation by step functions implies it for all 
$f\in\cL^\infty(\dom)$. 
\end{proof}

\noi
It turns out that already one inequality in 
\reff{dm-s2-3} is sufficient for the characterization of density measures.

\begin{theorem} \label{dm-s3}
Let $\dom\in\bor{\R^n}$, let $\C$ satisfy \reff{dm-e0a}, and let
$\me \in \bawl{\dom}$. Then measure $\mu$ is a density measure in
$\Dens_\C$ if and only if 
there is some $\lme\in\bawl{\dom}$ satisfying \reff{dm-s2-1} and
\begin{equation} \label{dm-s3-1}
\sI{\ccl\C}{f}{\me} \le \limsup_{\delta\downarrow 0} 
\mI{\Cd\cap\dom}{f}{\lme} \qmq{for all}
f\in\cL^\infty(\dom)\,.
\end{equation}
Moreover \reff{dm-s3-1} implies
\begin{equation} \label{dm-s3-2}
\liminf_{\delta\downarrow 0} \mI{\Cd\cap\dom}{f}{\lme}
\le \sI{\ccl\C}{f}{\me} \qmq{for all}
f\in\cL^\infty(\dom)\,.
\end{equation}
\end{theorem} 
\noi
The statement remains true if we interchange \reff{dm-s3-1} and
\reff{dm-s3-2}. If $\me\in\Dens_\C$ is given, then, by \reff{dm-e1}, we can
choose $\lme=\me$ for the implication in the proposition. Notice that 
either $\Dens_\C\ne\emptyset$ or \reff{dm-s2-1} implies \reff{dm-e00}
and, thus, $\C$ has to be a density set. For $\C\subset\cl\dom$ the result is
already contained in \cite{dens}. Let us provide a shortened version of that 
proof, that also covers our more general case.

\begin{proof}
If $\la\in\bawl{\dom}$ with \reff{dm-s2-1}
satisfies \reff{dm-s3-1} for some $f$, then also for $-f$. Using
\begin{equation*}
  \limsup_{\delta\downarrow 0} \mI{\Cd\cap\dom}{-f}{\lme} =
  - \liminf_{\delta\downarrow 0} \mI{\Cd\cap\dom}{f}{\lme}
\end{equation*}
we readily get \reff{dm-s3-2}. 

Now we assume that $\mu\in\Dens_\C$. Then \reff{dm-e1} implies
\reff{dm-s3-1} with $\la=\mu$ while \reff{dm-s2-1} follows from the
necessary condition \reff{dm-e00}. For the opposite direction 
we use that \reff{dm-s3-1} implies \reff{dm-s3-2}. With $f=\ch_A$ we
get \reff{dm-s2-2}, which gives $\mu\in\Dens_\C$ by Proposition~\ref{dm-s2}. 
\end{proof}

\noi
As $\Dens_C$, also the set $\lDens_\C^\E$
of $\la$-densities at $\C$ within $\E$ 
is convex and weak$^*$ compact. 

\begin{corollary} \label{dm-s4}
Let $\dom\in\bor{\R^n}$, let $\C$ satisfy \reff{dm-e0a},
let $\E\in\cB(\dom)$, and let the measure 
$\lme\in\bawl{\dom}$ satisfy \reff{dm-s2-1}.
Then $\lDens_\C^\E$
is a weak$^*$-compact and convex subset of $\Dens_C$. 
\end{corollary}

\begin{proof}
We have $\lDens_\C^\E\subset\Dens_\C$ by Proposition~\ref{dm-s2}.
With $\me_1,\me_2$ also a convex combination satisfies \reff{dm-s2-2},
which implies convexity of $\lDens_\C^\E$. Let us now consider any
$\me\not\in \lDens_\C^\E$. 
Then there is some $A\in\cB(\dom)$ such that at least one inequality in
\reff{dm-s2-2} is violated. Hence there is some $\eps>0$ such that it is 
also violated for all $\tilde\me\in\bawl{\dom}$ with 
$|\df{\tilde\me-\me}{\chi_A}|<\eps$. 
Therefore $\lDens_\C^\E$ is the complement of a
weak$^*$ open set. Thus it is a weak$^*$ closed subset of the weak$^*$ compact
set $\Dens_\C$, which gives the assertion.
\end{proof}

It turns out that any measure $\lme\in\bawl{\dom}$, which does not vanish
near $\C$, induces the existence of a density measure. 

\begin{proposition} \label{dm-s4a}     
Let $\dom\in\bor{\R^n}$, let $\C$ 
satisfy \reff{dm-e0a},
let $\E\in\cB(\dom)$, and let the measure 
$\lme\in\bawl{\dom}$ satisfy
\begin{equation} \label{dm-s4a-0} 
  \lme\ge 0 \qmq{and}  \lme(\Cd\cap\E) > 0 \qmq{for all} \delta>0\,.
\end{equation}
Then there exists a density measure $\me\in\Dens_\C$ on $\dom$
with $\cor{\me}\subset\ccl\C\cap\ccl\E$ and
\begin{equation}  \label{dm-s4a-1}
  \liminf_{\delta \downarrow 0}
  \mI{\Cd\cap\E}{f}{\lme} \le \sI{\ccl\C}{f}{\me} 
\le \limsup_{\delta\downarrow 0} 
  \mI{\Cd\cap\E}{f}{\lme} 
 \qmq{for all} f\in\cL^\infty(\dom)\,.
\end{equation}
Moreover  
$\E$ is an aura of $\me$ and $\me$ is a $\lme$-density (on $\dom$) 
at $\C$ within $\E$. 
\end{proposition}
\noi
The existence proof relies on the Hahn-Banach theorem where \reff{dm-s4a-1}
does not imply uniqueness of measure $\me$ in general. In \cite{dens}
the result is shown for the special case $\C\subset\cl\dom$ and $\E=\dom$.

\begin{proof}
We can argue analogously as in \cite{dens} (also for the case $\C=\{\infty\}$)
to get a measure $\mu\in\Dens_\C$ with $\cor{\mu}\subset\ccl\C$ and
\reff{dm-s4a-1}. This implies \reff{dm-s2-2} and, by Proposition~\ref{dm-s2},
we obtain that $\mu\in\lDens_\C^\E$ and that $\E$ is an aura of $\mu$. 
This readily implies $\cor{\mu}\subset\ccl\E$. 
\end{proof}

By Proposition~\ref{dm-s4a} we basically need a bounded measure $\la$
and a set $\E\in\cB(\dom)$ to ensure the existence of a density measure $\mu$
in $\Dens_\C$. Since the necessary condition \reff{dm-e00} is needed anyway, 
one can basically consider $\la=\lem$ for that task. 
The essential point is the boundedness of $\la$. But, if $\dom$ is bounded, 
then $\lme=\reme{\lem}{\dom}$ satisfies \reff{dm-s4a-0} with $\E=\dom$.
The next corollary provides a more general version in the case
$\C\subset\cl\dom$ and follows directly from Proposition~\ref{dm-s4a}.

\begin{corollary} \label{dm-s4b}
Let $\dom\in\bor{\R^n}$, let $\C$ satisfy \reff{dm-e0a},
let $\E\in\cB(\dom)$, and let
\begin{equation}\label{dm-s2c-0}
  \lem(\E\cap\C_\delta)>0\zmz{for all}\delta>0 \qmq{and}
  \lem(\E\cap\C_{\tilde\delta}) <\infty
  \qmq{for some} \tilde\delta>0\,.
\end{equation}
Then there exists an $\lem$-density $\dens_\C^\E\in\LDens_\C^\E$ 
at $\C$ within $\E$. Moreover $\dens_\C^\E$ has its core in 
$\ccl\C\cap\ccl\E$ and $\E$ is an aura of it.
\end{corollary}
\noi
Notice that, as in the previous proposition, the verified density 
$\dens_\C^\E$ is not unique.
For $\E=\dom$ and $\C=\{x\}$, the corollary verifies the existence 
of a density measure $\dens_x^\dom$ as considered in Example~\ref{pl-s2}. 

Let us discuss under which condition $\Dens_\C$ is not empty. 
If $\C\subset\cl\dom$ is a bounded density set, i.e. it satisfies
\reff{dm-e0a} and \reff{dm-e00}, then $\Dens_\C\ne\emptyset$ by
Corollary~\ref{dm-s4b} with $E=\dom$.  
If $\C\subset\cl\dom$ is an unbounded density set,
then the second condition in \reff{dm-s2c-0} may fail for $E=\dom$.
But if there is some bounded $E\in\cB(\dom)$ satisfying the first condition in 
\reff{dm-s2c-0}, then the second condition follows and Corollary~\ref{dm-s4b}
implies $\Dens_\C\ne\emptyset$ also in that case. 
The next result covers the remaining case where $\C$ is an unbounded density
set and there is no bounded $\E$ such that \reff{dm-s2c-0} is met
(which includes the case $\C=\{\infty\}$). 

\begin{proposition}\label{dm-s4c}
Let $\dom\in\bor{\R^n}$, let $\C$ 
satisfy \reff{dm-e0a},
let $\E\in\cB(\dom)$, and assume that we have for each 
$r>0$ that  
\begin{equation}\label{dm-s4c-1}
  \lem(\Cd\cap\E\cap B_r(\infty))>0 \qmq{for all} \delta>0\,.
\end{equation} 
Then there exists a density measure $\me\in\Dens_\C$ on $\dom$ such that 
$\cor\me=\{\infty\}$ and $\E$ is an aura of $\me$. 
\end{proposition}
\noi
Notice that \reff{dm-s4c-1} implies \reff{dm-e00}, i.e. $\C$ is a density set,
but we do not need some boundedness as in \reff{dm-s2c-0}.
Let us mention that $\me$ from the previous proposition is not an
$\lem$-density at $\C$.

\begin{proof}
The functional $p:X:=\cL^\infty(\dom)\to\R$ given by
\begin{equation*}
  p(f) := \limsup_{\delta,r\downarrow 0} \: 
  \essup{\Cd\cap\E\cap B_r(\infty)}{f}  
\end{equation*}
is positively 1-homogeneous and subadditive. Let 
$X_0\subset\cL^\infty(\dom)$ be the
linear subspace containing all functions $f$ of the form
\begin{equation*}
  f=c \text{ \:$\lem$-a.e. on } \Cd\cap\E\cap B_r(\infty) \text{ for some } 
  c\in\R\,,\;\delta>0\,,\;r>0 \, 
\end{equation*}
and we write $f(\infty):=c$ for $f\in X_0$. Then $f_0^*:X_0\to\R$ given by
\begin{equation*}
  f_0^*(f) = f(\infty)
\end{equation*}
is a linear functional on $X_0$ majorized by $p(f)$. 
The Hahn-Banach theorem provides a linear extension $f^*$ of $f_0^*$ onto $X$ 
such that
\begin{equation*}
  f^*(f) \le p(f) \le \|f\|_\infty \qmq{for all} f\in X\,.
\end{equation*}
Therefore $f^*\in\cL^\infty(\dom)^*$ and we identify $f^*$ with measure 
$\me\in\bawl{\dom}$ according to \reff{pl-dual}.
Obviously, $\|f^*\|=\|\me\|\le 1$.
For $f\in\cL^\infty(\dom)$
\begin{equation*}
  \I{\dom}{-f}{\me} \le p(-f) = 
  \limsup_{\delta,r \downarrow 0}\: \essup{\Cd\cap\E\cap B_r(\infty)}{-f}  
\end{equation*}
and, thus,
\begin{equation*}
  \liminf_{\delta,r \downarrow 0}\: \essinf{\Cd\cap\E\cap B_r(\infty)}{f}
  \le  \I{\dom}{f}{\me} \,.
\end{equation*}
Using $f=\chi_\dom$ we get $\|\me\|=1$. We have $\me\ge 0$, since
\begin{equation*}
  0 \le 
  \liminf_{\delta,r \downarrow 0}\: \essinf{\Cd\cap\E\cap B_r(\infty)}{\chi_B}
  \le \I{\dom}{\chi_B}{\me} = \me(B) \qmq{for all} B\in\cB(\dom)\,.
\end{equation*}
Moreover,
\begin{equation*}
  1 = \liminf_{\delta,r \downarrow 0}\: 
  \essinf{\Cd\cap\E\cap B_r(\infty)}{\chi_\dom} \le \me(\dom) \le
  \limsup_{\delta,r \downarrow 0}\:\essup{\Cd\cap\E\cap B_r(\infty)}{\chi_\dom}  
  = 1
\end{equation*}
and, for each $\tilde\delta>0$, 
\begin{equation*}
 1 = \liminf_{\delta,r \downarrow 0}\: 
  \essinf{\Cd\cap\E\cap B_r(\infty)}{\chi_{\C_{\tilde\delta}\cap\E}} 
  \le \me(\C_{\tilde\delta}\cap\E) \le
  \limsup_{\delta,r \downarrow 0}\: 
  \essup{\Cd\cap\E\cap B_r(\infty)}{\chi_{\C_{\ti\de}\cap\E}}  
  = 1\,.
\end{equation*}
By $\E\subset\dom$ we get $\mu\in\Dens_\C$ and $\E$ is aura of $\me$.
For any bounded $B\in\cB(\dom)$ we readily obtain
\begin{equation*}
 0 = \liminf_{\delta,r \downarrow 0}\: 
  \essinf{\Cd\cap\E\cap B_r(\infty)}{\chi_B} 
  \le \me(B) \le
  \limsup_{\delta,r \downarrow 0}\: 
  \essup{\Cd\cap\E\cap B_r(\infty)} = 0\,.
\end{equation*}
Therefore $\cor\me=\{\infty\}$, since $\cor\me\ne\emptyset$ by \reff{dm-s1-1}.
\end{proof}

From Corollary~\ref{dm-s4b} and Proposition~\ref{dm-s4c}
we obtain that the necessary condition \reff{dm-e00} is 
also sufficient for the existence of density measures. 

\begin{theorem} \label{dm-s4d}
Let $\dom\in\bor{\R^n}$ and let $\C$ 
satisfy \reff{dm-e0a}. Then
\begin{equation*}
  \Dens_\C\ne\emptyset \qmq{if and only if} 
  \text{\reff{dm-e00} is satisfied}.
\end{equation*}
Moreover, if there is some $\E\in\cB(\dom)$ with
\begin{equation}\label{dm-s4d-1}
  \lem(\C_\delta\cap\E)>0 \qmq{for all} \delta>0
\end{equation}
(i.e. \reff{dm-e00} is satisfied with $\E$ instead of $\dom$), then
there is some $\me\in\Dens_\C$ such that $\E$ is an aura of $\me$.
\end{theorem}
\noi
Notice that $\Dens_\C\ne\emptyset$ if $\C$ is a density set.
Since the special case $\C=\{x\}$ for some $x\in\ccl\dom$ plays an important 
role in our subsequent treatment, we use the notation 
\begin{equation} \label{cdom}
  \ccdom := \{y\in\ccl\dom\mid \lem(\dom\cap B_\delta(y))>0
              \zmz{for all} \delta>0\}\,, \qquad
  \cdom := \ccdom\cap\cl\dom  \,.
\end{equation}
By the previous result we then have that
\begin{equation*} \label{cdom1a}
  \Dens_x\ne\emptyset \qmq{if and only if}
  x\in \ccdom \,.
\end{equation*}
Notice that 
\begin{equation}\label{cdom1}
  \lem(\dom\setminus\cdom)=0
\end{equation}
and, if $\lem(\pa\dom)>0$, then 
$\lem(\cl\dom\setminus\cdom)>0$ is possible
(cf. \cite[p. 45]{evans}). 

\begin{proof}
First let $\Dens_\C\ne\emptyset$. Then the arguments around \reff{dm-e00}
imply that \reff{dm-e00} has to be fulfilled. For the reverse statement 
we apply Corollary~\ref{dm-s4b} if 
$\lem(\dom\cap\C_{\tilde\delta}) <\infty$ for some $\tilde\delta>0$.
Otherwise \reff{dm-e00} implies \reff{dm-s4c-1} for $\E=\dom$ and all $r>0$
and we can apply Proposition~\ref{dm-s4c}. 
For the proof of the last statement we assume that \reff{dm-s4d-1} is satisfied.
Then the previous result with $\E$ instead of $\dom$ gives some density 
$\mu\in\Dens_\C$ on $\E$ with aura $\E$. Extending $\mu$ with zero on $\dom$
we get the assertion.
\end{proof}

The next example indicates how Corollary~\ref{dm-s4b} provides a large
variety of density measures in $\Dens_\C$ by suitable selections of
$E\subset\dom$.

\begin{example} \label{dm-s4e}
Let $\dom\in\bor{\R^n}$ and let $\C\subset \cl{\dom}$ satisfy \reff{dm-e0a}.
We choose $x\in C$ and $E\subset\dom$ such that
\begin{equation} \label{dm-s4e-1}
  \ol E\cap C = \{x\} \qmq{and}
  \lem(B_\delta(x) \cap E) > 0 \zmz{for all} \delta>0\,.
\end{equation}
Figure~\ref{fig:cusp_dens} shows an example where $\C\subset\bd\dom$ and
$E$ has a cusp at $x$.
\begin{figure}[H]
	\centering
	\begin{tikzpicture}[scale=1.2]
		\draw (2,0) to[out=180,in=0] (-0.5,1);
		\draw (-0.5,1) to[out=180,in=90] (-1.5,0);
		\draw (-1.5,0) to[out=-90,in=180] (-0.5,-1);
		\draw (-0.5,-1) to[out=0,in=180] (2,0);
		\draw (-0.2,0) node {$E$};
		\draw[line width=1pt] (2,1.4) to (2.0,-1.4);
                \draw (2.2,1.25) node {$C$};
                \draw (1.1,1.1) node {$\dom$};
                \draw[line width=2pt] (2,0) circle [radius = 0.03];  
		\draw (2.2,0) node {$x$};
		\draw[dashed] (2,0) circle [radius = 1.0];
		\draw[dashed] (2,0) circle [radius = 0.7];
		\draw[dashed] (2,0) circle [radius = 0.4];
\end{tikzpicture}
\caption{The set $\dom$ is bounded on the right by $C$, the boundary $\bd\dom$
contains the point $x$, and $\dom$ contains the cuspidate set
$E$.}\label{fig:cusp_dens} 
\end{figure}
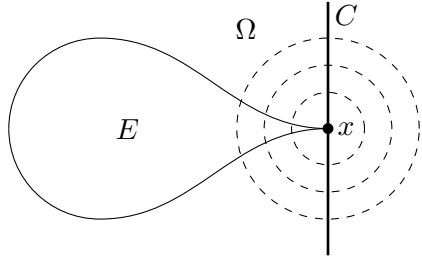
\noi
Corollary~\ref{dm-s4b} ensures the existence of an $\lem$-density   
$\me\in\LDens_x^\E\subs\Dens_\C$.
Provided \reff{dm-s4e-1} remains satisfied, 
we can continuously change the direction of the cusp of $E$ at $x$ to get a
continuum of different density measures in $\Dens_\C$ with core $\{x\}$. 
Alternatively we can move $x$ along $\C$ to get again a continuum of density
measures 
in $\Dens_\C$. This gives a first idea about the huge variety of densities at
a set $\C$.  
\end{example}

The next result about weak$^*$ accumulation points of density measures 
turns out to be a useful tool for our applications.

\begin{proposition}  \label{dm-s4f}  
Let $\dom\in\bor{\R^n}$.
\bgl
\item
Let $\C\subs\ol\dom$ be closed, then 
\begin{equation*}
  M_C := \{ \mu \in\bawl{\dom} \mid \cor{\mu}\subs\C\,,\;
  \tx{$\mu$ satisfies \reff{dm-e0}} \} 
\end{equation*}
is a weak$^*$ compact convex subset of $\bawl{\dom}$. 

\item
Let $\C\subset\cl\dom$ satisfy \reff{dm-e0a}
and let $\mu_k\in\bawl{\dom}$ be density measures with
\begin{equation*}
  \cor{\mu_k} \subs \C_{\de_k}=:\C_k \qmq{for some}  \delta_k\downarrow 0\,.
\end{equation*}
Then $\{\mu_k\}$ has a
weak$^*$ accumulation point $\mu\in\Dens_\C$ and for any 
$f\in\cL^\infty(\dom,\R^m)$
there is a subsequence, depending on $f$, such that 
\begin{equation} \label{dm-s4f-1}  
  \sI{\C}{f}{\mu} = \lim_{k'\to\infty} \sI{\C_{k'}}{f}{\mu_{k'}}\,.
\end{equation}
\el
\end{proposition}

\begin{proof}
For (1) we can argue as in the proof of Theorem~\ref{dm-s1} (2) 
in \cite{dens}. For (2) we notice that 
$\cor{\mu_k}\in\ol{\C_1}$ for all $k$. Then, by (1), 
$\{\mu_k\}$ has a weak$^*$ accumulation point $\mu$ satisfying \reff{dm-e0}
with $\ol{C_1}$ and having core in $\cl{\C_1}$. 
Now we fix $f=(f_1,\dots,f_m)$ in $\cL^\infty(\dom,\R^m)$.
Since $\mu$ is weak$^*$ accumulation point, 
each weak$^*$ open neighborhood of $\mu$ of the form
\begin{equation*}
  \big\{ \mu'\in\bawl{\dom}\:\big|\: 
          \max_{j=1,\dots,m}\big|\df{\mu-\mu'}{f_j}\big| < \ep \big\}
  \qmq{with} \ep>0
\end{equation*}
contains infinitely many $\mu_k$. Thus there is a subsequence $\{\mu_{k'}\}$ 
with \reff{dm-s4f-1} if $\cor\mu$ belongs to $\C$. Now, for fixed $k$ we have 
$\cor{\mu_l}\subs\C_k$ for all $l\ge k$ by assumption. 
Hence, as above, for any $A\subs\dom\setminus(\C_k)_\de$ with $\de>0$ 
there is some subsequence $\{\mu_{l'}\}$ such that
\begin{equation*}
  0\overset{l'>k}{=}\mu_{l'}(A) =\sI{\C_{l'}}{\ch_A}{\mu_{l'}} \to 
  \sI{\cl{\C_1}}{\ch_A}{\mu} = \me(A)=0 \,.
\end{equation*}
Then $\cor\mu$ must belong to $\C_k$ for all $k$. 
Therefore $\cor\mu\subs\C$ and $\mu$ satisfies \reff{dm-e0} with~$\C$.   
But this implies $\mu\in\Dens_\C$. 
\end{proof}

\subsection{Essential and approximate limit}
\label{dm-lim}

We are now interested in further estimates for integrals with density measures. 
For measures $\mu\in\Dens_\C$ and $f\in\cL^\infty(\dom)$ we have the standard
estimate 
\begin{equation*}
  \I{\C_\de}{f}{\mu} \le |\me|(C_\de) \: \essup{\C_\de}{f} =
  \essup{\C_\de}{f} \qmq{for all} \de>0\,.
\end{equation*}
Let us look for a refined version independent of $\de$. Moreover, 
we want to cover $\mu$-integrable functions 
beyond $\cL^\infty(\dom)$.
As an appropriate tool for some $\de$-independent sharper estimate we 
consider the {\it essential supremum} and the 
{\it essential infimum} of  
$f\in\cL^1_{\rm loc}(\dom)$ {\it near} $\C$ given by
\begin{eqnarray*}
  \sC f 
&:=&
  \lim_{\delta\downarrow 0}\, \essup{\Cd\cap\dom}{f} =
  \inf_{\delta>0}\, \essup{\Cd\cap\dom}{f}  \qmq{and}  \\
  \iC f 
&:=&
  \lim_{\delta\downarrow 0}\, \essinf{\Cd\cap\dom}{f} =
  \sup_{\delta>0}\, \essinf{\Cd\cap\dom}{f}   \,
\end{eqnarray*}
where we allow the values $\pm\infty$.
Since $\tx{ess\,sup}$ and $\tx{ess\,inf}$ on the right hand side are monotone
in $\de$, the values on the left are well defined. Clearly,
\begin{equation} \label{ess-bd-c}
  \zmz{if} \sC f\,,\;\iC f\ne \pm\infty \qmq{then}
  f\in\cL^\infty(\C_\de\cap\dom) \zmz{for some} \de>0\,.
\end{equation}
For $f\in\cL^1_{\rm loc}(\dom,\R^m)$ and $x\in\ccdom$ we call
$\alpha\in\R^m$ {\it essential limit} of $f$ at $x$ if
\begin{equation}\label{ess-con}
   \sx |f-\alpha| = 0\,
\end{equation}
and we write
\begin{equation}\label{ess-lim}
  \esslim_{y\to x}f(y) := \alpha \,.
\end{equation}
Analogously to \reff{ess-bd-c},
\begin{equation} \label{ess-bd}
  \esslim_{y\to x}f(y) = \alpha   \qmq{implies}
  f\in\cL^\infty(B_\de(x)\cap\dom)  \zmz{for some} \de>0\,. 
\end{equation}
In the scalar case $m=1$ with $x\in\ccdom$, the values 
$\ix f$ and $\sx f$ can be considered as essential limit inferior and 
superior, respectively, of $f$ for $y\to x$ and, clearly, 
\begin{equation*}
   \ix f = \sx f = \alpha \qmq{if and only if} \esslim_{y\to x}f(y) := \alpha\,.
\end{equation*}
We are now able to provide refined estimates for density integrals. 

\begin{proposition} \label{dm-s4g} 
Let $\dom\in \bor{\R^n}$, let $\C$ be a density set,
let $\me\in\Dens_\C$, and let the function 
$f\in\cL^1_{\rm loc}(\dom)$ be $\mu$-integrable.
Then 
\begin{equation} \label{dm-s4g-1}
\iC f \, \le \, \sI{\ccl\C}{f}{\me} \, \le \, \sC f \,.
\end{equation}
and
\begin{equation}  \label{dm-s4g-3}
  \iC f \le \liminf_{\de\dto 0} \mI{\C_\de\cap\dom}{f}{\lem} \le
  \limsup_{\de\dto 0} \mI{\C_\de\cap\dom}{f}{\lem} \le \sC f\,.
\end{equation}
Moreover, the mapping $\omega:\cL^\infty(\dom)\to\R$ given by
\begin{equation} \label{dm-s4g-2}
  \omega(f) := \sC f
\end{equation}
is positively 1-homogeneous, subadditive, convex, and continuous. 
\end{proposition}
\noi
By \reff{ess-bd-c}, the integrability assumption for $f$ can be dropped if 
both the most left and most right value in \reff{dm-s4g-1} are finite.  
\reff{dm-s4g-1} is already stated in \cite{dens} for $\C\subset\cl\dom$
and $\f\in\cL^\infty(\dom)$. 

\begin{proof}   
We obviously have that 
\begin{equation*}
  \sI{\ccl\C}{f}{\me} = \lim_{\delta\downarrow 0}\I{\Cd\cap\dom}{f}{\me} \le
  \limsup_{\delta\downarrow 0}
  \I{\Cd\cap\dom}{\essup{\Cd\cap\dom} f}{\me} 
  = \lim_{\delta\downarrow 0}\, \essup{\Cd\cap\dom} f \,
\end{equation*}
and an analogous inequality with $\tx{ess\,inf}$, which implies 
\reff{dm-s4g-1}. We argue the same way with $\lem$ instead of $\mu$ to get 
\reff{dm-s4g-3}. For the remaining statement we have that 
$\omega$ is certainly positively 1-homogeneous and subadditive and, thus, 
convex on $\cL^\infty(\dom)$. In addition 
$|\omega(f)|\le\|f\|_{\cL^\infty(\dom)}$.
Hence it is bounded on bounded sets and, as convex function, continuous
(cf. \cite[p. 34]{clarke-fa}).
\end{proof}

The inequalities in \reff{dm-s4g-1} are sharp. But, before
stating that precisely, we introduce an important notion from convex analysis. 
For a normed space $X$ the {\it support function} of a set $M\subset X$ is
given by
\begin{equation} \label{dm-supp}
  \omega_M(f^*) := \sup_{f\in M} \df{f^*}{f}  \qmq{for all} f^*\in X^*\,
\end{equation}
and the support function of a set $M^*\subset X^*$ is
\begin{equation} \label{dm-supp*}
  \omega_{M^*}(f):= \sup_{f^*\in M^*} \df{f^*}{f}  \qmq{for all} f\in X\,.
\end{equation}
Notice that closed convex sets $M\subset X$ and weak$^*$ closed convex sets
$M^*\subset X^*$ are uniquely determined by their support functions
(cf. \cite[p. 28]{clarke}).

\begin{proposition} \label{dm-s5}        \label{dm-s5}
Let $\dom\in\bor{\R^n}$, let $\C$ be a density set, and assume that 
$f\in\cL^1_{\rm loc}(\dom)$ is $\mu$-integrable for at
least one $\me\in\Dens_\C$. Then
\begin{equation}\label{dm-s5-1}
\sup_{\substack{\me\in\Dens_\C\tx{\rm s.t.}\\f\:\tx{\rm is $\mu$-integrable}}} 
\sI{\ccl\C}{f}{\me} 
= \sC f 
\end{equation}
and
\begin{equation*}
\inf_{\substack{\me\in\Dens_\C\tx{\rm s.t.}\\f\:\tx{\rm is $\mu$-integrable}}}
\sI{\ccl\C}{f}{\me} 
= \iC f \,.
\end{equation*}
Moreover, $\omega$ from \reff{dm-s4g-2} is the support function of 
$\Dens_\C$ (as set in $L^\infty(\dom)^*$).
\end{proposition}
\noi
Notice that $f$ is $\mu$-integrable for all $\mu\in\Dens_\C$ if 
$f\in\cL^\infty(\dom)$. 
In \cite{dens} the result is shown for bounded density sets $\C$ and
$f\in\cL^\infty(\dom)$. 

\begin{proof}   
First we assume that 
$\sC f = \lim \limits_{\delta \downarrow 0}\essup{\C_\de\cap\dom}{f}<\infty$
and for $\eps>0$ we set
\begin{equation*}
M_\eps := \big\{x\in\dom \:\big|\: f(x) \ge
\lim \limits_{\delta \downarrow 0}\essup{\C_\de\cap\dom}{f} - \eps \big\}\,.
\end{equation*}
Then $\lem(\C_{\delta}\cap M_\eps) > 0$ for all $\delta>0$
and, thus, $\C$ is also density set of $M_\ep$. By Theorem~\ref{dm-s4d}
there is some density measure $\me_\eps\in\Dens_\C$ on $M_\ep$, which we 
identify with its zero extension on $\dom$. Hence $M_\ep$ is aura of $\mu_\ep$
and 
\begin{equation*}
\sI{\ccl\C}{f}{\me_\eps} = 
\lim_{\delta\downarrow 0} \I{\C_\delta\cap M_\ep}{f}{\me_\eps} 
\ge \lim_{\delta\downarrow 0}
\essup{\dnhd{\C}{\delta} \cap \dom}{\fun} - \eps \,.
\end{equation*}
Consequently,
\begin{equation*}
\sup_{\substack{\me \in \Dens_\C\\f\:\mu\tx{\rm -integrable}}}
\sI{\ccl\C}{f}{\me} \ge
\sup_{\eps > 0} \sI{\ccl\C}{f}{\me_\eps} \ge
\lim \limits_{\delta\downarrow 0}
\essup{\dnhd{\C}{\delta} \cap \dom}{\fun} \,.
\end{equation*}
With \reff{dm-s4g-1} we obtain \reff{dm-s5-1}. 
Let now $\sC f = \infty$. Then, for each $k\in\N$ there is some $k'>k$ such 
that $\lem(\C_{\delta}\cap M_k) > 0$ for all $\delta>0$ with
\begin{equation*}
  M_k := \big\{x\in\dom \:\big|\: k < f(x) < k' \big\}\,.
\end{equation*}
As above we get a measure $\mu_k\in\Dens_\C$ with aura $M_k$. Since $f$ is
essentially bounded on $M_k$, it is $\mu_k$-integrable with 
\begin{equation*}
\sup_{\substack{\me \in \Dens_\C\\f\:\mu\tx{\rm -integrable}}} \sI{\ccl\C}{f}{\me} 
\ge \sI{\ccl\C}{f}{\me_k} \ge k \qmq{for all} k\in\N\,.
\end{equation*}
Thus we obtain \reff{dm-s5-1} also in this case. The other equality 
follows analogously.

For the remaining statement we now consider $\Dens_\C$ as set in $X^*$ for
$X=\cL^\infty(\dom)$.  
We identify each $\me\in\Dens_\C$ with some $f^*\in X^*$ according to
\reff{pl-dual}. Then
\begin{equation*}
  \df{f^*}{f} = \I{\dom}{f}{\me} = \sI{\ccl\C}{f}{\me}
  \qmq{for all} f\in \cL^\infty(\dom)\,.
\end{equation*}
The definition of the support function and \reff{dm-s5-1} directly implies 
that $\omega$ from \reff{dm-s4g-2} is the support function of $\Dens_\C$
(use that all $f\in\cL^\infty(\dom)$ are integrable).
\end{proof}

Since $\omega$ from \reff{dm-s4g-2} is the support function of 
$\Dens_\C$, we can supplement
Proposition~\ref{dm-s4g} by the next corollary. 

\begin{corollary} \label{dm-s6}
Let $\dom\in \bor{\R^n}$ and let $\C$ 
satisfy \reff{dm-e0a}.
If $\me\in\bawl{\dom}$ satisfies
\begin{equation} \label{dm-s6-1}
  \sI{\ccl\C}{f}{\me} \le \sC f \z \big(= \omega(f)\big) 
  \qmq{for all} f\in\cL^\infty(\dom)\,,
\end{equation}
then $\me\in\Dens_\C$.
\end{corollary}

\begin{proof}  
Since $\omega$ is the support function of the convex weak* closed
set $\Dens_\C$ and since a convex weak* closed set is uniquely determined by
its support function, all $\me$ satisfying inequality \reff{dm-s6-1}
belong to $\Dens_\C$ (cf. \cite[p. 29]{clarke}).
\end{proof}

We now consider the action of $\Dens_\C$ on vector functions 
$f\in\cL^\infty(\dom,\R^m)$
where 
\begin{equation*}
  \df{\me}{f} := \Big(\I{\dom}{f_1}{\me},\dots,\I{\dom}{f_m}{\me}\Big)\in\R^m
  \qmq{for} f=(f_1,\dots,f_m)
\end{equation*}
and 
\begin{equation*}
  \df{\Dens_\C}{f} := \big\{ \df{\me}{f} \:\big|\: \me\in\Dens_\C \big\}
  \subset \R^m\,.
\end{equation*}
Recall that $\Dens_\C$ is convex and weak$^*$ compact.

\begin{proposition} \label{dm-s7}
Let $\dom\in \bor{\R^n}$, let $\C$ be a density set, and let
$f\in\cL^\infty(\dom,\R^m)$. Then $\df{\Dens_\C}{f}\subset\R^m$ is compact and
convex and its support function $\omega$ is given by
\begin{equation*}
  \omega_f(v) :=  \sC\, (f\cdot v) 
  \qmq{for all} v\in\R^m \,.
\end{equation*}
For $f\in\cL^\infty(\dom)$, i.e. $m=1$, we have that
\begin{equation} \label{dm-s7-2}
\df{\Dens_\C}{f} = 
\big[\, \iC f, \sC f \,\big] \,.
\end{equation}
\end{proposition}
\noi
\reff{dm-s7-2} is already stated in \cite{dens}.

\begin{proof}   
We consider the mapping $\df{\cdot}{f}:\bawl{\dom}\to\R^m$. Since it is
linear, it maps the convex set $\Dens_\C$ onto a convex set. 
For fixed $\me\in\bawl{\dom}$ and $\eps>0$, the set
\begin{equation*}
  \big\{ \lme\in\bawl{\dom}\:\big|\: |\df{\lme-\me}{f}|<\eps \big\}
\end{equation*}
is weak* open and it is the preimage of $B_\eps\big(\langle\me,f\rangle\big)$. 
Hence $\df{\cdot}{f}$ is weak* continuous and maps the weak* compact set 
$\Dens_\C$ onto a compact set in $\R^m$. 
By the definition of the support function of $\df{\Dens_\C}{f}$ we have
\begin{equation*}
  \omega_f(v) = \sup_{w\in\df{\Dens_\C}{f}} w\cdot v 
  = \sup_{\me\in\Dens_\C}\sI{\ccl\C}{f\cdot v}{\me} 
  \overset{\reff{dm-s5-1}}{=} \sC (f\cdot v) \, 
\end{equation*}
for all $v\in\R^m$. 
For $f\in\cL^\infty(\dom)$, Proposition~\ref{dm-s5} and the convexity of 
$\df{\Dens_\C}{f}$ imply 
\begin{equation*}
  \big(\, \iC f, \sC f \,\big) \subset \df{\Dens_\C}{f} \subset
  \big[\, \iC f, \sC f \,\big] \,.
\end{equation*}
Since $\df{\Dens_\C}{f}$ is compact, the last statement follows.
\end{proof}

Next we provide conditions for the existence of the essential limit
(cf. \reff{ess-con}).

\begin{proposition} \label{dm-s11}
Let $\dom\in\cB(\R^n)$, let $x\in\cdom$,
let $f\in\cL^1_{\rm loc}(\dom)$, and let $\alpha\in\R$. 
Then the following assertions are
equivalent: 
\bgl
\item
We have that
\begin{equation} \label{dm-s11-1}
  \esslim_{y\to x} f=\alpha \qquad (\tx{i.e. } 
  \ix f = \sx f = \alpha) \,.
\end{equation}
\item
For all $\E\in\cB(\dom)$ satisfying
\begin{equation} \label{dm-s11-3}
  \lem(B_\delta(x)\cap\E)>0 \qmq{for all} \delta>0\,,
\end{equation}
there is an $\lem$-density $\dens_x^\E\in\LDens_x^\E$ 
at $x$ within $\E$ such that
\begin{equation} \label{dm-s11-2}
  \sI{x}{f}{\dens_x^\E} = \alpha \,.
\end{equation}
\item
We have that $f$ is $\dens_x$-integrable with 
\begin{equation*}
  \sI{x}{f}{\dens_x} = \alpha \qmq{for all} \dens_x\in\Dens_x\,.
\end{equation*}
\el
\end{proposition}
\noi
Notice that \reff{dm-s11-3} ensures that $\LDens_x^\E\ne\emptyset$
by Corollary~\ref{dm-s4b}. Let us mention that \reff{dm-s11-1} 
implies
\begin{equation}\label{dm-s11-5}
  \lim_{\de\dto 0}\mI{B_\de(x)\cap\dom}{f}{\lem} = \al
\end{equation}
by \reff{dm-s4g-3}, but the opposite is not true in general (consider e.g.
the sign function $f(y)=\pm 1$ for $\pm y>0$ on $\dom=\R$ and take $x=0$).

\begin{proof}
If (1) is satisfied, then $f$ is $\dens_x$-integrable for any
$\dens_x\in\Dens_x$ by \reff{ess-bd} and the integral equals $\al$ by
\reff{dm-s4g-1}, which gives (3). If (3) is met then $f$ has to
satisfy \reff{ess-bd}, since otherwise we can argue as in the proof
of Proposition~\ref{dm-s5} to get some $\mu$ in $\Dens_x$ with 
$\sI{x}{f}{\mu}>\al$. Hence $f$ is $\mu$-integrable for all $\mu\in\Dens_x$
and Proposition~\ref{dm-s5} implies (1). 
Since (3) trivially implies
(2), it remains to show that (2) implies (1). 

Let us assume that (2) is satisfied and (1) is not, i.e.
\begin{equation*}
  \alpha^- := \ix f < \sx f =: \alpha^+ \qmq{for} \alpha^\pm\in\cl\R\,. 
\end{equation*}
We choose $\be^\pm\in\R$ such that
\begin{equation*}
  \alpha^- < \be^- < \be^+ < \alpha^+\,.
\end{equation*}
Then 
\begin{equation*}
  \lem(E^+\cap B_\delta(x))>0 \qmq{for all} \delta>0 \zmz{with}
  E^+:=\{f>\be^+\} \,.
\end{equation*}
By Corollary~\ref{dm-s4b} we have $\LDens_x^{E^+}\!\ne\emptyset$ and,
for any $\dens_x^{E^+}\!\in\LDens_x^{E^+}\!$,
\begin{equation*}
  \sI{x}{f}{\dens_x^{E^+}} = \I{E^+}{f}{\dens_x^{E^+}}
  \ge \I{E^+}{\be^+}{\dens_x^{E^+}} =  \be^+ \,.
\end{equation*}
Analogously, $\LDens_x^{E^-}\!\ne\emptyset$ and, for all
$\dens_x^{E^-}\!\in\LDens_x^{E^-}\!$,
\begin{equation*}
  \sI{x}{f}{\dens_x^{E^-}} \le \sI{E^-}{\be^-}{\dens_x^{E^-}} =  \be^- 
  \qmq{with} E^-:=\{f<\be^-\}\,.
\end{equation*}
But this contradicts our assumption for $E=E^\pm$ in (2) 
and gives the assertion. 
\end{proof}

We have that $\sx f$ and $\ix f$ are defined for $x\in\cdom$ and all
$f\in\cL^1_{\rm loc}(\dom)$. 
Notice that $x\in\cdom$ for $\lem$-a.e. $x\in\dom$ (cf. \reff{cdom1}).
A further quantity describing the local behavior of
a function $f$ near $x$ is the approximate limit. 
For $f\in \cL^1_{\rm loc}(\dom,\R^m)$
we call $\alpha\in\R^m$ {\it approximate limit} of $f$ at $x\in\cdom$ 
(with respect to~$\dom$) if
\begin{equation*} \label{dm-app}
  \lim_{\delta\downarrow 0} 
  \frac{\lem\big(B_\delta^\dom(x)\cap\{|f-\alpha|\ge\eps\}\big)}
       {\lem(B_\delta^\dom(x))}
  = 0 \qmq{for all} \eps >0\,
\end{equation*} 
(where $B_\delta^\dom(x):=B_\delta(x)\cap\dom$) and we write 
\begin{equation} \label{dm-app1}
  \alim_{y\to x} f(y)=\alpha \,.
\end{equation}
Notice that \reff{dm-app1} is equivalent to
\begin{equation} \label{dm-app2}
  \alim_{y\to x} |f(y)-\alpha| = 0 \,.
\end{equation}
The approximate limit of $f$ exists for $\lem$-a.e. $x\in\dom$, since for 
any $\eps>0$ one has 
\begin{equation*}
  \frac{\lem\big(B_\delta^\dom(x)\cap\{|f-\alpha|\ge\eps\}\big)}
       {\lem(B_\delta^\dom(x))} \le \frac{1}{\eps}\,
  \mI{B_\delta(x)}{|f-\alpha|}{(\reme{\lem}{\dom})} 
\end{equation*}
and for $(\reme{\lem}{\dom})$-a.e. $x\in\dom$ there is some $\alpha\in\R^m$ 
such that the right hand side tends to zero as $\delta\downarrow 0$ 
(cf. \cite[p. 44]{evans}).

For $m=1$ we define the  
{\it upper and lower  approximate limit} of $f$ at $x\in\cdom$
(with respect to $\dom$) as
\begin{eqnarray}
  \alimsup_{y\to x} f(y) 
&:=&  \label{dm-aps}
  \inf\bigg\{ \alpha\in\R\:\Big|\: \lim_{\delta\downarrow 0}
  \frac{\lem\big(B_\delta^\dom(x)\cap\{f>\alpha\}\big)}
       {\lem(B_\delta^\dom(x))} = 0 \,\bigg\} \qmq{and} \\
  \aliminf_{y\to x} f(y) 
&:=&  \label{dm-api}
  \sup\bigg\{ \alpha\in\R\:\Big|\: \lim_{\delta\downarrow 0}
  \frac{\lem\big(B_\delta^\dom(x)\cap\{f<\alpha\}\big)}
       {\lem(B_\delta^\dom(x))} = 0 \,\bigg\} \,,
\end{eqnarray}
respectively, where
\begin{equation*}
  -\infty \le \aliminf_{y\to x} f(y) \le \alimsup_{y\to x} f(y) \le \infty\,. 
\end{equation*}
In the scalar case $m=1$
we also extend the definition of the approximate limit by
\begin{equation*}
  \alim_{y\to x} f(y)=\Big\{
  \begin{array}{rl} \infty & \text{ if } \aliminf_{y\to x} f(y) = \infty \,, \\
                    -\infty  & \text{ if } \alimsup_{y\to x} f(y) = -\infty \,.
  \end{array}
\end{equation*}
Certainly we then have that $\alim_{y\to x} f(y)=\alpha$ 
if and only if
\begin{equation}\label{dm-s9-3}
  \alimsup_{y\to x} f(y) = \aliminf_{y\to x} f(y) = \alpha\,. 
\end{equation}
Let us provide some estimate for the upper and lower approximate limit.

\begin{proposition}\label{dm-s9}
Let $\dom\in\bor{\R^n}$, let $f\in \cL^1_{\rm loc}(\dom)$, and let
$x\in\cdom$. Then
\begin{equation}\label{dm-s9-1}
  \teinf{x}\, f \le \aliminf_{y\to x} f(y) \le
  \alimsup_{y\to x} f(y) \le \tesup{x}\, f \,.
\end{equation}
Moreover, 
\begin{equation*}
  \esslim_{y\to x} f(y)=\al \qmq{implies} \alim_{y\to x}f(y)=\al \,. 
\end{equation*}
\end{proposition}

\begin{proof}
For $\rho>0$ we set $\alpha_\rho:= \essup{B_\rho^\dom(x)}{f}$. Then
\begin{equation*}
  \alpha:= \sx f \le\alpha_\rho  \qmq{and}
  \lem\big(B_\delta^\dom(x)\cap\{f>\alpha_\rho\}\big) = 0
  \qmq{for all} 0<\delta<\rho \,.
\end{equation*}
Consequently, 
\begin{equation*}
  \alimsup_{y\to x} f \le \inf_{\rho>0} \alpha_\rho = \alpha
\end{equation*}
which gives the most right inequality. The most left inequality 
follows analogously. The remaining inequality is a consequence of the
definition of the approximate limit. 

If $\esslim_{y\to x} f(y)=\al$, then by \reff{ess-con}, \reff{dm-s9-3} 
and \reff{dm-s9-1},
\begin{equation*}
    0\le \alim_{y\to x} |f(y)-\al|\le \sx |f-\alpha| = 0 \,
\end{equation*}
and the assertion follows from \reff{dm-app2}.
\end{proof}

The next example shows that the inequalities in \reff{dm-s9-1} can be strict. 
Therefore the existence of the essential limit is a stronger condition 
than the existence of the approximate limit. 

\begin{example}\label{dm-s9a}
Let $\dom=B_1(0)\subset\R^2$ and $x=0$. We consider
\begin{equation*}
  \dom_{1\pm}=\{(x_1,x_2)\in\dom \mid \pm x_2>\sqrt {|x_1|}\}\,, \quad
  \dom_{2\pm} = \{(x_1,x_2)\in\dom \mid \pm x_1> x_2^2\}
\end{equation*}
and $f\in\cL^\infty(\dom)$ with
\begin{equation*}
  f(x_1,x_2) = \bigg\{ 
  \begin{array}{ll} 
     \pm 2 & \text{on } \dom_{1\pm} \,,   \\
     \pm 1 & \text{on } \dom_{2\pm} \,. 
  \end{array}
\end{equation*}
Then
\begin{equation*}
  \teinf{x}\, f = -2 < -1 = \aliminf_{y\to x} f < 
  \alimsup_{y\to x} f = 1 < 2 \le \tesup{x}\, f \,.
\end{equation*}
Clearly, $\esslim_{y\to x} f$ and $\alim_{y\to x}f$ do not exist.   
\end{example}

We now recall a sufficient condition for the existence of the approximate limit 
and we provide some integral characterization by means of density measures.
For the scalar case $m=1$ the next results are partially contained in 
\cite{trace} and \cite{dens}. 

\begin{proposition} \label{dm-s10}
Let $\dom\in\cB(\R^n)$, let $f\in\cL^1_{\rm loc}(\dom,\R^m)$,  
let $x\in\cdom$, and let $\alpha\in\R^m$. 

\bgl
\item
We have that
\begin{equation}\label{dm-s10-0}
  \lim_{\de\downarrow 0} \:\mI{B_\de(x)\cap\dom}{|f-\alpha|}{\lem} = 0 
\end{equation}
implies 
\begin{equation*}
  \alim_{y\to x} f(y) = \alpha\,.
\end{equation*}
The converse is true for $f\in\cL^\infty_{\rm loc}(\dom,\R^m)$.
If \reff{dm-s10-0} is satisfied, then
\begin{equation*}
  \lim_{\de\downarrow 0} \:\mI{B_\de(x)\cap\dom}{f}{\lem} = \alpha\,. 
\end{equation*}

\item
Let $\dens_x^\dom\in\LDens_x^\dom$ be  an $\lem$-density at $x$ within
$\dom$. Then 
\begin{equation} \label{dm-s10-1}
  \sI{x}{|f-\alpha|}{\dens^\dom_x} = 0 
\end{equation}
if and only if
\begin{equation} \label{dm-s10-2}
  \alim_{y\to x} f(y)= \alpha \,.
\end{equation}
If \reff{dm-s10-1} or \reff{dm-s10-2} is satisfied, then
$f$ is $\dens_x^\dom$-integrable and 
\begin{equation} \label{dm-s10-3}
  \sI{x}{f}{\dens^\dom_x} = \alpha\,. 
\end{equation}

\item
We have that
\begin{equation} \label{dm-s10-4}
  \esslim_{y\to x}f(y) := \alpha  \qmq{implies} \alim_{y\to x} f(y)= \alpha\,.
\end{equation}
\el
\end{proposition}
\noi
Let us mention that \reff{dm-s10-0} is satisfied with some $\alpha\in\R^m$ for 
$\lem$-a.e. $x\in\dom$ if $f$ is fixed (cf. \cite[p.~44]{evans} for the
measure $\reme{\lem}{\dom}$).
Moreover notice that \reff{dm-s10-0} implies 
\begin{equation*}\label{dm-s10-5}
  \alpha = \sI{x}{f}{\dens_x^\dom} =
  \lim_{\de\downarrow 0} \:\mI{B_\de(x)\cap\dom}{f}{\lem} \,.
\end{equation*}
In the scalar case this means equality in estimate \reff{dm-s2-3} (with
$\la=\lem$, $\C=\{x\}$, $\E=\dom$). 
For open $\dom$ and $x\in\dom$, assertion (1) 
can be found in \cite[p. 162]{ambrosio}, together with an example showing 
that the converse can fail for unbounded $f$. 

\begin{proof}
For (1) we get the first implication from \cite[p. 28]{trace}.
The converse can be found in \cite[p. 162]{ambrosio} for open $\dom$ and 
for general $\dom\in\cB(\R^n)$ we can argue the same way (cf. also 
Theorem 5.1 in \cite{dens}).  
The last implication follows from a standard estimate.

For (2) let first 
\reff{dm-s10-1} be satisfied and assume that \reff{dm-s10-2} is wrong.
Then there are $\eps,\gamma,\delta_k>0$ with $\delta_k\dto 0$ such
that 
\begin{equation*}
  \frac{\lem\big(B_{\delta_k}^\dom(x)\cap\{|f-\alpha|\ge\eps\}\big)}
       {\lem(B_{\delta_k}^\dom(x))} > \gamma \qmq{for all} k\in\N\,.
\end{equation*}
By \reff{dm-s2-2} and \reff{dm-s2-3}   
with $C=\{x\}$, $E=\dom$, 
$A=\{|f-\alpha|\ge\eps\}$,
and $\lme=\lem$, we get 
\begin{equation*}
  \gamma \le
  \liminf_{k\to\infty} \mI{B_{\delta_k}(x)\cap\dom}{\chi_A}{\lem} \le
  \dens_x^\dom(A) = \sI{x}{\chi_A}{\dens_x^\dom} \,.
\end{equation*}
Hence
\begin{equation*}
  \eps\gamma \le
  \sI{x}{\eps\chi_A}{\dens_x^\dom} \le \sI{x}{|f-\alpha|\chi_A}{\dens_x^\dom} 
  \le \sI{x}{|f-\alpha|}{\dens_x^\dom}    \,.         
\end{equation*}
But this contradicts \reff{dm-s10-1} and verifies
\reff{dm-s10-2}. Then $f$ is $\dens_x^\dom$-integrable and it satisfies 
\reff{dm-s10-3} by \cite[p. 28]{trace}.
Let now \reff{dm-s10-2} be satisfied. By \reff{dm-app2} this is equivalent to 
\begin{equation*}
  \alim_{y\to x} |f-\alpha|=0 \,.
\end{equation*}
Then \reff{dm-s10-1} follows from 
\cite[p. 28]{trace}. 

For (3) we have by definition that $\sx |f-\alpha| = 0$ (cf. \reff{ess-con}).
Then \reff{dm-app2} and \reff{dm-s9-1} imply the assertion.
\end{proof}

\begin{corollary} \label{dm-s10a}
Let $\dom\in\cB(\R^n)$, let $f\in\cL^1_{\rm loc}(\dom,\R^m)$, and 
let $x\in\cdom$. If $\alim_{y\to x} f(y)$ exists, then it is unique. 
\end{corollary}
\noi
This is known for $\dom=\R^n$ (cf. \cite[p. 46]{evans}). 
With the equivalence from Proposition~\ref{dm-s10}, the proof is quite
simple. 

\begin{proof}
Let $\dens_x^\dom\in\LDens_x^\dom$ and let 
$\alim_{y\to x} f(y)=\al_j$ for $j=1,2$. By \reff{dm-s10-1},
\begin{equation*}
  |\alpha_1-\alpha_2| = \sI{x}{|\alpha_1-\alpha_2|}{\dens_x^\dom} \le
  \sI{x}{|\alpha_1-f|}{\dens_x^\dom}+\,\sI{x}{|f-\alpha_2|}{\dens_x^\dom} = 0\,.
\end{equation*}
Hence $\alpha_1=\alpha_2$.
\end{proof}

\section{Extreme points and \zo measures}
\label{ep}

Let $M$ be a subset of a linear space. Then $u\in M$ is an 
{\it extreme point} of $M$ if for every $u_1,u_2\in M$ with $u_1\ne u_2$ 
\begin{equation*}
  u=\alpha u_1+(1-\alpha)u_2\,, \z \alpha\in[0,1]
  \qmq{implies} \alpha\in\{0,1\}\,,
\end{equation*}
i.e. $u$ cannot be an interior point of a segment connecting two points of $M$. 
Extreme points are of special relevance for convex sets by
the Krein-Milman theorem (cf. Clarke \cite[p.~168]{clarke-fa},
Zeidler \cite[p.~157]{zeidler-iii}).

\begin{theorem}[Krein-Milman] \label{ep-s1}
Let $X$ be a Hausdorff locally convex vector 
space, assume that $M\subset X$ is a compact
convex subset, and let $M_e$ denote its set of extreme points. Then 
  \begin{equation*}
    M=\overline{\op{conv}}\, M_e 
  \end{equation*}
where $\overline{\op{conv}}$ is the closed convex hull.
\end{theorem}
\noi
Thus a compact convex set is spanned by its extreme points.   
Since $\cL^\infty(\dom)$ is a normed space, its dual 
$\bawl{\dom}$ equipped with the weak$^*$ topology
is a Hausforff locally convex vector space.
Therefore Theorem~\ref{ep-s1} is applicable to its weak$^*$ compact convex
sets. 

\begin{corollary} \label{ep-s1a}
Let $\dom\in\cB(\R^n)$, let $M\subset\bawl{\dom}$ be convex and weak$^*$
compact, and let $M_e$ be the set of its extreme points. 
Then $M=\overline{\op{conv}}\, M_e$.
\end{corollary}

A nontrivial measure $\mu\in\bawl{\dom}$ with $\dom\in\cB(\R^n)$ is called
{\it \zo measure} if 
\begin{equation} \label{ep-e1}
  \me(\bals) = 0 \zmz{or} \me(\bals)=1 
  \qmq{for every} \bals \in \bor{\dom} \,.
\end{equation}
The next result of Toland \cite[p. 64]{toland} shows their importance 
for $\bawl{\dom}$.

\begin{proposition} \label{ep-s2a}
Let $\dom \in \bor{\R^n}$ and let $B^*$ be the closed unit ball 
in $\bawl{\dom}$. Then $\me$ is an extreme point of $B^*$
if and only if $\mu$ or $-\mu$ is a \zo measure.
\end{proposition}
\noi
Since $B^*$ is weak$^*$ compact by the Alaoglu theorem, Corollary~\ref{ep-s1a}
is applicable and tells us that $B^*$ is spanned, up to sign, 
by \zo measures.  

It turns out that \zo measures are density measures that concentrate
near a single point~$x$, i.e. they localize similar as measures $\dens_x^\dom$. 
A more refined analysis  
shows that they only concentrate in direction of a ray emanating from~$x$.  
For that we define the open {\it cone with vertex at} 
$x\in\R^n$, {\it rotational axis} $v\in\R^n\setminus\{0\}$,    
and {\it opening angle} $2\alpha\in(0,\pi)$ by  
\begin{equation}\label{ep-e4}
  K(x,v,\alpha) := \big\{ y\in\R^n \:\big|\:
  y\ne x, \varangle (y-x,v) < \alpha \big\} \,  
\end{equation}
($\varangle$ is the angle between vectors). 
For unbounded $\dom$, where $\infty$ can belong to the core of a \zo measure, 
we also consider cones with vertex at infinity given by
\begin{equation*}
  K(\infty,v,\alpha):= K(0,v,\alpha) \,.
\end{equation*}
\noi
Let us now provide some first properties of \zo measures that can basically be
found in \cite[p.~15-16]{dens} 
(cf. also \cite[p.~31-32]{dens-ar}, \cite[p.~77]{toland}).

\begin{theorem}\label{ep-s3}  
Let $\dom \in \bor{\R^n}$ and let $\me\in\bawl{\dom}$ be a \zo measure. 
Then there is some $x\in\ccl\dom$ such that
\begin{equation} \label{ep-s3-1}
  \cor\me=\{x\}\,, \quad \me\in\Dens_x \qmq{and} \tx{$\me$ is pure.}
\end{equation}
Moreover, there is a unique $v\in \R^n$ with $|v|=1$ such that 
\begin{equation} \label{ep-s3-2}
\me\big( K(x,v,\alpha)\cap \dom\big) = 1 \qmq{for all}
\alpha \in \big(0,\tfrac{\pi}{2}\big) \,.
\end{equation}
\end{theorem}
\noi
Since \zo measures are density measures at some $x\in\ccl\dom$, we use the
notation 
\begin{equation*}
  \ZO_x := \big\{\me\in\Dens_x \:\big|\: \mu \text{ is \zo measure} \big\}\,. 
\end{equation*}
For $\mu\in\ZO_x$ and the related $v$ from \reff{ep-s3-2},
we clearly have that all sets
\begin{equation*}
  K(x,v,\alpha)\cap\dom\cap B_\delta(x) \zmz{with}
  \alpha \in \big(0,\tfrac{\pi}{2}\big)\,, \z \delta>0
\end{equation*}
are aura of $\me$. This readily implies that \zo measures 
at the same $x$ but with different directions $v$ in \reff{ep-s3-2} cannot
agree. 

\begin{proof}
The proof can be found in \cite{dens} up to the case $x=\infty$ in
\reff{ep-s3-2}. For this remaining case  
we take $x=0$ and argue analogously as in \cite{dens}. 
\end{proof}

By Theorem~\ref{dm-s1} and Corollary~\ref{ep-s1a}, 
the Krein-Milman theorem is applicable to the sets  
$\Dens_C$ and it is useful to know their extreme points. 
Notice that $\Dens_\C$ is always strictly contained in the boundary 
$\bd B^*$ of the unit ball $B^*$ in $\bawl{\dom}$. Moreover, $\Dens_\C$ is
planar because it is contained in the affine hyperplane 
\begin{equation*}
  \Big\{\mu\in \bawl{\dom} \:\Big|\: \I{\dom}{}{\mu}=1\Big\}\,.
\end{equation*}
Thus the extreme points of
$\Dens_x$ have to be contained in its boundary relative to $\bd B^*$. 
Therefore, $\Dens_\C$ might have extreme points that are different from those of
$B^*$. This is not the case according to the next result, that
is basically contained in \cite{dens}.

\begin{theorem} \label{ep-s2}
Let $\dom \in \bor{\R^n}$, let $\C$ be a density set, 
and let $\me \in \Dens_\C$.
Then 
\begin{equation*}
  \tx{$\mu$ is extreme point of $\Dens_\C$ \quad if and only if \quad
  $\mu$ is \zo measure.}
\end{equation*}
\end{theorem}

\begin{proof}
In \cite{dens} the assertion is shown for $\C\subset\cl\dom$.
The remaining case $\C=\{\infty\}$ follows the same way.
\end{proof}

A density measure $\dens_x^E\in\LDens_x$
according to Corollary~\ref{dm-s4b} cannot be extreme point of $\Dens_\C$,
even if $E$ has a cusp at $x$. To see this we notice that there is always 
a disjoint decomposition $E=E_1\cup E_2$ with 
\begin{equation*}
  \lem(B_\delta(x)\cap E_j)> \tfrac14 \lem(B_\delta(x)\cap E)
  \qmq{for all} \de>0\,.
\end{equation*}
Then $\dens_\C^E(E_j)>\tfrac14$ for $j=1,2$ by \reff{dm-s2-2} and, hence,
$\dens_\C^E(E)$ cannot be a \zo measure. 

Let us provide a more refined
description where \zo measures live. For that we first assume that 
$\dom$ is just a set and $\cM\subset\cP(\dom)$ be a collection of subsets of
$\dom$. Then a subset $\cF\subset\cM$ is called {\it filter} in $\cM$ if  
\bgl[--]
\item[(a)]
$\dom\in\cF$, $\emptyset\not\in\cF$, 
\item[(b)]
$B_1,B_2\in\cF$ implies $B_1\cap B_2\in\cF$,  
\item[(c)]
$B_2\in\cM$ with $B_2\supset B_1\in\cF$ implies $B_2\in\cF$.
\el
\noi
A filter $\cU\subset\cM$ is called {\it ultrafilter} in $\cM$ 
if $\cF=\cU$ for any filter
$\cF\subset\cM$ with $\cU\subset\cF$. 

We will consider filters $\cF$ in sets 
$\alg\subset\cP(\dom)$ that are an algebra on
$\dom$ (i.e. $\emptyset,\dom\in\cA$ and $\cA$ is stable under
finite unions and intersections). 
For example if $\me\in\ZO_x$, then 
the sets $A$ with $\me(A)=1$ are obviously a filter in the algebra $\cB(\dom)$. 
In general, let $\cF$ be a filter in an algebra $\alg\subset\cP(\dom)$
and let $A\in\alg$ be such that 
\begin{equation*}
  A\not\in\cF \qmq{but} 
  A\cap B\ne\emptyset \zmz{for all} B\in\cF\,.
\end{equation*}
Then we readily get that
\begin{equation}\label{ep-e3}
  \tilde\cF := \big\{\tilde B\in\alg \:\big|\: 
  A\cap B \subset \tilde B,\: B\in\cF \big\} \,
\end{equation}
is a filter containing $\cF$ and the set $A$. Thus $\cF$ is even strictly
contained in $\tilde\cF$. Since a filter has to contain
all supersets of its sets, $\tilde\cF$ is even the smallest filter of that
kind. That implies a simple characterization of ultrafilters
(cf. \cite[p.~11]{rao}). 

\begin{proposition}\label{ep-s6a}
  Let $\alg$ be an algebra on the set $\dom$ and let $\cF$ be a filter in
  $\alg$. 
  Then $\cF$ is ultrafilter in $\alg$ if and only if
\begin{equation} \label{ep-uf}
  \zmz{for every} A\in\alg \quad
  \mbox{either \z $A\in\cF$ \z or \z $\dom\setminus A\in\cF$ }\,.
\end{equation}
\end{proposition}
\noi
Let us provide the short standard proof, that is not contained in \cite{rao},
for the convenience of the reader. 

\begin{proof}
Let $\cF$ be ultrafilter in $\alg$. Assume that $A\in\alg$ but $A\not\in\cF$. 
Then there is
some $B\in\cF$ with $A\cap B=\emptyset$, since otherwise 
$\tilde\cF\supsetneq \cF$ for filter $\tilde\cF$ from \reff{ep-e3} and 
$\cF$ cannot be ultrafilter. Hence $\dom\setminus A\in\cF$ as superset of
$B$ which gives \reff{ep-uf}.
For the opposite let \reff{ep-uf} be satisfied and let 
$\tilde\cF\supsetneq \cF$ be a
filter in $\alg$. Then there is some $A\in\alg$ with $A\in\tilde\cF$ and
$A\not\in\cF$. Hence $\dom\setminus A\in\cF$ by \reff{ep-uf} and
$\dom\setminus A\in\tilde\cF$ by $\cF\subset\tilde\cF$. Consequently 
$A\cap(\dom\setminus A)=\emptyset$ belongs to $\tilde\cF$, a contradiction 
that implies that $\cF$ must be ultrafilter. 
\end{proof}

Notice that, in analogy to Section~\ref{pl},  measures $\mu$, 
and in particular \zo measures, 
can be defined more generally on any set $\dom$ with respect to an algebra 
$\cA$ on it. Thus we obtain the 
next characterization (cf. also \cite[p. 38]{rao}, \cite[p. 43]{toland}).

\begin{theorem}   \label{ep-s6}
  Let $\alg$ be an algebra on the set $\dom$ and let $\mu$ be a measure on
  $\dom$ related to~$\cA$. Then $\me$ is a
  \zo measure if and only if there is an ultrafilter $\cU$ in $\alg$ with
  \begin{equation} \label{ep-s6-1}
    \me(A)=\bigg\{ \begin{array}{ll}   1 & \text{if } A\in\cU\,, \\
                               0 & \text{if } A\notin\cU\,. 
    \end{array}
  \end{equation}
The correspondence between \zo measure $\mu$ and ultrafilter $\cU$ is unique. 
\end{theorem}

\begin{proof}
If $\cU\subset\alg$ is an ultrafilter, then $\me$ given by \reff{ep-s6-1}
is a \zo measure by \reff{ep-uf}. If $\me\in\ba(\dom,\alg)$ is a \zo measure,
then $\cF := \{ A\in\alg\mid \me(A)=1\}$ 
is a filter in $\alg$ that satisfies \reff{ep-uf}. 
By \reff{ep-uf} and \reff{ep-s6-1}
the correspondence between \zo measures and ultrafilters is unique.
\end{proof}

If we apply Theorem~\ref{ep-s6} to a \zo measure 
$\me\in\bawl{\dom}$, then it is related to an ultrafilter 
$\cU\subset\bor{\dom}$ that also belongs to
\begin{equation*}
  \bor{\dom}^+ := \{ B\in\bor{\dom}\mid \lem(B)>0 \} \,,
\end{equation*}
since $\me(B)=0$ for $\lem$-null sets. Otherwise each ultrafilter 
$\cU\subset\bor{\dom}^+$ obviously induces a \zo measure 
$\me\in\bawl{\dom}$ by \reff{ep-s6-1}. 
Thus Theorem~\ref{ep-s6} directly implies the next specialization. 

\begin{corollary}\label{ep-s7}
Let $\dom\in\cB(\R^n)$. Then $\me\in\bawl{\dom}$ is a
  \zo measure if and only if there is an ultrafilter $\cU$ in 
  $\bor{\dom}^+$ such that $\me$ satisfies \reff{ep-s6-1}.
  The correspondence between $\me$ and $\cU$ is unique. 
\end{corollary}
\noi
As in Theorem~\ref{ep-s3} for a \zo measure, 
we obtain a localization of the related ultrafilter. 

\begin{proposition}   \label{ep-s7a}
Let $\dom\in\cB(\R^n)$, let $\me\in\ZO_x$ 
for $x\in\ccl\dom$, and let $\me$ be related to the
ultrafilter $\cU\subset\bor{\dom}^+$ through \reff{ep-s6-1}. Then
\begin{equation*}
  \bigcap_{U\in\,\cU} \ccl U = \{x\} \,.
\end{equation*}
\end{proposition}
\noi  
If $x\in\cl\dom$, then we can obviously replace $\ccl U$ with $\cl U$ in the
assertion.  

\begin{proof}
Since $\me(U)=1$ for all $U\in\cU$ and $\cor{\me}=\{x\}$, we have 
\begin{equation}\label{ep-s7a-5}
  x\in\ccl U \zmz{and} \lem\big(U\cap B_\delta(x)\big) > 0
  \qmq{for all} U\in\cU\,, \z \delta>0\,.
\end{equation}
Assume now that 
\begin{equation}\label{ep-s7a-6}
  y\in\bigcap_{U\in\,\cU} \ccl U \qmq{for some} y\ne x\,.
\end{equation}
Then there is some $\delta>0$ such that 
$B_\delta(x)\cap B_\delta(y)=\emptyset$. Hence we have by \reff{ep-s7a-5}
\begin{equation}\label{ep-s7a-7}
  \emptyset \ne U\cap B_\delta(x) \subset U\setminus B_\delta(y) =
  \big(\dom\setminus B_\delta(y)\big)\cap U  
  \qmq{for all} U\in\cU \,.
\end{equation}
Therefore 
\begin{equation*}
  \tilde\cU := \big\{V\in\bor{\dom} \:\big|\: 
  \big(\dom\setminus B_\delta(y)\big)\cap U 
  \subset V,\: U\in\cU \big\} \,
\end{equation*}
is the smallest filter in $\bor{\dom}$ containing
$\cU$ and $\dom\setminus B_\delta(y)$ (cf. \reff{ep-e3}).
By \reff{ep-s7a-5} and \reff{ep-s7a-7} we have $\tilde\cU\subset\bor{\dom}^+$.
From \reff{ep-s7a-6} we get $\dom\setminus B_\delta(y)\not\in\cU$.
Then $\cU\subsetneq\tilde\cU$. Consequently $\cU$ cannot be ultrafilter in 
$\bor{\dom}^+$, which is a contradiction. Thus \reff{ep-s7a-6} must be wrong
and \reff{ep-s7a-5} implies the assertion. 
\end{proof}

Let us now consider the existence of ultrafilters in $\bor{\dom}^+$.
It is well known that each filter $\cF$ in an algebra $\alg$ is contained in
an ultrafilter $\cU\subset\cA$. But $\bor{\dom}^+$ is not an algebra, since
intersections might be $\lem$-null sets. We say that $\cS\subset\bor{\dom}^+$
has the {\it finite intersection property}  
(in $\bor{\dom}^+$) if
\begin{equation}\label{ep-fi}
  S_1,\dots,S_k\in\cS \zmz{for some} k\in\N \qmq{implies} 
  \bigcap_{j=1}^k S_j \in \bor{\dom}^+\,.
\end{equation}
In this case  
\begin{equation*}
  \cF_{\cS} := \Big\{ F\in \bor{\dom}^+ \:\Big|\: 
  k\in\N,\: S_1,\dots,S_k\in\cS,\: 
  \bigcap_{j=1}^k S_j \subset F  \Big\} 
\end{equation*}
is obviously a filter in $\bor{\dom}^+$ containing $\cS$. 
Notice that each filter in $\bor{\dom}^+$ has the finite intersection
property. 

\begin{proposition} \label{ep-s8}
Let $\cS\subset\bor{\dom}^+$. Then there is an ultrafilter $\cU$ in 
$\bor{\dom}^+$ containing $\cS$ if and only if $\cS$ 
has the finite intersection property.
\end{proposition}
\noi
The proposition gives a simple condition for the existence of ultrafilters
$\cU\in\bor{\dom}^+$ and, by Corollary~\ref{ep-s7}, also for the existence 
of \zo measures $\me\in\bawl{\dom}$ (cf. also \cite[p.~59]{yosida},
\cite[p.~44]{toland}). 
The proof uses standard arguments. Let us sketch it for completeness, since
the result is not known in the stated form.

\begin{proof}
If $\cS$ does not have the finite intersection property, then there are
$S_1,\dots,S_k\in\cS$ with
$\bigcap_{j=1}^kS_j\not\in\bor{\dom}^+$. Consequently there cannot exist a
filter in $\bor{\dom}^+$ that contains $\cS$. Otherwise, if $\cS$ has the
finite intersection property, then $\cF_\cS$ defined above is a filter
containing $\cS$. Thus the set $\frF$ of all filters $\cF$ in $\bor{\dom}^+$
containing $\cS$ is non-empty. If $\frF_0\subset\frF$ is totally ordered
(by set inclusion), 
then $\tilde\cF=\bigcup_{\cF\in\frF_0}\cF$ belongs to $\frF$ and is
an upper bound of $\frF_0$. By Zorn's lemma there is a maximal
element of $\frF$ which is an ultrafilter $\cU\subset\bor{\dom}^+$ containing
$\cS$.    
\end{proof}

In our previous investigation we have characterized \zo measures
$\me\in\bawl{\dom}$ by ultrafilters $\cU\subset\bor{\dom}^+$.
The existence of such measures implicitly follows from Corollary~\ref{ep-s1a}
and Proposition~\ref{ep-s2a}. A more explicit statement is given by  
Corollary~\ref{ep-s7} and Proposition~\ref{ep-s8}. 
For example for any $A\in\bor{\dom}^+$ there is an
ultrafilter containing $A$ and, thus, there is a \zo measure $\me$ with
$\me(A)=1$. Hence a collection of pairwise disjoint sets in $\bor{\dom}^+$
provides different \zo measures. By Theorem~\ref{ep-s3} we know
that \zo measures concentrate near a single point. 
The next example shows that there is a huge
variety of \zo measures concentrated near the same point.

\begin{example}  \label{ep-s9}
For $\dom\in\cB(\R^n)$ and fixed $x\in\cl\dom$ we are
looking for \zo measures $\me$ in $\bawl{\dom}$ with $\cor{\me}=\{x\}$. 
First we consider some $A\in\bor{\dom}$ with a cusp at $x$ in direction
$v\in\R^n\setminus\{0\}$. More precisely, we assume that 
\begin{equation*}
  \qmq{for all $\alpha>0$ there is some $\ti\delta>0$ with}
  A\cap B_{\ti\de}(x)\subset K(x,v,\alpha)
\end{equation*}
(cf. \reff{ep-e4}) and
\begin{equation*}
  A\cap B_\delta(x)\in\bor{\dom}^+ 
  \qmq{for all} \delta >0\,.
\end{equation*}
Then 
\begin{equation*}
  \cS := \{A\cap B_\delta(x)\mid \delta>0 \}
\end{equation*}
has the finite intersection property in $\bor{\dom}^+$. 
By Proposition~\ref{ep-s8}, there is an
ultrafilter $\cU\subset\bor{\dom}^+$ containing $\cS$ 
and, by Corollary~\ref{ep-s7}, there is a corresponding \zo measure
$\me\in\bawl{\dom}$. By
\begin{equation*}
  \bigcap_{S\in\cS} \cl S = \{x\}\,,
\end{equation*}
Proposition~\ref{ep-s7a} implies that $\cor{\me}=\{x\}$ and that 
each $S\in\cS$ is an aura of measure $\me$. 
If $\dom$ has subsets $A$ with cusp at $x$ for different directions $v$, then
Theorem~\ref{ep-s3} implies, that the corresponding \zo measures differ. 
If $\dom$ contains a conical set 
\begin{equation*}
  K(x,\tilde v,\tilde\alpha)\cap B_{\tilde\delta}(x) \qmq{for some}
  \tilde v\in\R^n\setminus\{0\}\,, \z \tilde\alpha,\tilde\delta>0\,,
\end{equation*}
then $\dom$ contains sets with a cusp at $x$ for any direction $v$ contained
in the cone $K(x,\tilde v,\tilde\alpha)$. Thus there are \zo measures $\me$
for all these  
$v$. Since $\me$ related to different $v$ differ, there is a continuum of \zo
measure $\me\in\ZO_x$ in that case. 

Let us now roughly motivate that a set $A\subs\dom$ with a typical cusp at $x$
is aura of a continuum of pairwise different \zo measures.
For technical simplicity we take $\dom=\R^2$ and we assume that $A$ 
has a cusp at $(x,y)=(0,0)$ in direction $(1,0)$. More precisely we assume
that that there is some $\gamma>0$ and some continuous function
$g:(0,\gamma)\to(0,1)$ such that 
\begin{equation*}
  \lim_{x\downarrow 0}\tfrac{g(x)}{x}=0   \qmq{and}
  \{(x,y)\in\R^2\mid |y|<g(x)\,,\; 0<x<\gamma \} \subset A\,.
\end{equation*}
Then we consider pairwise disjoint curves $\theta_\tau:(0,\gamma)\to A$ 
emanating from the origin and given by
\begin{equation*}
  \theta_\tau(s)=(s,\tau g(s)) \qmq{for all} s\in I:=(0,\gamma)\,, \;
  \tau\in\big(-\tfrac12,\tfrac12\big) \,.
\end{equation*}
For $\theta_\tau(I)$ we consider the open neighborhood    
\begin{equation*}
  U_\tau := \big\{ (x,y)\in\R^2\:\big|\: x\in(0,\gamma)\,, \;
  y\in(\tau g(x) - \tfrac12 g(x)^2, \tau g(x) + \tfrac12g(x)^2) \big\} 
  \subset A\,.
\end{equation*}
Obviously $U_\tau$ has a cusp at the origin in direction $(1,0)$ and we have
\begin{equation*}
  U_\tau\cap B_\delta(x)\in\bor{\dom}^+ 
  \qmq{for all} \delta >0\,.
\end{equation*}
With $U_\tau$ instead of $A$ we can now argue as above to
verify the existence of a \zo measure $\mu_\tau$ 
with aura $U_\tau\subset A$.
If $\tau_1\ne\tau_2$, then we readily find some $\delta>0$ such that
\begin{equation*}
  U_{\tau_1}\cap U_{\tau_2} \cap B_\delta(x) = \emptyset\,. 
\end{equation*}
Hence $\mu_{\ta_1}\ne\mu_{\ta_2}$. Thus we have a 
continuum of pairwise different \zo measures $\mu_\tau$ with aura in the
cuspidate subset $A$ of $\dom$. For higher dimensions we can easily extend this
construction.

Let us provide a further strategy to construct \zo measures. For that we
choose a sequence $\{x_k\}\subset\dom$ of pairwise different points
with $x_k\to x$ and a sequence of $\delta_k>0$ such that
\begin{equation*}
  \lem(B_{\delta_k}(x_k)\cap\dom)>0 \qmq{for all} k\in\N\,.
\end{equation*}
Then the collection of the   
\begin{equation}\label{ep-s9-5}
  B_k:= \bigcup_{l\ge k} B_{\delta_l}(x_l) \,
\end{equation}
has the finite intersection property. As above we 
can construct an ultrafilter and a related \zo measure, having each $B_k$ as
aura. 

Moreover, we can consider sets $A\in\cB(\dom)$ with a cusp at point $x$ 
and we can intersect them with the $B_k$ from \reff{ep-s9-5} 
to construct further \zo measures. 
Certainly, the examples discussed here are far from being complete for the 
possible constructions of \zo measures by means of Propositions~\ref{ep-s8}.
However they already demonstrate the huge variety of such measures.  
\end{example}

Let us now consider integrals with respect to \zo measures. 
By \reff{dm-s4g-1} we have for $x\in\ccdom$, $\me\in\ZO_x$, and
$\mu$-integrable $\f\in\cL^1_{\rm loc}(\dom)$ that
\begin{equation*} 
\ix f \, \le \, \sI{x}{f}{\me} \, \le \, \sx f \,.  
\end{equation*}
We want to know which values the integral
$\sI{x}{f}{\me}$ can really attain for \zo measures $\me$ if we fix $f$
and $x$. Including the vector case, 
we call $\gamma\in\R^m$ {\it essential value} of 
$f\in\cL^1_{\rm loc}(\dom,\R^m)$ at $x\in\ccdom$ if 
\begin{equation}\label{ep-ess}
  \lem\big(\{y\in\dom\mid |f-\gamma|<\eps\} \cap B_\delta(x) \big) >0
  \qmq{for all} \eps,\delta > 0\,.
\end{equation}
Notice that the definition does not make sense outside of $\ccdom$ and that  
the set of all essential values at some $x$ can be empty for $f\in\cL^1(\dom)$,
as the example 
\begin{equation*}
  f(x)=\tfrac{1}{\sqrt{|x|}} \zmz{on} (-1,1)\subs\R
\end{equation*}
shows for $x=0$. 
For $\cL^\infty$-functions 
the next results are basically contained 
in Theorem~5.1, Theorem~6.2 and Lemma~6.3 in Toland~\cite{toland} where,
however, the localization near $x$ is not stated explicitely and, instead of
$\ccl\dom$, a one point compactification of $\dom$ in a more abstract setting 
is used. 

\begin{proposition} \label{ep-s20}
Let $\dom\in\cB(\R^n)$, 
let $x\in\ccdom$, let $f\in\cL^1_{\rm loc}(\dom,\R^m)$, and let
$\gamma\in\R^m$. 

\bgl
\item
We have that $\gamma$ is essential value of $f$ if and only if
there is some $\mu\in\ZO_x$ with
\begin{equation} \label{ep-s20-2}
  \mu\big(\{|f-\gamma|<\eps\}\big) = 1 \qmq{for all} \eps>0\,.
\end{equation}

\item
For $\mu\in\ZO_x$ we have \reff{ep-s20-2} if and only if 
\begin{equation} \label{ep-s20-1}
  \sI{x}{f}{\mu}=\gamma\,
\end{equation}

\item
If $\mu\in\ZO_x$ satisfies \reff{ep-s20-1}, then
\begin{equation}\label{ep-s20-3}
  \sI{x}{|f|}{\mu}=|\gamma| \,.
\end{equation}
\el
\end{proposition}
\noi
We readily see that the set of essential values cannot be empty 
for $f\in\cL^\infty_{\rm loc}(\dom)$ at $x\in\ccdom$, since
$\ZO_x\ne\emptyset$. For the proof it
would be quite cumbersome to collect and transfer the 
related proof arguments from \cite{toland}, that are spread around and use
different notation, and to indicate our extensions.
Therefore we provide a complete proof adapted to our setting.

\begin{proof}
For (1) we first assume that $\gamma$ is essential value of $f$.
Then the collection $\cS$ of sets 
\begin{equation*}
  \{|f-\gamma|<\eps\} \cap B_\delta(x) 
  \qmq{for} \eps,\delta > 0\,
\end{equation*}
has the finite intersection property. By Theorem~\ref{ep-s6} and
Proposition~\ref{ep-s8} there is some $\me\in\ZO_x$ such that \reff{ep-s20-2}
is met. If otherwise \reff{ep-s20-2} is satisfied for $\me\in\ZO_x$, then 
\reff{ep-ess} follows and $\gamma$ is essential value of $f$. 

For (2) we first consider $m=1$. Let  
$\me\in\ZO_x$ satisfy \reff{ep-s20-2}. 
Then $f$ is $\me$-integrable, since $f$ is essentially bounded on
$\{|f-\gamma|<\eps\}$ for all $\eps>0$. Thus 
\begin{equation*}
  \sI{x}{f}{\mu} = \I{\{|f-\gamma|<\eps\}}{f}{\mu} \;\in\;
  (\gamma-\eps,\gamma+\eps) \qmq{for all} \eps>0\,, 
\end{equation*}
which gives \reff{ep-s20-1}. Let now \reff{ep-s20-1} be satisfied for 
$\me\in\ZO_x$. If \reff{ep-s20-2} were wrong, then 
there is some $\eps>0$ such $\me\{|f-\gamma|\ge\eps\} = 1$. Hence
\begin{equation*}
  0 \overset{\reff{ep-s20-1}}{=} 
  \sI{x}{f-\gamma\,}{\mu} = \I{\{|f-\gamma|\ge\eps\}}{f-\gamma\,}{\mu}
  \z\:\bigg\{ 
  \begin{array}{ll} \ge\z\eps & \zmz{if}\mu\big(\{f-\gamma\ge\eps\}\big)=1\,, \\
                    \le-\eps & \zmz{if}\mu\big(\{f-\gamma\le-\eps\}\big)=1\,,
  \end{array}
\end{equation*}
which is a contradiction, and \reff{ep-s20-2} follows.

For $m>1$ we use $f=(f_1,\dots,f_m)$ and $\gamma=(\gamma_1,\dots,\gamma_m)$. 
If \reff{ep-s20-2} is satisfied,
then we have for any $\eps>0$ and all $j$ that
\begin{equation*}
  \big\{|f-\gamma|<\eps \big\} \subset
  \{|f_j-\gamma_j|< \eps\} \qmq{and, thus,} 
  \me\big(\{|f_j-\gamma_j|< \eps\}\big)=1 \,.
\end{equation*}
Then, by the scalar case, $\sI{x}{f_j}{\mu}=\gamma_j$ which implies 
\reff{ep-s20-1}. Let now \reff{ep-s20-1} be satisfied and assume that
\reff{ep-s20-2} is wrong. Then, as in the scalar
case, there is some $\eps>0$ such that $\me\{|f-\gamma|\ge\eps\} = 1$.
Thus there is some $j$ with 
$\mu\big(\{|f_j-\gamma_j|\ge\tfrac{\eps}{\sqrt m}\}\big) = 1$. But this gives
a contradiction as in the scalar case and \reff{ep-s20-2} follows.

For \reff{ep-s20-3} we use that
\begin{equation*}
  1=\mu\big(\{|f-\gamma|<\eps\}\big) \le 
  \mu\big(\big\{\big||f|-|\gamma|\big|<\eps\big\}\big)
\end{equation*}
for all $\eps>0$. 
\end{proof}

The previous property of the integral implies some special calculus rule.

\begin{proposition} \label{ep-s22}
Let $\dom\in\cB(\R^n)$, let $x\in\ccdom$, let $\mu\in\ZO_x$ be a \zo measure, 
and let $f,g\in\cL^1_{\rm loc}(\dom,\R^m)$ be $\mu$-integrable. Then 
\begin{equation} \label{ep-s22-4}
  \sI{x}{f\cdot g}{\mu} = \sI{x}{f}{\mu}\cdot\sI{x}{g}{\mu}\,.
\end{equation}
\end{proposition}
\noi
This follows from Theorem 6.2 in Toland \cite{toland} for
$f,g\in\cL^\infty(\dom)$ and $x\in\cdom$ and it can readily be extended to our
setting. Let us briefly sketch the proof for completeness.

\begin{proof}
It is sufficient to consider scalar functions $f_1,f_2$ and we set 
$\sI{x}{f_j}{\mu}=\gamma_j$.
By Proposition~\ref{ep-s20}, $\ga_j$ is essential value of $f_j$. Then,
for any $\eps>0$ there is $\tilde\eps>0$ with
\begin{equation*}
  \{|f_1-\gamma_1|<\tilde\eps\} \cap \{|f_2-\gamma_2|<\tilde\eps\}
  \:\subset\: \{|f_1f_2-\gamma_1\gamma_2|<\eps\} \,.
\end{equation*}
Since $\me \big(\{|f_j-\gamma_j|<\tilde\eps\}\big) = 1$ and since 
$\me(A\cap B)=\me(A)\me(B)$ 
for \zo measures by \reff{ep-e1}, we get
$\me \big(\{|fg-\gamma_1\gamma_2|<\eps\}\big)=1$ for all $\eps>0$.
But this gives $\sI{x}{f_1f_2}{\mu}=\gamma_1\gamma_2$ by
Proposition~\ref{ep-s20}.
\end{proof}

\section{Applications}
\label{app}

In this section we provide several applications for density measures. 
As preparation let us discuss the precise representative of $\lem$-integrable
functions. 

Let $\dom\subset\bor{\R^n}$. 
For a function $f\in L^1_{\rm loc}(\dom,\R^m)$ one has  
\begin{equation}   \label{app-e2}
  \lim_{\de\downarrow 0} \:\mI{B_\de(x)\cap\dom}{|f-f(x)|}{\lem} = 0 \,
 \qmq{for \z $\lem$-a.e. $x\in\dom$} 
\end{equation}
(cf. \cite[p. 44]{evans} with $\mu=\reme{\lem}{\dom}$) while the left hand
side makes sense only for $x\in\cdom$.
For a class $f\in\cL^1_{\rm loc}(\dom,\R^m)$ we thus have, for at least 
$\lem$-a.e. $x\in\dom\cap\cdom$, that there is some $\al\in\R^m$ with
\begin{equation}  \label{app-e0}
  \lim_{\de\downarrow 0} \:\mI{B_\de(x)\cap\dom}{|f-\alpha|}{\lem} = 0 \,.
\end{equation}
In slight deviation from standard use we call $x\in\cdom$ satisfying
\reff{app-e0} {\it Lebesgue point} of~$f$ (with $\al$).  
Notice that sometimes $\alpha$ satisfying \reff{app-e0} 
is called approximate limit of $f$ at $x$, which is a stronger
version than we are using (cf. \cite[p. 160]{ambrosio}).
By Proposition~\ref{dm-s10}, \reff{app-e0} implies 
for any $\lem$-density $\dens_x^\dom\in\LDens_x$ that
\begin{equation} \label{app-e00}
  \alpha = \alim_{y\to x}f(y) = \sI{x}{f}{\dens_x^\dom} =
  \lim_{\de\downarrow 0} \:\mI{B_\de(x)\cap\dom}{f}{\lem} 
  \qmq{for $\lem$-a.e. $x\in\dom\cap\cdom$\,.}
\end{equation}
We in particular have that the equalities in \reff{app-e00} are valid at all 
Lebesgue points $x$ of $f$. 
This suggests the definition of a {\it precise representative} of 
$f\in\cL^1_{\rm loc}(\dom,\R^m)$ at $x\in\cl\dom$ either as usual by
\begin{equation}\label{app-s1-5}
  f^\star(x) := \Bigg\{
  \begin{array}{ll}
    \lim_{\de\downarrow 0} \:\mI{B_\de(x)\cap\dom}{f}{\lem} & 
    \zmz{if the limit exists,}  \\[2pt]
    0 & \zmz{otherwise} 
  \end{array}
\end{equation}
(cf. \cite[p. 46]{evans}) or as integral
representation with an $\lem$-density $\dens_x^\dom\in\LDens_x^\dom$ by
\begin{equation}\label{app-s1-6}
  \pr{f}(x) := \bigg\{
  \begin{array}{ll}
    \sI{x}{f}{\dens_x^\dom} & 
    \zmz{if the integral exists,}  \\[2pt]
    0 & \zmz{otherwise.} 
  \end{array}
\end{equation}
Notice that the existence of the limit in \reff{app-s1-5} and of the integral
in \reff{app-s1-6} requires $x\in\cdom$. Though $\pr f(x)$ and $f^\star(x)$
are (formally) defined for all $x\in\cl\dom$, the points $x$ where
they are agree with the integral or the limit 
in their definition are of special
interest. We say that $f^\star(x)$ {\it exists as limit} if the limit in
\reff{app-s1-5} exists and we say that $\pr f(x)$ {\it exists as integral} if 
$\sI{x}{f}{\dens_x^\dom}$ exists for at least one $\dens_x^\dom\in\LDens_x^\dom$
and, in this case, we always identify $\pr f(x)$ with some existing integral
(which is not unique in general, but at points where \reff{app-e0} is met). 
At $x\in\cdom$ where \reff{app-e0}
fails, $f^\star(x)$ and $\pr f(x)$ can differ. 
$\pr f(x)$ exists as integral for all $x\in\cdom$ if 
$f\in \cL^\infty(\dom)$, but it might depend on the special choice of 
$\dens_x^\dom$ on an $\lem$-null set. If $f$ is continuous at $x$, 
then clearly $\pr f(x)=f^\star(x)=f(x)$. Now we readily obtain 
some calculus rules for $\pr f(x)$.

\begin{proposition} \label{app-s1}
Let $\dom\in\cB(\R^n)$, let $x\in\cdom$, let 
$f_1,f_2\in\cL^1_{\rm loc}(\dom,\R^m)$, let $f_b\in\cL^\infty_{\rm loc}(\dom)$, 
and assume
that $\alim_{y\to x} f_j(y)$ exists for $j\in\{1,2,b\}$.  Then, for all $j$,
\begin{equation*}
  \pr f_j(x) = \alim_{y\to x} f_j(y) = \sI{x}{f}{\dens_x^\dom} 
  \qmq{for all $\lem$-densities} \dens_x^\dom\in\LDens_x^\dom  
\end{equation*}
and, for $c_1,c_2\in\R$, 
\begin{equation*}
  \pr{(c_1f_1+c_2f_2)}(x) = c_1\pr f_1(x)+c_2\pr f_2(x) \qmq{and}
  \pr{(f_1f_b)}(x)=\pr{f_1}(x)\pr{f_b}(x) \,.
\end{equation*}
where the precise representatives on the left hand sides exist as integral and 
are independent of the measure $\dens_x^\dom\in\LDens_x^\dom$.
\end{proposition}
\noi 
Notice that an analogous rule for $f^\star_j$ is valid if
$f_1,f_2\in\cL^\infty_{\rm loc}(\dom)$, but a stronger assumption is needed
for unbounded $f_1,f_2$ (cf. Proposition~\ref{dm-s10}).

\begin{proof}
The representations for the $\pr f_j(x)$ follow from Proposition~\ref{dm-s10}. 
Then we have for any $\lem$-density $\mu\in\LDens_x^\dom$ that
\begin{equation*}
  \sI{x}{\big|(c_1f_1-c_2f_2)-(c_1\pr f_1(x)-c_2\pr f_2(x))\big|}{\mu} \le
  \sum_{j=1}^2 |c_j|\:
  \sI{x}{\big|f_j-\pr{f_j}(x)\big|}{\mu} \,.
\end{equation*}
Furthermore, for 
$f=f_1f_b-\pr{f_1}(x)\pr{f_b}(x)$,   
\begin{equation*}
  \sI{x}{|f|}{\mu} \le
  \|f_b\|_\infty\:\sI{x}{\big|f_1-\pr{f_1}(x)\big|}{\mu} +
  |\pr{f_1}(x)| \:\sI{x}{\big|f_b-\pr{f_b}(x)\big|}{\mu} \,. 
\end{equation*}
Now Proposition~\ref{dm-s10} (2) implies the assertion.
\end{proof}

Let us provide some condition for continuity of $\cL^1$-functions 
by means of the essential limit.

\begin{proposition}\label{app-s28}
Let $\dom\in\cB(\dom)$ and let $f\in\cL^1_{\rm loc}(\dom,\R^m)$. 
If $\esslim_{y\to x}f$ exists for all
$x\in\cdom$, then $\pr f$ is continuous and equals $f^\star$ on $\cdom$.  
\end{proposition}
\noi

\begin{proof}
We have $\pr f=f^\star$ on $\cdom$ 
by Proposition~\ref{dm-s11} and the subsequent
comment. Let now 
$x,x_k\in\cdom$ with $x_k\to x$. Again by Proposition~\ref{dm-s11}, there are 
$\dens_{x_k}\in\Dens_{x_k}$ such that 
\begin{equation*}
  \pr f(x_k) = \sI{x_k}{f}{\dens_{x_k}} \qmq{for all} k\in\N\,.
\end{equation*}
By \reff{ess-bd} we have that $f\in\cL^\infty(B_\de(x)\cap\dom)$ for some
$\de>0$ and we can assume that all $x_k$ belong to $B_\de(x)$.  
Let now $\{x_{k'}\}$ be any subsequence. By Proposition~\ref{dm-s4f} there is
a weak$^*$ accumulation point $\me\in\Dens_{x}$
of $\{\dens_{x_{k'}}\}$ and there is a subsequence $\{x_{k''}\}$ (related 
to~$f$) such that
\begin{equation*}
  \pr f(x) = \sI{x}{f}{\me} = \lim_{k''\to\infty}
  \sI{x_{k''}}{f}{\dens_{x_{k''}}^\dom} = \lim_{k''\to\infty}\pr f(x_{k''}) \,
\end{equation*}
where the first equality follows from Proposition~\ref{dm-s11}. 
Since $\pr f(x)$ is independent of any $\mu\in\Dens_x$ by
Proposition~\ref{dm-s11}, the subsequence principle gives the assertion.
\end{proof}

\subsection{Weak convergence in $\cL^\infty(\dom)$}
\label{app-wc}

We assume that $\dom\in\cB(\R^n)$ and study the weak convergence 
in $\cL^\infty(\dom)$. A comprehensive treatment of this subject 
in general measure spaces and later specialized to locally compact Hausdorff
spaces $\dom$ can be found in Toland \cite{toland}. In this section 
we intend to transfer some central results of \cite{toland} 
to our setting, which is typical for the treatment of partial
differential equations, and to combine them with our previous results.
A significant difference is that \cite{toland} uses a one-point
compactification for $\dom$, while we are working   
with $\ccl\dom$ and obtain more differentiated results on the boundary. 

Since $\cL^\infty(\dom)^*$ is isometrically isomorphic
to $\bawl{\dom}$, a sequence $\{f_k\}\subset\cL^\infty(\dom)$ is weakly
convergent to $f\in\cL^\infty(\dom)$ by definition if
\begin{equation} \label{app-e1}
  \I{\dom}{f_k}{\me} \to \I{\dom}{f}{\me} \qmq{for all} 
  \me\in\bawl{\dom} 
\end{equation}
(cf. \reff{pl-dual}). Let us start with some equivalent 
formulations for weak convergence. 

\begin{theorem}  \label{app-s4}
Let $\dom\in\cB(\R^n)$ and let $f_k, f\in\cL^\infty(\dom)$. 
Then the following assertions are equivalent:
\bgl
\item
$f_k\wto f$\,,
\item
$|f_k-f|\wto 0$\,,
\item
\mbox{}\vspace{-6mm} 
\begin{equation} \label{app-e1a}
  \I{\dom}{f_k}{\me} \to \I{\dom}{f}{\me} \qmq{for all \zo measures} 
  \me\in\bawl{\dom} \,, \hspace{10mm}
\end{equation}
\item \mbox{}\vspace{-6mm} 
\begin{equation} \label{app-e1b}
  \I{\dom}{f_k}{\me} \to \I{\dom}{f}{\me}
  \qmq{for all} \mu\in\Dens_x\,, \; x\in\ccdom\,,  \hspace{32mm}
\end{equation}
\item
for all $\gamma>0$ and for any subsequence $\{f_{k_j}\}_j$ there
is some $m\in\N$ such that
\begin{equation} \label{app-e1c}
  \lem\Big(\bigcap_{j=1}^m \big\{|f_{k_j}-f|>\gamma\big\} \Big) = 0\,.
\end{equation}
\el
\end{theorem} 
\noi
The equivalence of (1)-(3) can be found in 
\cite[pp. 67, 68, 70]{toland}. Since the equivalence of 
(1)-(4) easily follows from our previous results, we give a short proof for
that. The last condition corresponds to Theorem~8.7 in \cite{toland}. 
Let us sketch that proof within our setting for the convenience of the reader. 

\begin{proof}
Clearly (1) implies (3) and (4). The opposite directions follow from 
Corollary~\ref{ep-s1a}, Proposition~\ref{ep-s2a}, Theorem~\ref{ep-s3}, and
Theorem~\ref{ep-s2}.
We have that (3) implies (2) by Proposition~\ref{ep-s20} (3).
If $|f_k-f|\wto 0$, we use that
$|\I{\dom}{f_k-f}{\mu}|\le\I{\dom}{|f_k-f|}{\mu}$ for 
any \zo measure $\mu$ and (3) follows.

Let now $|f_k-f|\wto 0$ and assume that \reff{app-e1c} is wrong. 
Then there are some
$\gamma>0$ and some subsequence $\{f_{k_j}\}$ such that 
\begin{equation*} 
  S_j:=\big\{|f_{k_j}-f|>\gamma\big\} \in \cB(\dom)^+
  \qmq{for all} j\in\N
\end{equation*}
and $\cS:=\{S_j\mid j\in\N\}$ has the finite intersection property in
$\cB(\dom)^+$.
Using Corollary~\ref{ep-s7} and Proposition~\ref{ep-s8}, we get some 
\zo measure $\mu\in\bawl{\dom}$ with $\mu(S_j)=1$ for all $j\in\N$. Hence
\begin{equation*}
  \I{\dom}{|f_{k_j}-f|}{\mu} = \I{\{|f_{k_j}-f|>\gamma\}}{|f_{k_j}-f|}{\mu}
  \overset{j\to\infty}{\not\!\!\longrightarrow} 0
\end{equation*}
which contradicts (3) with $|f_k-f|$ instead of $f_k$ and, thus, 
\reff{app-e1c} follows.  

For the opposite let \reff{app-e1c} be satisfied and assume that 
$|f_k-f|\not\wto 0$. By (3) there is some \zo measure $\mu\in\bawl{\dom}$,
some $\gamma>0$, and a subsequence $\{f_{k_j}\}$ such that
\begin{equation*}
  \I{\dom}{|f_{k_j}-f|}{\mu} =: \gamma_{k_j} > \gamma 
  \qmq{for all} j\in\N\,.
\end{equation*}
By Proposition~\ref{ep-s20} with $\eps=\gamma_{k_j}-\gamma$, 
\begin{equation*}
  1 = 
 \mu\big(\big\{\big||f_{k_j}-f|-\gamma_{k_j}\big|<\gamma_{k_j}-\gamma\big\}\big) 
  \le \mu\big(\big\{|f_{k_j}-f|>\gamma\big\}\big) = 1
  \qmq{for all} j\in\N\,. 
\end{equation*}
By Corollary~\ref{ep-s7}, measure $\mu$ corresponds to some ultrafilter in 
$\cB(\dom)^+$. Thus  
\begin{equation*}
   \mu\Big(\bigcap_{j=1}^m\big\{|f_{k_j}-f|>\gamma\big\}\Big) = 1  
  \qmq{for all} m\in\N\, 
\end{equation*}
and, hence, \reff{app-e1c} is wrong. But this is a contradiction and 
verifies $|f_k-f|\wto 0$.
\end{proof}
\noi
Though these conditions are analytically useful, it might be difficult to
check them in examples. Therefore we want to provide further conditions for
weak convergence. Proposition~\ref{dm-s7} and \reff{app-e1b} directly imply
some sufficient condition (cf. also \cite[p. 82]{toland}). 

\begin{proposition}  \label{app-s5}
Let $\dom\in\cB(\R^n)$, let $f$, $f_k\in\cL^\infty(\dom)$, and let
\begin{equation} \label{app-s5-1}
  \lim_{k\to\infty} \ix\, (f_k-f) =
  \lim_{k\to\infty}\sx\, (f_k-f) = 0
  \qmq{for all} x\in\ccdom\,.
\end{equation}
Then $f_k\wto f$.
\end{proposition}
\noi
At first glance one might think, that \reff{app-s5-1} implies norm convergence
in $\cL^\infty(\dom)$. The next example shows that this is not the case. 

\begin{example} \label{app-s5a}
Let $\dom=(0,1)\subset\R$ and let
\begin{equation*}
  f(x)=0 \zmz{on}\dom\,, \quad
  f_k(x) = \bigg\{ 
  \begin{array}{ll} 
     1 & \text{on } \big(\frac{1}{2^{k+1}},\frac{1}{2^k}\big) \,, \\
     0 & \text{otherwise}\,.
  \end{array}
\end{equation*}
Then 
\begin{equation*}
\lim_{k\to\infty}\ix f_k = \lim_{k\to\infty}\sx f_k = 0 
\qmq{for all} x\in[0,1]=\ccdom \,.
\end{equation*}
Hence, by Proposition~\ref{app-s5}, $f_k\wto 0$. 
Moreover $\|f_k\|_\infty=1$ for all $k\in\N$ and, thus, we obtain that
$f_k\not\to 0$. 
\end{example}
\noi 
Now, \reff{app-e1b} and  Proposition~\ref{dm-s7} applied to $f_k-f$ 
could suggest that \reff{app-s5-1} is also necessary for weak convergence. 
But this is wrong as the next example shows
(cf. \cite[p. 82]{toland}).

\begin{example} \label{app-s5b}
Let $\dom=(0,1)\subset\R$ and let $\cN=\{N_l\}_{l\in\N}\subset\cP(\cl\dom)$ 
be a countable neighborhood base for $[0,1]=\cl\dom=\ccdom$ as, e.g.,
\begin{equation*}
  \cN=\{B_\rho(q)\cap\cl\dom\mid \rho,q \in\cl\dom\cap\Q \} \quad
  \mbox{($\Q$ rational numbers)}\,.
\end{equation*}
Then there is a sequence  of pairwise disjoint subsets $A_k\subset\cl\dom$ 
such that 
\begin{equation*}
  \lem(A_k\cap N_l)>0 \qmq{for all} k,l\in\N\,
\end{equation*}
(cf. Theorem 2.25 in \cite{toland}). For $f_k=\chi_{A_k}$ we obtain that
$f_k\wto 0$ by Theorem~\ref{app-s4} (5). Moreover,
\begin{equation*}
  \lim_{k\to\infty} \ix f_k = 0\,, \z \lim_{k\to\infty}\sx f_k = 1
  \qmq{for all} x\in[0,1]=\ccdom\,.
\end{equation*}
Hence \reff{app-s5-1} is violated for $f=0$ and, thus, it is not necessary
for weak convergence. 

Alternatively we can consider $f_k=\chi_{A_1\cap I_k}$ for
$I_k=\big(0,\frac1k\big)$. Then, by the same argument, 
$f_k\wto 0$. But now 
\reff{app-s5-1} is violated only at $x=0$. 
\end{example}

Now we provide some necessary conditions for weak convergence 
in $\cL^\infty(\dom)$ by means of $\lem$-densities
$\dens_x^E$, which are not \zo measures.
Let us start with $\lem$-densities $\dens_x^\dom$ at $x$ within $\dom$. 
By Corollary~\ref{dm-s4b} there is an $\lem$-density
$\dens_x^\dom\in\LDens_x^\dom$ for each $x\in\cdom$.
Since all $f\in\cL^\infty(\dom)$ are $\dens_x^\dom$-integrable,
their precise representative according to \reff{app-s1-6} is given by
\begin{equation*}
  \pr{f}(x) = \sI{x}{f}{\dens_x^\dom}
  \qmz{for all} x\in\cdom\,
\end{equation*}
and it agrees with $f^\star(x)$ as limit 
for $\lem$-a.e. $x\in\dom\cap\cdom$.  
Hence we can directly derive a pointwise necessary condition for weak
convergence from \reff{app-e00} and Theorem~\ref{app-s4}~(4).

\begin{proposition} \label{app-s6}
Let $\dom\in\cB(\R^n)$ and let $f_k\wto f$ in $\cL^\infty(\dom)$. Then
\begin{equation} \label{app-s6-1}
  \pr{f}(x) = \lim_{k\to\infty} \pr{f_k}(x) \qmz{for all} x\in\cdom\,  
\end{equation}
where $\pr f(x)$ and $\pr f_k(x)$ have to be related to the same
$\dens_x^\dom$ which, however, can be arbitrary chosen in  
$\LDens_x$.
Moreover, let $\dom_0\subs\cdom$ be the $\lem$-null set of all $x$ where 
$f^\star(x)$ and $f_k^\star(x)$ do not exist as limit 
for almost all $f_k$. Then
\begin{equation} \label{app-s6-0}
  f^\star(x) = \lim_{k\to\infty} f_k^\star(x) \qmz{for $\lem$-a.e.} 
  x\in\cdom\setminus\dom_0\,.  
\end{equation}
\end{proposition}
\noi
Let us emphasize that \reff{app-s6-1} holds for all $x$ in $\cdom$\!. 
Furthermore,   
\reff{app-s6-1} and \reff{app-s6-0} take into account 
boundary points in $\bd\dom\cap\cdom$. While this does not influence 
\reff{app-s6-0} in the typical case where $\bd\dom\cap\cdom$ is an $\lem$-null
set, it is relevant in \reff{app-s6-1}. Thus, in particular \reff{app-s6-1}
sharpens the result from \cite[p. 69]{toland} which is of the type 
\reff{app-s6-0} without boundary points. 
The next example shows that the stronger version \reff{app-s6-1}
can make a difference.

\begin{example} \label{app-s8}
Let $\dom=(0,1)\subset\R$ and let
\begin{equation*}
  f_k(x) = \bigg\{ 
  \begin{array}{ll} 
     1 & \text{on } \big(0,\frac 1k\big) \, \\
     0 & \text{otherwise}\,
  \end{array}
  \qmq{for all} k\in\N\,.
\end{equation*}
Obviously, $\pr f_k(x)=f_k^\star(x)\to 0$ for all $x\in\dom$. 
Hence $f=0$ is the only candidate for a weak limit. But, with any 
$\cL^1$-density $\dens_0^\dom$ at $x=0$ within $\dom$,  
\begin{equation*}
  f_k^\star(0)=\pr{f_k}(0) = \sI{0}{f_k}{\dens_0^\dom}=1  
  \zmz{for all} k\in\N 
  \qmq{while} f^\star(0)=\pr{f}(0) = 0\,.
\end{equation*}
Notice that \reff{app-s6-0}, that has to be true only for $\lem$-a.e. $x$, 
is still satisfied with $f=0$, but 
\reff{app-s6-1} is violated at $x=0$ and implies $f_k\not\wto 0$.
This shows the advantage of $\pr f$, while $f^\star$, which is usually defined
only on $\dom$, cannot rule out weak convergence in this case
(cf. also \cite[p.~69]{toland}).

For a slight modification we choose some continuous function 
$g\in\cL^\infty(\dom)$ such that $g^\star(0)$ does not exist as limit.
For $g_k(x)=f_k(x)+g(x)$ we then have that 
$g_k^\star(x)=\pr g_k(x)$ tends to  $g(x)$ on $(0,1)$. Therefore 
$g$ is the only candidate for the weak limit. Since $g^\star(0)$
is not defined as limit, it cannot enter \reff{app-s6-0} at all. 
However, $g\in\cL^\infty(\dom)$ is $\dens_0^\dom$-integrable for any 
$\lem$-density $\dens_0^\dom$ in $\LDens_0^\dom$. Therefore 
\begin{equation*}
  \sI{0}{g}{\dens_0^\dom}=\gamma 
  \qmq{for some} \gamma\in\R\,
\end{equation*}
and
\begin{equation*}
  \pr{g_k}(0) = \sI{0}{g_k}{\dens_0^\dom}=1+\gamma 
  \zmz{for all} k\in\N 
  \qmq{while} \pr{g}(0) = \gamma\,.
\end{equation*}
Consequently, we have $g_k\not\wto g$ by \reff{app-s6-1}. 
Notice that $\ga$ might depend on the special
choice of $\dens_0^\dom$ which, however, does not bother our analysis. 
\end{example}
\noi
Based on Corollary~\ref{dm-s4b} we can readily extend
Proposition~\ref{app-s6}.

\begin{corollary} \label{app-s6a}
Let $\dom\in\cB(\R^n)$, let $f_k\wto f$ in $\cL^\infty(\dom)$, let $x\in\cdom$, 
and let $E\in\cB(\dom)$ be such that
\begin{equation} \label{app-s6a-0}
  \lem(B_\delta(x)\cap E)>0 \qmq{for all} \delta>0\,.
\end{equation}
Then
\begin{equation} \label{app-s6a-1}
  \sI{x}{f}{\dens_x^E} = \lim_{k\to\infty} \sI{x}{f_k}{\dens_x^E} \,
\end{equation}
for each $\lem$-density $\dens_x^E\in\LDens_x^E$. 
\end{corollary}
\noi
The next example shows, that this really sharpens Propositions~\ref{app-s6}.

\begin{example} \label{app-s10}
For $\dom=B_2(0)\subset\R^2$ we consider the cuspidate subsets
\begin{equation*}
  \dom_k = \big\{(x,y)\in\dom \:\big|\: 
           \sqrt{|x|}\le y \le \tfrac1k \big\} \,, \quad k\in\N\,,
\end{equation*}
and the functions in $\cL^\infty(\dom)$ given by
\begin{equation*}
  f(x) = 0 \zmz{on} \dom\,, \quad
  f_k(x) = \bigg\{ 
  \begin{array}{ll} 
     1 & \text{on } \dom_k\,,   \\
     0 & \text{on } \dom\setminus\dom_k\,.
  \end{array}
\end{equation*}
Then $\pr f_k(x)=f_k^\star(x)=0$ for all $x\in\cl\dom=\cdom$ and both 
\reff{app-s6-1} and \reff{app-s6-0} are satisfied.
But, let $\dens_0^{\dom_1}$ be an $\cL^2$-density at
$x=0$ within $\dom_1$. Then
\begin{equation*}
  \lim_{k\to\infty} \sI{0}{f_k}{\dens_0^{\dom_1}} = 1 \ne 0 = 
  \sI{0}{f}{\dens_0^{\dom_1}} \,.
\end{equation*}
Hence, by \reff{app-s6a-1}, we conclude that $f_k\not\wto f$ in
$\cL^\infty(\dom)$. Notice that, alternatively, we can readily apply
\reff{app-e1c} in this case.  
\end{example}

\subsection{Integral characterization of derivatives}
\label{app-der}

We will discuss several notions of differentiability 
for Lebesgue integrable functions, we study their properties, 
and we derive related integral conditions. In particular we show how the
classical pointwise definition of differentiability for functions can be
extended to classes of $\cL^1$-functions. 

Let $\dom\in\cB(\dom)$ and let 
$f\in\cL^1_{\rm loc}(\dom,\R^m)$. 
We say that $f$ is {\it approximately differentiable} at $x\in\cdom$ if there
are some $\alpha\in\R^m$ and a linear map $L:\R^n\to\R^m$ such that 
\begin{equation} \label{ader}
  \alim_{y\to x} \frac{f(y)-\alpha-L(y-x)}{|y-x|} = 0\,.
\end{equation}
We call $D_{\rm ap}f(x):=L$ {\it approximate derivative} of $f$ at $x$ (with
$\alpha$).  

\begin{theorem} \label{app-s25}
Let $\dom\in\cB(\R^n)$, let $x\in\cdom$, 
let $f\in\cL^1_{\rm loc}(\dom,\R^m)$, let $\alpha\in\R^m$, and let
$\dens_x^\dom\in\LDens_x^\dom$ be an $\lem$-density at $x$ within $\dom$. 
Then  $f$ is approximately differentiable at $x$ with 
$D_{\rm ap}f(x)=L$ and with $\alpha$ if and only if there is 
a linear map $L:\R^n\to\R^m$ with
\begin{equation}\label{app-s25-1}
   \sI{x}{\frac{|f(y)-\alpha-L(y-x)|}{|y-x|}}{\dens^\dom_x} = 0 \,. 
\end{equation}
In the case of approximate differentiability at $x$ with $\alpha$ we have that
\begin{equation}\label{app-s25a-1}
  \alpha = \alim_{y\to x} f(y) = \pr f(x) = \sI{x}{f}{\dens_x^\dom} \,.
\end{equation}
If there is an open convex cone $K_x\subset\R^n$ with vertex at $x$ 
and some $r>0$ such that 
\begin{equation} \label{app-s25a-2}
  K_x\cap B_r(x) \subset \dom \,,
\end{equation}
then $D_{\rm ap}f(x)$ is unique if it exists. 
\end{theorem}
\noi
Notice that one has uniqueness of $D_{\rm ap}f(x)$ if 
$\dom\subset\R^n$ is open and $x\in\dom$. 

\begin{proof}
Let $\dens_x^\dom$ be an $\lem$-density at $x$ within $\dom$. Then
\reff{ader} is equivalent to \reff{app-s25-1} by Proposition~\ref{dm-s10}.
For \reff{app-s25a-1} we set 
\begin{equation*}
  f_1(y)=\frac{f(y)-\alpha}{|y-x|}\,, \quad 
  f_2(y)=D_{\rm ap}f(x)\,\frac{y-x}{|y-x|}\,,
  \quad g(y) = |y-x|\,.
\end{equation*}
Then $f_1-f_2$ is $\dens_x^\dom$-integrable 
by \reff{app-s25-1} and $f_2$, $g$ are $\dens_x^\dom$-integrable 
as $\cL^\infty$-functions. 
As sum, also $f_1$ is $\dens_x^\dom$-integrable. Since
$|gf_1|\le|f_1|$ on $B_1(x)\cap\dom$, we get that $|gf_1|$ is 
$\dens_x^\dom$-integrable too (cf. \cite[p.~113]{rao}). Therefore, for any
$\de>0$, 
\begin{equation*}
 \sI{x}{|f-\alpha|}{\dens_x^\dom} =
 \I{B_\de(x)\cap\dom}{|y-x|\,|f_1(y)|}{\dens_x^\dom} \le
 \de\:  \sI{x}{|f_1|}{\dens_x^\dom} \,.
\end{equation*}
Consequently, $\sI{x}{|f-\alpha|}{\dens_x^\dom} = 0$ and 
Proposition~\ref{dm-s10} implies \reff{app-s25a-1}. 

For uniqueness we assume that $L_1$, $L_2$ are approximate derivatives of $f$
at $x$. Then, by \reff{app-s25-1},
\begin{equation*}
  \sI{x}{\big|(L_1-L_2)\tfrac{y-x}{|y-x|}\big|}{\dens_x^\dom} \le
  \sum_{j=1}^2\sI{x}{\big|\tfrac{f(y)-\pr f(x)}{|y-x|}-
          L_j\tfrac{y-x}{|y-x|}\big|}{\dens_x^\dom} = 0 \,.
\end{equation*}
Thus, by Proposition~\ref{dm-s10}, 
\begin{equation*}
  \alim_{y\to x} g(y)=0 \qmq{for} 
  g(y):=(L_1-L_2)\tfrac{y-x}{|y-x|} \,.
\end{equation*}
Since $g\in\cL^\infty(\dom)$, Proposition~\ref{dm-s10} (1) implies,
with $B_\de^\dom(x)=B_\de(x)\cap\dom$,
\begin{equation} \label{app-s25a-8}
  \lim_{\delta\dto 0} \mI{B_\delta^\dom(x)}{|g|}{\lem}=0 \,.
\end{equation}
If $L:=L_1-L_2\ne 0$, then the kernel $\{z\in\R^n\mid Lz=0\}$ 
is a strict subspace of $\R^n$ and, thus, it is an $\lem$-null set.
Then, by \reff{app-s25a-2}, there is an $\tilde y\in K_x\cap B_r(x)$ with
$L(\tilde y-x)\ne 0$. Since $K_x$ is open, there is some $\rho>0$ such that
\begin{equation*}
  B_\rho(\tilde y)\subset K_x \qmq{and}
  |L\big(\tfrac{y-x}{|y-x|}\big)|\ge
  \tfrac12\big|L\big(\tfrac{\tilde y-x}{|\tilde y-x|}\big)\big|=:\gamma
  \zmz{for all} y\in\tilde K_x  
\end{equation*}
where $\tilde K_x$ is the cone hull of $B_\rho(\tilde y)$ related to the vertex
$x$. Then $\tilde K_x\subset K_x$ and there is some $\tilde\gamma>0$ with
\begin{eqnarray*}
  \mI{B_\delta^\dom(x)}{|g|}{\lem} 
&=&  \tfrac{1}{\lem(B_\delta^\dom(x))} \Big(
  \I{B_\delta^\dom(x)\cap\tilde K_x}{\big|L\tfrac{y-x}{|y-x|}\big|}{\lem} +
  \I{B_\delta^\dom(x)\setminus\tilde K_x}{\big|L\tfrac{y-x}{|y-x|}\big|}{\lem}
  \Big)  \\
&\ge&
  \gamma\: \frac{\lem(B_\delta^\dom(x)\cap\tilde K_x)}{\lem(B_\delta^\dom(x))} 
\;=\; \tilde\gamma \;>\; 0 \qmq{for all} 0<\delta<r  \,
\end{eqnarray*}
(recall \reff{app-s25a-2}).
But this contradicts \reff{app-s25a-8} and implies $L_1=L_2$.
\end{proof}

\begin{remark} \label{app-s25c}  
(1) The approximate derivative is usually defined on open sets $\dom\subs\R^n$
for points $x\in\dom$ and the definitions found in the literature differ in
detail. In \cite[p. 232]{evans} the approximate derivative is defined for
functions $f\in L^1_{\rm loc}(\dom,\R^m)$ (and not for classes as above)
with $\alpha=f(x)$ in the definition. In \cite[p. 160, 165]{ambrosio} a
stronger notion of approximate limit based on \reff{dm-s10-0} is used and the
definition of the approximate derivative contains the approximate limit of $f$
at $x$ instead of $\alpha$.    

(2) It is well known that BV functions $f\in\cB\cV_{\rm loc}(\dom)$ 
(with $\dom$ open)
are approximately differentiable at $\lem$-a.e. $x\in\dom$. 
The approximate derivative $D_{\rm ap}f$ 
agrees with (the density of) the weakly
absolutely continuous part $D_{\rm ac}f$ with respect to $\lem$ 
(which is usually called absolutely continuous
part) of the weak derivative $D_wf$ of $f$
(cf. \cite[p. 233]{evans}, \cite[p. 176]{ambrosio}). 
Since Sobolev functions 
$f\in\cW^{1,p}_{\rm loc}(\dom)$ belong to $\cB\cV_{\rm loc}(\dom)$, where the 
weak derivative $D_wf$ agrees with $D_{\rm ac}f$, 
also these functions are approximately differentiable $\lem$-a.e. 
and $D_{\rm ap}f$ equals $D_wf$ as $\cL^1$-function. 

(3) If $f\in\cL^1_{\rm loc}(\dom)$ is approximately differentiable at 
$x\in\cdom$ with $\al$, then $f^\star(x)$, if it exists as limit, 
must not equal $\al$ in general. To see this we consider 
$\dom=B_1(0)\subs\R^2$ and let $\dom'\subs\dom$ have a cusp at $x=0$.   
If $f=0$ on $\dom\setminus\dom'$, then $f$ is approximately differentiable at
$x=0$ with $D_{\rm ap}f(0)=0$ and $\al=\pr f(0)=0$. But one can choose 
$\dom'$ and $f$ on it in such a way 
that $f^\star(0)\ne 0$ (cf. \cite[p.~31]{trace}).
\end{remark}

The integral characterization \reff{app-s25-1} 
of the approximate derivative 
allows to derive some elementary calculus rules quite easily. 

\begin{proposition} \label{app-s25b}
Let $\dom\in\cB(\R^n)$ and let
$f_1,f_2\in\cL^1_{\rm loc}(\dom,\R^m)$ and $f_b\in\cL^\infty(\dom)$
be approximately differentiable at $x\in\cdom$. Then also
$c_1f_1+c_2f_2$ and $f_1f_b$ are approximately differentiable at $x$ 
for all $c_1,c_2\in\R$ with
\begin{equation*}
  D_{\rm ap}(c_1f_1+c_2f_2)(x) = 
  c_1 D_{\rm ap}f_1(x) + c_2 D_{\rm ap} f_2(x) \,,
\end{equation*}
\begin{equation*}
  D_{\rm ap}(f_1f_b)(x) = 
  \pr f_b(x) D_{\rm ap}f_1(x) + \pr f_1(x) \otimes D_{\rm ap}f_b(x) \,
\end{equation*}
(where $\otimes$ denotes the tensor product).
\end{proposition}
\noi
By Remark~\ref{app-s25c} (2), we obtain a product rule for 
Sobolev functions in $\cW^{1,1}_{\rm loc}(\dom)$ for $\lem$-a.e. $x\in\dom$, 
that extends the rule in \cite[p. 129]{evans} for the product 
of essentially bounded Sobolev functions. 

\begin{proof}
Let $\dens_x^\dom\in\LDens_x°\dom$ and let
$D_{\rm ap}f_j(x)=L_j$ for $j\in\{1,2,b\}$.
We use \reff{app-s25-1}, \reff{app-s25a-1} and conclude 
\begin{eqnarray*}
&& \hspace*{-20mm}
 \sI{x}{\frac{|c_1f_1(y)+c_2f_2(y)-(c_1\pr{f_1}(x)+c_2\pr{f_2}(x)) -
   (c_1L_1+c_2L_2)(y-x)|}{|y-x|}}{\dens^\dom_x} \\
&\le&
 \sum_{j=1}^2 
 |c_j| \: \sI{x}{\frac{|f_j(y)-\pr{f_j}(x)-L_j(y-x)|}{|y-x|}}{\dens^\dom_x} 
\;=\; 0 \,, 
\end{eqnarray*}
which gives the sum rule. For the product rule we first have 
$\pr{(f_1f_b)}(x)=\pr{f_1}(x)\pr{f_b}(x)$ by Proposition~\ref{app-s1}.
Then, with
$g_j(y):=f_j(y)-\pr{f_j}(x)-L_j(y-x)$, $j\in\{1,b\}$, we get
\begin{eqnarray*}
  f(y)
&:=& (f_1f_b)(y)-\pr{(f_1f_b)}(x)-
 \big(\pr{f_b}(x)L_1 + \pr{f_1}(x) \otimes L_b\big)(y-x) \\
&=&
  g_1(y)f_b(y) + \pr{f_1}(x)g_b(y) + (f_b(y)-\pr{f_b}(x))L_b(y-x) \,.
\end{eqnarray*}
Moreover, $\pr f_b(x) = \alim_{y\to x} f_b(y)$ by \reff{app-s25a-1}.
Hence, with Proposition~\ref{dm-s10},
\begin{eqnarray*}
  \sI{x}{\frac{|f(y)|}{|y-x|}}{\dens^\dom_x} 
&\le&
 \|f_b\|_\infty\:
 \sI{x}{\frac{ |g_1(y)|}{|y-x|}} {\dens^\dom_x} 
\;+\; 
 |\pr{f_1}(x)|\:  \sI{x}{\frac{ |g_b(y)|}{|y-x|}} {\dens^\dom_x} \\ 
&&\;+\;
 \|L_b\|\:   \sI{x}{|f_b(y)-\pr{f_b}(x)|}{\dens^\dom_x} 
\;=\; 0 \,, 
\end{eqnarray*}
which gives the product rule.
\end{proof}

For a stronger notion of differentiability we can replace the approximate
limit in \reff{ader} by the essential limit (cf. \reff{ess-lim}). 
As before we assume that $\dom\in\cB(\dom)$ and  
$f\in\cL^1_{\rm loc}(\dom,\R^m)$. 
We say that $f$ is {\it essentially differentiable} at $x\in\cdom$ if there
are some $\alpha\in\R^m$ and a linear map $L:\R^n\to\R^m$ such that 
\begin{equation} \label{eder}
  \esslim_{y\to x} \frac{f(y)-\alpha-L(y-x)}{|y-x|} = 0\,.
\end{equation}
Then $D_{\rm es}f(x):=L$ is called {\it essential derivative} of $f$ at $x$
(with $\alpha$). 

\begin{theorem} \label{app-s26}
Let $\dom\in\cB(\dom)$, let $x\in\cdom$, let  
$f\in\cL^1_{\rm loc}(\dom,\R^m)$, and let $\alpha\in\R^m$.
Then we have that $f$ is essentially differentiable at $x$ with 
$D_{\rm es}f(x)=L$ and with $\alpha$ 
if and only if there is a linear map $L:\R^n\to\R^m$ 
such that for any $\E\in\cB(\dom)$ with 
\begin{equation} \label{app-s26-0}
  \lem(B_\delta(x)\cap\E)>0 \qmq{for all} \delta>0
\end{equation}
there is an $\lem$-density $\dens_x^\E\in\LDens_x^\E$ at
$x$ within $\E$ satisfying
\begin{equation} \label{app-s26-1}
   \sI{x}{\frac{|f(y)-\alpha-L(y-x)|}{|y-x|}}{\dens^E_x} = 0 \,.
\end{equation}
In the case of essential differentiability of $f$ at $x$ with $\alpha$, 
we have that
\begin{equation} \label{app-s26-2}
  \alpha = \esslim_{y\to x} f(y) = \alim_{y\to x} f(y) = 
  \pr f(x)= f^\star(x)=\sI{x}{f}{\dens_x}\,
\end{equation}
for any $\dens_x\in\Dens_x$, 
that $f$ is also approximately differentiable at $x$ with 
\begin{equation*}
   D_{\rm ap}f(x) = D_{\rm es}f(x) \,,
\end{equation*}
and that there is some $\de>0$ such that
\begin{equation} \label{app-s26b-3}
  f\in\cL^\infty(B_\de(x)\cap\dom) \qmq{and}
  \pr f \zmz{exists as integral on} B_\de(x)\cap\cdom\,.
\end{equation}
\end{theorem}
\noi
Notice that Theorem~\ref{app-s25} provides a sufficient condition for
uniqueness of the essential derivative. In \reff{app-s26b-3} 
the value of $\pr f(y)$ may depend on the
special choice of the measure $\dens_y^\dom\in\LDens_y^\dom$ used in 
the definition of $\pr f$ (cf. \reff{app-s1-6}). 
Here \reff{app-s26b-3} means, that there is at least one 
selection of $\dens_y^\dom$ such that $\pr f(y)$ can be computed by an
integral. But notice that $\pr f(x)$ is independent of $\dens_x^\dom$ by
\reff{app-s26-2}. We readily see, that  
the calculus rules in Proposition~\ref{app-s25b} are also valid 
with the essential derivative if it exists. 

\begin{proof}
The equivalence of \reff{eder} and \reff{app-s26-1} is a direct
consequence of Proposition~\ref{dm-s11}.
Let now $f$ be essentially differentiable at $x$ with $\alpha$. 
By \reff{dm-s9-1}, definition \reff{eder} with $\alpha$ and $L=D_{\rm es}f(x)$
implies \reff{ader} with $\alpha$ and $L=D_{\rm es}f(x)$. 
Thus $f$ is approximately differentiable with $D_{\rm ap}f(x) = D_{\rm es}f(x)$
and $\alpha$. For \reff{app-s26-2} we first use \reff{app-s25a-1}.
Then, for $\alpha = \esslim_{y\to x} f(y)$,  we argue as in the proof of
\reff{app-s25a-1} with $\dens_x^\E$ instead of $\dens_x^\dom$ to get 
\begin{equation*}
  \sI{x}{|f-\alpha|}{\dens_x^\E}=0
\end{equation*}
for all $\dens_x^\E$ used in \reff{app-s26-1}. Then 
$\esslim_{y\to x}|f-\alpha|=0$ by Proposition~\ref{dm-s11}, which gives the
desired equality. Again by Proposition~\ref{dm-s11}
we obtain $\alpha=\sI{x}{f}{\dens_x}$ for all
$\dens_x\in\Dens_x$. With \reff{dm-s11-5} we verify \reff{app-s26-2}. 
The local boundedness follows from \reff{ess-bd}. Consequently, 
$f$ is integrable with respect to any density $\dens_y\in\Dens_y$
at $y\in B_r(x)\cap\dom$. 
Hence $\pr f(y)$ exists for these $y$ (with any fixed
selection of $\dens_y$).  
\end{proof}

Now we treat the relation to classical differentiability. 
Due to the pointwise definition we have to consider  
functions or representatives $\ti f\in L^1_{\rm loc}(\dom,\R^m)$
for $\dom\in\cB(\R^n)$. We say that $\ti f$ 
is {\it differentiable} at $x\in\cl\dom$ if there are some $\alpha\in\R^m$ and 
a linear map $L:\R^n\to\R^m$ such that 
\begin{equation}\label{app-s27-2}
  \lim_{\substack{y\to x\\y\in\dom}} \frac{\ti f(y)-\alpha-L(y-x)}{|y-x|} = 0\,
\end{equation}
and we call $D\ti f(x):=L$ {\it derivative} of $\ti f$ at $x$ (with $\al$).
Here it is a simple consequence that
\begin{equation}\label{app-s27-2a}
  \alpha=\lim_{y\to x} \ti f(y) \,
\end{equation}
(which must not equal $\ti f(x)$). For a class in 
$f\in\cL^1_{\rm loc}(\dom,\R^m)$, differentiability obviously depends on the
representative $\ti f$ chosen. Let us start with a first observation. 

\begin{proposition} \label{app-s27a}
Let $\dom\in\cB(\dom)$, let $x\in\cdom$, let 
$f\in\cL^1_{\rm loc}(\dom,\R^m)$, and assume that there is a  
representative $\ti f$ of $f$ that is differentiable at $x$ with $\al$. Then 
$f$ is essentially differentiable at $x$ with $\al$ such that
\begin{equation*}
  \al=\esslim_{y\to x} f(y) = \pr f(x)=f^\star(x) =\sI{x}{f}{\dens_x} \qmq{and}
  D_{\rm es}f(x)=D\ti f(x)\,
\end{equation*}
for all $\dens_x\in\Dens_x$. Moreover, there is some $\de>0$ such that
\begin{equation*}
  \pr f  \qmq{exists as integral on} B_\de(x)\cap\cdom\,.
\end{equation*}
\end{proposition}
\noi
Recall the explaination after Theorem~\ref{app-s26} for the precise meaning 
of the last statement.

\begin{proof}
Since any representative of $f$ agrees with $\ti f$ $\lem$-a.e. on $\dom$, 
with $\alpha=\lim_{y\to x} \ti f(y)$ we get that
\begin{eqnarray*}
  \sx\, \frac{|f(y)-\alpha-D\ti f(x)(y-x)|}{|y-x|} 
&=&
  \lim_{\delta\downarrow 0}\, 
  \essup{B_\de(x)\cap\dom}{\,\frac{|\ti f(y)-\alpha-D\ti f(x)(y-x)|}{|y-x|}} \\
&\le&
  \lim_{\delta\downarrow 0}\, 
  \sup_{B_\de\cap\dom} \,\frac{|\ti f(y)-\alpha-D\ti f(x)(y-x)|}{|y-x|} \\
&=&
  \lim_{\substack{y\to x\\y\in\dom}}
  \frac{|\ti f(y)-\alpha-D\ti f(x)(y-x)|}{|y-x|} 
\:=\: 0\,.
\end{eqnarray*}
Thus, by definition,  
\begin{equation*}
  \esslim_{y\to x} \frac{|f(y)-\alpha-D\ti f(x)(y-x)|}{|y-x|} = 0
\end{equation*}
and $f$ is essentially differentiable at $x$ with \reff{app-s26-2} and
$D_{\rm es}f(x)=D\ti f(x)$. The existence of $\pr f$ as integral 
on some neighborhood of $x$ follows from \reff{app-s26b-3}.
\end{proof}

The fact, that $\pr f$ exists as integral on a neighborhood of $x$ if some
representative is differentiable there, motivates to consider differentiability
for classes $f$ by means of $\pr f$. 
In that case we have to restrict the limit in \reff{app-s27-2} to
$y\in\cdom$. We say that $f\in\cL^1_{\rm loc}(\dom,\R^m)$
with $\dom\in\cB(\R^n)$ is {\it (precisely) differentiable} at $x\in\cdom$
if $\pr f$ exists as integral on $B_\de(x)\cap\cdom$ for some $\de>0$ and if 
there are some $\alpha\in\R^m$ and a linear map $L:\R^n\to\R^m$ such that 
\begin{equation}\label{app-s27-3}
  \lim_{\substack{y\to x\\\:y\in\dom\cap\cdom}} 
  \frac{\pr f(y)-\alpha-L(y-x)}{|y-x|} = 0\,
\end{equation}
and we call $Df(x):=L$ {\it (precise) derivative} of $f$ at $x$ (with $\al$).
As in \reff{app-s27-2a} we get
\begin{equation}\label{app-s27-3a}
  \alpha=\lim_{y\to x} \pr f(y) \,.
\end{equation}
Here \reff{app-s27-3} and \reff{app-s27-3a} have to be satisfied 
with $\pr f(y)=\sI{y}{f}{\dens_y^\dom}$ 
for some selection of measures $\dens_y\in\LDens_y^\dom$ such that
$\sI{y}{f}{\dens_y^\dom}$ exists. Notice that this definition extends 
the classical pointwise definition of differentiability to classes $f$
of functions by means of a pointwise condition for the precise representative
$\pr f$. 
Now we show that essential differentiability of (a class) $f$ 
is equivalent to its (precise) differentiability. 

\begin{theorem}\label{app-s27}
Let $\dom\in\cB(\dom)$, let $x\in\cdom$, and let 
$f\in\cL^1_{\rm loc}(\dom,\R^m)$. 
\bgl
\item
We have that $f$ is essentially differentiable at $x$ with $\al$ if and only
if $f$ is (precisely) differentiable at $x$ with $\al$. In this case we have
\begin{equation*}
  Df(x) =D_{\rm es}f(x)
\end{equation*}
and
\begin{equation} \label{app-s27-1}
  \al = \esslim_{y\to x} f(x) = \pr f(x)=f^\star(x) =\sI{x}{f}{\dens_x}
  \:\zmz{for all}\: \dens_x\in\Dens_x . 
\end{equation}    

\item
If $f$ has a representative $\ti f$ that is differentiable at $x$ with $\al$, 
then $f$ is (precisely) differentiable at $x$ with $\al$ such that
$Df(x)=D\ti f(x)$ and $\al$ satisfies \reff{app-s27-1}.
\el
\end{theorem}

\begin{proof}
For (1) we first assume that $f$ is essentially differentiable at $x$
and we assume that \reff{app-s27-3} is wrong for $\al=\pr f(x)$ and 
$L=D_{\rm es}f(x)$. Then there is some $\ep>0$ and 
there are $y_k\in\dom\cap\cdom$ with $y_k\to x$ and $y_k\ne x$ such that 
\begin{equation}\label{app-s27-5}
  \frac{\big|\pr f(y_k)-\pr f(x)-D_{\rm es}f(x)(y_k-x)\big|}{|y_k-x|} > 8\eps
  \qmq{for all} k\in\N\,.
\end{equation}
Since $f$ is essentially differentiable at $x$, the definition of the essential
limit provides some $\rho>0$ and
some $\tilde\dom\subset\dom$ with $\lem(\dom\setminus\tilde\dom)=0$ 
such that
\begin{equation*}
  \frac{\big|f(y)-\pr f(x)-D_{\rm es}f(x)(y-x)\big|}{|y-x|} < \eps
  \qmq{for all} y\in B_{2\rho}(x)\cap\tilde\dom\,.
\end{equation*}
By \reff{app-s26b-3} we can assume that $\rh>0$ is so small, that $\pr f$
exists as integral on $B_\rh(x)\cap\cdom$. 
Let us fix some $y_k\in B_\rho(x)$. Then $y_k\in\cdom\setminus\ti\dom$
and, by the boundedness of the linear operator $D_{\rm es}f(x)$,
there is some $\delta_k\in(0,\rho)$ such that
\begin{equation*}
  |y'-x|<2|y_k-x| \,, \z \big|D_{\rm es}f(x)(y'-y'')\big|<\eps|y_k-x|
  \qmq{for all} y',y''\in B_{\delta_k}(y_k)\,.
\end{equation*}
Since $B_{\delta_k}(y_k)\subs B_{2\rho}(x)$, we obtain
for $y',y''\in B_{\delta_k}(y_k)\cap\tilde\dom$ that
\begin{eqnarray*}
  \frac{\big|f(y')-f(y'')\big|}{|y_k-x|} 
&\le& 
  2\frac{\big|f(y')-\pr f(x)-D_{\rm es}f(x)(y'-x)\big|}{|y'-x|} \\
&&
+\;2\frac{\big|f(y'')-\pr f(x)-D_{\rm es}f(x)(y''-x)\big|}{|y''-x|} \;+\;
\frac{|D_{\rm es}f(x)(y'-y'')|}{|y_k-x|} \\
&\le&
5\eps \,.
\end{eqnarray*}
By $y_k\in B_\rh(x)$, there is some $\dens_{y_k}^\dom\in\LDens_{y_k}^\dom$
such that $\pr f(y_k)=\sI{y_k}{f}{\dens_{y_k}^\dom}$.
Then we obtain for any  
$y\in B_{\delta_k}(y_k)\cap\tilde\dom$ that
\begin{eqnarray*}
  |\pr f(y_k)-f(y)| 
&=& 
  \Big|\: \sI{y_k}{f(y')- f(y)}{\dens_{y_k}^\dom(y')}  \Big| \\
&\le&
\sI{y_k}{|f(y')- f(y)|}{\dens_{y_k}^\dom(y')} 
\:\le\:  5\eps\, |y_k-x| \,
\end{eqnarray*}
(notice that we can disregard $y'$ from the $\lem$-null set
$\dom\setminus\ti\dom$).
With some fixed $y\in B_{\delta_k}(y_k)\cap\tilde\dom$ we get
\begin{eqnarray*}
  \frac{\big|\pr f(y_k)-\pr f(x)-D_{\rm es}f(x)(y_k-x)\big|}{|y_k-x|} 
&\le& 
\frac{|\pr f(y_k)- f(y)|}{|y_k-x|} + 
\frac{|D_{\rm es}f(x)(y_k-y)|}{|y_k-x|}\\
&& 
+\: 2\:\frac{\big|f(y)-\pr f(x)-D_{\rm es}f(x)(y-x)\big|}{|y-x|} \\
&\le& 8\eps \,.
\end{eqnarray*}
But this contradicts \reff{app-s27-5} and implies \reff{app-s27-3}
with $\al=\pr f(x)$ and $L=D_{\rm es}f(x)$.
Hence $f$ is differentiable at $x$ as stated if we take into account  
\reff{app-s26-2}. 

Let now $f$ be differentiable at $x$ with $\al$. Then we argue as in the
proof of Proposition~\ref{app-s27a} with $\pr f$ instead of $\ti f$. Since
$\pr f$ is only defined on $\cdom$, we have to use that 
\begin{equation*}
  \big(B_\de(x)\cap\dom\big) \setminus \big(B_\de(x)\cap\cdom\big)
\end{equation*}
is an $\lem$-null set (cf. \reff{cdom1}). 

For (2) we first apply Proposition~\ref{app-s27a} to get that $f$ is
essentially differentiable at $x$ with $\al$ such that
$D_{\rm es}f(x)=D\ti f(x)$ and $\al$ satisfies \reff{app-s27-1}. Then the
assertion follows from the previous statement~(1).  
\end{proof}

For Sobolev functions $f\in W^{1,p}(\dom)$ with $p>n$ it is known 
that they are differentiable at $\lem$-a.e. $x\in\dom$ 
and that the derivative $Df(x)$ agrees with the weak derivative 
$D_wf(x)$ $\lem$-a.e. on $\dom$ (cf. \cite[p. 235]{evans}). For a class
$f\in\cW^{1,p}(\dom)$ there is a continuous representative that coincides with 
$\pr f$ and $f^\star$. Let us improve the statement from
Remark~\ref{app-s25c} for such functions.

\begin{proposition} \label{app-s27c}
Let $\dom\subs\R^n$ be open, let $f\in\cW^{1,p}_{\rm loc}(\dom,\R^m)$ with
$n<p\le\infty$, and let
\begin{equation*}
  \ti\dom = \big\{ x\in\dom \:\big|\: \tx{there is a representative $\ti f$ of
  $f$ that is differentiable at $x$}\big\}\,. 
\end{equation*}
Then $\lem(\dom\setminus\ti\dom)=0$ and $f$ is (precisely) differentiable at
any $x\in\ti\dom$ with 
\begin{equation}  \label{app-s27c-1}
  \pr f(x)=f^\star(x)=  \esslim_{y\to x} f(y) = \sI{x}{f}{\dens_x}
  \qmq{for all} \dens_x\in\Dens_x\,
\end{equation}  
and 
\begin{equation}\label{app-s27c-2}
  Df(x) =  (Df)^\star(x) \,.
\end{equation}
For $\lem$-a.e. $x\in\ti\dom$ one has 
\begin{equation}\label{app-s27c-3}
  Df(x) =  
  \pr{(Df)}(x) = \sI{x}{Df}{\dens_x^\dom}
  \qmq{for all}  \dens_x^\dom\in\LDens_x^\dom\,.
\end{equation}
\reff{app-s27c-3} is in particular valid for $x\in\ti\dom$, if
$Df\in\cL^\infty(B_{\ti\de}(x),\R^n)$ for some
$\ti\de>0$.
\end{proposition}
\noi
Notice that, in analogy to $\pr f(x)$ and $f^\star(x)$, 
the derivatives $D_{\rm ap}f(x)$, $D_{\rm es}f(x)$, and $Df(x)$ 
have a precise pointwise definition for (classes) 
$f\in\cL^1_{\rm loc}(\dom,\R^m)$ 
as long as they exist, while the weak derivative $D_wf$ is just an
$\cL^1$-function. Nevertheless, the derivatives $D_{\rm ap}f$, $D_{\rm es}f$, 
and $Df$ can be considered as $\cL^1$-function if they are integrable, which
is the case locally for $f\in\cW^{1,p}_{\rm loc}(\dom,\R^m)$ with $p\ge 1$.
Let us mention that \reff{app-s27c-3} is satisfied at Lebesgue points of
$Df$ (cf. \reff{app-e00}). 

\begin{proof}
Let the representative $\ti f$ of $f$ be differentiable at $x\in\dom$.
Then, by Theorem~\ref{app-s27},  
$f$ is differentiable at $x$ with $Df(x)=D\ti f(x)$ 
and \reff{app-s27-1} holds, which gives \reff{app-s27c-1}.
Recall that any representative $\ti f$ of $f$ is differentiable 
$\lem$-a.e. on $\dom$ and that its derivative $D\ti f$ agrees with the weak
derivative $D_wf$ $\lem$-a.e. on $\dom$ (cf. \cite[p.~235]{evans}).
Therefore $\lem(\dom\setminus\ti\dom)=0$ and $Df=D_wf$ $\lem$-a.e. on $\dom$.  
Hence $Df$ is a representative of $D_wf$ and locally an
$\cL^1$-function. 
By \reff{app-e2} with $Df$ instead of $f$, we get that 
$\lem$-a.e. $x\in\ti\dom$ is a Lebesgue point of $Df$ with $\al=Df(x)$.
Hence we obtain \reff{app-s27c-3} by \reff{app-e00}. If 
$Df\in\cL^\infty(B_{\ti\de}(x),\R^n)$ for $x\in\ti\dom$ and some $\ti\de>0$ and
if $\dens_x^\dom\in\LDens_x^\dom$, 
then \reff{dm-s2-3} implies \reff{app-s27c-3} as long as 
\reff{app-s27c-2} is satisfied, which remains to be shown.

For \reff{app-s27c-2} we fix $x\in\ti\dom$. It is sufficient to consider the
scalar case $m=1$ and 
for the components of vectors in $\R^n$ we use indices $i,j\in\{1,\dots,n\}$.
The divergence theorem gives for $\de>0$, $c\in\R$, and the outer unit normal 
$\nu=(\nu_1,\dots,\nu_n)$ that
\begin{equation} \label{app-s27c-5}
\begin{split}
  \I{\bd B_\delta(x)}{c\nu_j}{\cH^{n-1}(y)} &=0\,, \quad  \\
  \I{\bd B_\delta(x)}{c(y_i-x_i)\nu_j}{\cH^{n-1}(y)} &= \bigg\{
  \begin{array}{ll}
    0 & \tx{for } i\ne j\,,  \\
    c|B_\delta(x)| & \tx{for } i=j\,.
  \end{array}
\end{split}
\end{equation}
Since $\pr f$ is differentiable at $x$, \reff{app-s27-3} implies
\begin{equation*}
  \pr f(y) = \pr f(x) + Df(x)(y-x) + o(|y-x|) \,.
\end{equation*}
Recall that $Df=(f_{x_1},\dots,f_{x_n})$ agrees $\lem$-a.e. with the weak
derivative $D_wf$ of $f$. 
Then the divergence theorem, which is valid for the Sobolev function $f$ 
with their weak derivative $D_wf$ and the continuous trace $\pr f$ on 
$\bd B_\de(x)$ (cf. \cite[p.~133]{evans}, \cite[p.~168]{pfeffer}), implies 
\begin{eqnarray*}
  \I{B_\delta(x)}{f_{x_j}(y)}{\lem(y)}
&=&
  \I{\bd B_\delta(x)}{\pr f(y)\nu_j(y)}{\cH^{n-1}(y)} \\
&=&
  \I{\bd B_\delta(x)}
  {\big(\pr f(x)+Df(x)(y-x)+o(\de)\big)\nu_j(y)}{\cH^{n-1}(y)} \\
&\overset{\reff{app-s27c-5}}{=}&
  f_{x_j}(x)|B_\delta(x)| + o(\de) \I{\bd B_\delta(x)}{}{\cH^{n-1}} \,.
\end{eqnarray*}
Dividing by $|B_\delta(x)|$ and taking the limit $\de\dto 0$ we obtain
\begin{equation*}
  f_{x_j}(x) = \lim_{\delta\dto 0}\:\mI{B_\delta(x)}{f_{x_j}}{\lem} 
  \qmq{for all} j=1,\dots,n\,,
\end{equation*}
which verifies \reff{app-s27c-2} and completes the proof. 
\end{proof}

Let us now provide some kind of mean value theorem for certain Sobolev functions
$f$ in $\cW^{1,p}(\dom,\R^m)$ with $p>n$. Here $Df$ is the (precise)
derivative of 
$f$, that agrees with the weak derivative $D_wf$ $\lem$-a.e. on $\dom$, 
and we identify $f$ with its continuous representative~$\pr f$. 
By $[x,y]$ we denote the closed line segment connecting the points $x$ and 
$y$.

\begin{theorem}\label{app-s30}
{\rm (Mean value theorem)}
Let $\dom\subset\R^n$ be open, let $x,y\in\dom$ with $x\ne y$ and 
$[x,y]\subs\dom$, and let $f\in\cW^{1,p}_{\rm loc}(\dom,\R^m)$
for $n<p\le\infty$. If there is some $\ti\de>0$ with
\begin{equation*}
  Df\:(y-x) \in \cL^\infty([x,y]_{\ti\de},\R^m) \,,
\end{equation*}
then we have for all 
$\lem$-densities $\dens_{[x,y]}^\dom\in\LDens_{[x,y]}^\dom$ that
\begin{equation*}
  f(y)-f(x) = \sI{[x,y]}{Df}{\dens_{[x,y]}^\dom} \:(y-x) \,.
\end{equation*}
\end{theorem}
\noi
Notice that the assumptions are satisfied for 
$f\in\cW^{1,\infty}_{\rm loc}(\dom)$. 

\begin{proof}
Let $\di_\de(x)\subs\R^n$ be the $(n-1)$-dimensional disc of radius $\de>0$
that is centered at~$x$ and orthogonal to $v:=y-x$.
By the continuity of $f$ we have that
\begin{eqnarray*}
  f(y)-f(x)
&=&
 \lim_{\delta\downarrow 0}
 \Big(\:\mI{\di_\delta(y)}{f}{\hm} - \mI{\di_\delta(x)}{f}{\hm}\Big)\\
&=& 
  \lim_{\delta\downarrow 0} \mI{\di_\delta(x)}{f(z+v)-f(z) }{\hm(z)} \,.
\end{eqnarray*}
Since $f$ is absolutely continuous on $\cH^{n-1}$-a.e. line parallel to $v$,
we get
\begin{eqnarray*}
&& \hspace{-20mm}
  \mI{\di_\delta(x)}{f(z+v)-f(z) }{\hm(z)} \\
&=&
  \mI{\di_\delta(x)}{\int_0^1 Df(z+tv)\: v \,dt\,}{\hm(z)} \\
&=&
  \tfrac{1}{|v|\hm(\di_\de)}   \I{\di_\delta(x)}
  {\int_0^{|v|} Df\big(z+t\tfrac{v}{|v|}\big)\: v\,dt\,}{\hm(z)} \,.
\end{eqnarray*}
Let 
\begin{equation*}
  \zy_\de:=\big\{z\in\di_\de(\ti z) \:\big|\: \ti z\in[x,y] \big\}
\end{equation*}
be the rotational cylinder with axis $[x,y]$, height $|v|$, and
bounded by $\di_\de(x)$ and $\di_\de(y)$. Then
Fubini's theorem gives that
\begin{equation*}
  \mI{\di_\delta(x)}{f(z+v)-f(z) }{\hm(z)} =
  \mI{\zy_\de}{Df\: v}{\lem} \,.
\end{equation*}
With the $\de$-neighborhood $[x,y]_\de$ of $[x,y]$, we have
\begin{eqnarray*}
  \mI{[x,y]_\de}{Df\: v}{\lem} 
&=&
  \tfrac{\lem(\zy_\de)}{\lem([x,y]_\de)} \:
  \mI{\zy_\de}{Df\: v}{\lem} + 
  \tfrac{\lem([x,y]_\de\setminus\zy_\de)}{\lem([x,y]_\de)}\:
  \mI{[x,y]_\de\setminus\zy_\de}{Df\: v}{\lem} \,. 
\end{eqnarray*}
We readily obtain that 
\begin{equation*}
  \lim_{\de\downarrow 0}\frac{\lem(\zy_\de)}{\lem([x,y]_\de)} = 1 \qmq{and}
  \lim_{\de\downarrow 0} 
  \frac{\lem([x,y]_\de\setminus\zy_\de)}{\lem([x,y]_\de)} = 0 \,.
\end{equation*}
Thus, by $Df\: v\in\cL^\infty([x,y]_{\ti\de},\R^m)$,
\begin{equation*}
  f(y)-f(x) =   
  \lim_{\de\dto 0}\:\mI{\zy_\de}{Df\: v}{\lem} =
  \lim_{\de\dto 0}\:\mI{[x,y]_\de}{Df\: v}{\lem} \,. 
\end{equation*}
Using \reff{dm-s2-3} for each component, 
we get for any $\lem$-density
$\dens_{[x,y]}^\dom\in\LDens_{[x,y]}^\dom$ that
\begin{equation*}
  f(y)-f(x) = \sI{[x,y]}{Df\: v}{\dens_{[x,y]}^\dom} 
\end{equation*}
which verifies the assertion. 
\end{proof}

\subsection{Generalized derivatives}
\label{gd}

We provide a new approach to the set-valued derivatives in the sense of Clarke
in finite dimensions. We assume that $\dom\subset\R^n$ is open and that
$F:\dom\to\R^m$ is locally Lipschitz continuous. 
This is equivalent to $F\in\cW_{\rm loc}^{1,\infty}(\dom,\R^m)$ where we
identify $F$ with its continuous representative $\pr F$. 
By Rademacher's theorem, $F$ is differentiable $\lem$-a.e. 
in the sense of \reff{app-s27-2} or \reff{app-s27-3}
and the derivative $DF$ agrees 
with the weak derivative $\lem$-a.e. on $\dom$ 
(cf. \cite[p. 81, 131, 235]{evans}). 
Then $\lem(\dom\setminus D_F)=0$ for 
\begin{equation*}
  D_F := \big\{ x\in\dom \:\big|\: F \tx{ is differentiable at } x \big\}\,.
\end{equation*}
We define the {\it generalized derivative} of $F$ at $x\in\R^n$ by  
\begin{equation} \label{gd-e1}
  \partial F(x) := \df{\Dens_x}{DF}\,
\end{equation}
which is a well-defined bounded set in $\R^{m\times n}$. We show below
that it coincides with Clarke's generalized Jacobian.
The support function of $\partial F(x)$ is given by
\begin{equation} \label{gd-e2}
  F^\circ(x;\sfV) := \max_{\sfF\in\partial F(x)} \sfF:\sfV 
  \qmq{for all} \sfV\in\R^{m\times n}
\end{equation}
where $\sfF:\sfV$ denotes the scalar product of matrices (cf. \reff{dm-supp}). 
Let us formulate some basic properties.

\begin{proposition}\label{gd-s0}
Let $\dom\subset\R^n$ be open, let $F:\dom\to\R^m$ be locally Lipschitz
continuous, let $x\in\dom$, and let 
$\ell\ge 0$ be a Lipschitz constant of $F$ near $x$. Then:
\bgl
\item The generalized derivative 
$\partial F(x)\subset\R^{m\times n}$ is compact, convex and bounded by 
\begin{equation} \label{gd-e1a}
  |\sfF| \le \ell \qmq{for all} \sfF\in\partial F(x)\,.
\end{equation}

\item
$F^\circ(x;\cdot)$ is positively \mbox{1-homogeneous}, subadditive, and
satisfies 
\begin{equation*}
  |F^\circ(x;\sfV)| \le  \ell|\sfV| \qmq{for all} \sfV\in\R^{m\times n}
\end{equation*}
(where $|\sfV|$ is the Euclidean norm) 
and, thus, it is convex and continuous. 

\item 
$\partial F$ is closed at $x$, i.e. 
\begin{equation*}
  x_j\to x,\; \sfF_j\in\partial F(x_j),\; \sfF_j\to\sfF
  \qmq{implies} \sfF\in\partial F(x)\,.
\end{equation*}
\item
$\partial F$ is upper semicontinuous at $x\in\R^n$, i.e. for all $\eps>0$
there is $\delta>0$ such that
\begin{equation*}
  \partial F(y) \subset B_\eps(\partial F(x)) \qmq{for all}
  y\in B_\delta (x)\,.
\end{equation*}
\el
\end{proposition}

\begin{proof}
For (1) we use Proposition~\ref{dm-s7} to get that $\pa F(x)$ 
is compact and convex. The estimate follows from 
$|DF(y)|\le\ell$ near $x$.

For (2) we notice that a support function is always 
positively 1-homogeneous and subadditive 
(cf. \cite[p. 29]{clarke}). Thus it is convex and the estimate follows from
(1). As bounded convex function it is continuous.  

(3) is a direct consequence of (4), which remains to be shown. 
If (4) were wrong, there are some $\ep>0$ and sequences 
$x_k\to x$ and $\sfF_k\in\pa F(x_k)\setminus B_\eps(\partial F(x))$.
By definition we have
\begin{equation*}
  \sfF_k=\sI{x_k}{DF}{\mu_k} \qmq{for some} \mu_k\in\Dens_{x_k}\,.
\end{equation*}
The $\mu_k$ have a weak$^*$ accumulation point $\mu\in\Dens_x$ 
by Proposition~\ref{dm-s4f}. Moreover, there is a subsequence, denoted the
same way, such that
\begin{equation*}
  \sfF_{k}=\sI{x_{k}}{DF}{\mu_{k}} \to 
  \sfF:=\sI{x}{DF}{\mu}=\df{\mu}{DF}\in\pa F(x)\,.
\end{equation*}
But this is a contradiction and verifies (4).
\end{proof}

We now show some characterization of the generalized derivative.

\begin{theorem} \label{gd-s2}
Let $\dom\subs\R^n$ be open, let 
$F:\dom\to\R^m$ be locally Lipschitz continuous, and let $x\in\R^n$. Then,
for any $\lem$-null set $N\subs\dom$,
\begin{eqnarray} 
  \partial F(x) 
&=& \label{gd-s2-1} 
  \op{conv}\big\{\lim_{j\to\infty} DF(x_j)\:\big|\:
  x_j\to x\,,\;x_j\in D_F\setminus N \big\}  \\
&=& \label{gd-s2-2} 
  \big\{ \sfF\in\R^{m\times n}\:\big|\:  \sfF:\sfV\le\suf Fx\sfV
  \zmz{for all} \sfV\in\R^{m\times n}\big\} \,.
\end{eqnarray}
\end{theorem}
\noi
The set on the right hand side in \reff{gd-s2-1} is the 
{\it generalized Jacobian} of $F$ at $x$ in the sense of Clarke. 
It is remarkable that it is independent of $\lem$-null sets (cf. 
\cite[p.~70]{clarke}, \cite[p.~282]{clarke-fa}). In the scalar case $m=1$ it
agrees with Clarke's {\it generalized gradient} 
(cf. \cite[p.~27,~63]{clarke}). Notice that the set of limits in
\reff{gd-s2-1} is bounded and closed and, thus, compact. Hence the convex hull
coincides with its closed convex hull (cf. \cite[p.~158]{rockafellar}). 

\begin{proof}
Since the closed convex set $\pa F(x)$ is uniquely defined by its support
function $\suf Fx\sfV$, the equality in \reff{gd-s2-2} is standard
(cf. \cite[p.~29]{clarke}, \cite[p.~70]{clarke-fa}).

Let $\cD_x^NF$ be the set of limits on the right hand side in \reff{gd-s2-1}. 
We show that its convex hull agrees with $\df{\Dens_x}{DF}$. If
$\mu_x\in\ZO_x$, then $\sfF=\sI{x}{DF}{\mu_x}$ is essential value
of $DF$ at $x$ and, hence,
\begin{equation*}
   \lem\big(\{|DF-\sfF|<\eps\}\cap B_\de(x)\big) >0 \qmq{for all} \eps,\de>0\,
\end{equation*}
(cf. Proposition~\ref{ep-s20}). Therefore, we can find 
\begin{equation*}
  x_k\in \big(\{|DF-\sfF|<\tfrac1k\}\cap B_{\frac1k}(x)\big)\setminus N
 \qmq{for all} k\in\N\,.
\end{equation*}
Thus $x_k\to x$ and $DF(x_k)\to\sfF$. Consequently, 
$\df{\ZO_x}{DF}\subset \cD_x^NF$.
Since $\df{\Dens_x}{DF}$ is compact and convex by Proposition~\ref{dm-s7} 
and since $\ZO_x$ are the extreme points of $\Dens_x$ by Theorem~\ref{ep-s2}, 
we can take the closed convex hull to get 
$\df{\Dens_x}{DF}\subset\op{conv}\cD_x^NF$. 

Let now $\sfF\in\cD_x^NF$. Then $DF(x_k)\to \sfF$ for some sequence $x_k\to x$.
Recall that $DF$ agrees $\lem$-a.e. with the weak derivative $D_wF$ of $F$ 
considered as $\cW^{1,\infty}_{\rm loc}$-function and that we have identified
$F$ with its continuous representative $\pr F$. Then, by
Proposition~\ref{app-s27c} there are $\lem$-densities
$\dens_{x_k}^\dom\in\LDens_{x_k}^\dom$ and by Proposition~\ref{dm-s4f} 
there is a weak$^*$ accumulation point $\mu\in\Dens_x$ and a subsequence,
denoted the same way, such that
\begin{equation*}
  DF(x_{k}) = \sI{x_{k}}{DF}{\dens_{x_{k}}^\dom} \to\, 
  \sI{x}{DF}{\me}=\sfF \,.
\end{equation*}
Hence $\sfF\in\df{\Dens_x}{DF}$ and, thus, also  
$\op{conv}\cD_x^NF \subs\df{\Dens_x}{DF}$.
\end{proof}

We get also an representation for $\suf Fx\sfV$.

\begin{proposition} \label{gd-s1}
Let $\dom\subs\R^n$ be open and let 
$F:\dom\to\R^m$ be locally Lipschitz continuous. Then
we have for any $x\in\R^n$ and any $\lem$-null set $N\subs\dom$ that 
\begin{equation} \label{gd-s1-1}
  \suf Fx\sfV = \tesup{x}\: DF:\sfV 
  = \limsup_{\substack{y\to x\\y\in D_F\setminus N}} \,DF(y):\sfV
  \qmq{for all} \sfV\in\R^{m\times n} \,.
\end{equation}
For a scalar locally Lipschitz continuous function 
$f:\dom\to\R$ we also have 
\begin{equation} \label{gd-s1-2}
  f^\circ(x;v)  = 
  \limsup_{\substack{y\to x\\t\downarrow 0}} \frac{f(y+tv)-f(y)}{t} 
  \qmq{for all} v\in\R^n \,.
\end{equation}
\end{proposition}
\noi
The right hand side in \reff{gd-s1-2} is the {\it generalized directional
derivative} of $f$ at $x$ in direction~$v$ in the sense of Clarke
(cf. \cite[p.~25]{clarke}). Notice that the limsup in
\reff{gd-s1-1} is independent of an $\lem$-null set.

\begin{proof}
Recall that $\suf Fx\cdot$ is the support function of $\pa F(x)$
(cf. \reff{gd-e2}). Then \reff{gd-e1} and Proposition~\ref{dm-s7} 
imply the first equality in \reff{gd-s1-1}. For the second equality 
we denote  the set of limits in \reff{gd-s2-1} by 
$\cD^N_xF\subset\R^{m\times n}$ (with $N$ from \reff{gd-s1-1}). 
By Theorem~\ref{gd-s2},
the support functions of $\partial F(x)$ and $\cD^N_xF$ have to 
agree.   
We readily get that $\cD^N_xF$   
is compact and, by the definition of its support function, we thus obtain  
\begin{equation} \label{gd-e2a}
  \suf Fx\sfV = \max_{\sfF\in \cD^N_xF} \sfF:\sfV 
  \qmq{for all} \sfV\in\R^{m\times n}\,.
\end{equation}
Hence, by definition of $\cD^N_xF$,
\begin{equation*}
  \suf Fx\sfV = \limsup_{\substack{y\to x\\y\in D_F\setminus N}} \,DF(y):\sfV 
  \qmq{for all} \sfV\in\R^{m\times n}\,,
\end{equation*}
which verifies the second equality in \reff{gd-s1-1}.

In the scalar case we fix $v\in\R^n$. Then there are sequences 
$x_k\to x$, $t_k\downarrow 0$ and, by Theorem~\ref{app-s30},
density measures $\mu_{y,t}\in\Dens_{[y,y+tv]}$ such that
\begin{eqnarray*}
  f^\circ(x;v)  
&\overunderset{\reff{gd-s1-1}}{N=\emptyset}{=}&
  \limsup_{\substack{y\to x\\y\in D_F}} \, Df\cdot v
\;=\; 
  \lim_{k\to\infty} Df(x_k)\cdot v \\
&=&
  \lim_{k\to\infty} \frac{f(x_k+t_kv)-f(x_k)}{t_k}  
\;\le\;
  \limsup_{\substack{y\to x\\t\downarrow 0}} \frac{f(y+tv)-f(y)}{t} \\
&=&
 \lim_{\de\downarrow 0} \sup_{\substack{y\in B_\delta(x)\\0<t<\delta}}   
 \sI{[y,y+tv]}{Df\cdot v}{\mu_{y,t}} 
\;\le\;
  \lim_{\de\downarrow 0} \:\essup{B_{\de(1+|v|)}(x)}{Df\cdot v}    \\
&=&
 \tesup{x}\: (Df\cdot v) 
\;\overset{\reff{gd-s1-1}}{=}\; 
  f^\circ(x;v) \,,        
\end{eqnarray*}
which verifies \reff{gd-s1-2}.
\end{proof}

For $\dom\subs\R^n$ open, we say that $F:\dom\to\R^m$ is 
{\it strictly differentiable} at $x\in\R^n$ if 
there is a linear map $L:\R^n\to\R^m$ such that for all $v\in\R^n$
\begin{equation} \label{gd-e3}
  \lim_{\substack{y\to x\\t\downarrow 0}}
  \frac{F(y+tv)-F(y)}{t} = Lv \,
\end{equation}
and we call $D_sF(x):=L$ {\it strict derivative} of $F$ at $x$.  
Let us mention that continuous differentiability
of $F$ near $x$ is sufficient 
for strict differentiability at $x$ (cf. \cite[p. 32]{clarke}).

\begin{proposition} \label{gd-s3}
Let $\dom\subs\R^n$ be open and let $F:\dom\to\R^m$.  
\bgl
\item
If $F$ is strictly differentiable at $x\in\dom$, then
$F$ is Lipschitz continuous near $x$ and
\begin{equation} \label{gd-s3-1}
  \partial F(x) = \{ D_sF(x) \} = \big\langle \Dens_x, DF \big\rangle 
  = \sI{x}{DF}{\me} \qmq{for all} \me\in\Dens_x\,. 
\end{equation}
Moreover, $DF-D_sF(x)=0$ i.m. $\!\mu$ for any $\mu\in\Dens_x$.

\item
If $F$ is Lipschitz continuous near $x\in\dom$ and $\partial F(x)$ 
is a singleton, then $F$ is strictly differentiable at $x$ and
$\partial F(x) = \{ D_sF(x) \}$.

\item
If $F$ is Lipschitz continuous near $x\in\dom$ and 
differentiable at $x$, then we have that $DF(x)\in\pa F(x)$.

\el
\end{proposition}
\noi
For the scalar case $m=1$ the result can be found in \cite[p. 32, 33]{clarke}).
Though the implication that $F$ is Lipschitz continuous near $x$ is shown in 
\cite[p.~31]{clarke}, we sketch it for completeness. 

\begin{proof}
For (1) we first assume that $F$ is not Lipschitz continuous near $x$. 
Then there are sequences $y_k,z_k$ converging to $x$ such that
\begin{equation}\label{gd-s3-3}
  |F(z_k)-F(y_k)| > k |z_k-y_k| \qmq{and} 
  v_k:=\tfrac{z_k-y_k}{|z_k-y_k|}\to: v\,. 
\end{equation}
Hence $z_k=y_k+t_kv_k$ for $t_k=|z_k-y_k|$.
By \reff{gd-e3} there is some $\rh>0$ such that
\begin{equation}\label{gd-s3-4}
  \Big|\frac{F(y+tv)-F(y)}{t} - D_sF(x)v\Big|<1 
  \qmq{whenever} |y-x|, t \le\rh\,.
\end{equation}
By continuity there is some $\ti k$ such that, for all $k>\ti k$, 
\reff{gd-s3-4} remains valid with $v=v_k$ and we can assume that
$|y_k-x|,t_k\le\rh$. Taking $y=y_k$, $t=t_k$, and $v=v_k$ 
in \reff{gd-s3-4}, we get a contradiction for large $k>\ti k$, 
since the fraction exceeds $k$ by \reff{gd-s3-3}. 
Thus \reff{gd-e3} cannot be true and $F$ is Lipschitz continuous near $x$. 
From \reff{gd-e3} we get
\begin{equation} \label{gd-s3-5}
  D_sF(x)v =
  \lim_{\substack{y\to x\\y\in D_F}} \lim_{t\downarrow 0}\frac{F(y+tv)-F(y)}{t}
  = \lim_{\substack{y\to x\\y\in D_F}} DF(y)v   \qmq{for all} v\in\R^n\,.
\end{equation}
Hence 
\begin{equation}\label{gd-s3-6}
   \lim_{\substack{y\to x\\y\in D_F}} DF(y) = D_sF(x)
\end{equation}
and, thus, $\partial F(x) = \{D_sF(x)\}$ by \reff{gd-s2-1}.
Then \reff{gd-s3-1} follows from \reff{gd-e1}. From \reff{gd-s3-1} we obtain
\begin{equation*}
  \sI{x}{DF-D_sF(x)}{\mu}=0 \qmq{for all} \mu\in\Dens_x .
\end{equation*}
For any fixed $\mu\in\Dens_x$, we have that  
$\frac{1}{\mu(A)}\reme{\mu}{A}$, extended by zero on $\dom$,  
belongs to $\Dens_x$ if $\mu(A)\ne 0$. Hence,
\begin{equation*}
  \sI{A}{DF-D_sF(x)}{\mu}=0 \qmq{for all} A\in\cB(\R^n)\,, 
\end{equation*}
Thus, $DF-D_sF(x)=0$ i.m. $\mu$
by \reff{pl-e5}.  

For (2) we assume that $F$ is Lipschitz continuous near $x$ and that
$\partial f(x)=\{\sfF\}$ is a singleton. Let us fix $v\in\R^n$ and  
sequences $y_k\to x$, $t_k\downarrow 0$. Then, by Theorem~\ref{app-s30}
there are measures $\mu_k\in\LDens_{[y_k,y_k+t_kv]}^\dom$ and by 
Proposition~\ref{dm-s4f} there is a weak$^*$ accumulation point 
$\mu\in\Dens_x$ and a subsequence, denoted the same way, such that
\begin{equation*}
  \frac{F(y_k+t_kv)-F(y_k)}{t_k} = 
  \sI{[y_k,y_k+t_kv]}{DF}{\mu_k}\: v \to 
  \sI{x}{DF}{\mu}\: v \overset{\reff{gd-e1}}{=} \sfF v\,.
\end{equation*}
The arbitrariness of $v$ and of the sequences $\{y_k\}$, $\{t_k\}$
and the subsequence principle imply
strict differentiability of $F$ at $x$ with $D_sF(x)=\sfF$.

For (3) we have by \reff{app-s27c-3} that $DF(x)=\sI{x}{DF}{\mu}$
for any $\mu\in\LDens_x$. This gives the assertion by \reff{gd-e1}.
\end{proof}

Let us recall the mean value formula from \cite[p. 72]{clarke}.
We provide a new proof that shows that this is in fact a specialization of 
Theorem~\ref{app-s30} to locally Lipschitz continuous functions. 

\begin{theorem} \label{gd-s8}
{\rm (Mean value theorem)}
Let $\dom\subs\R^n$ be open, let $F:\dom\to\R^m$ be locally
Lipschitz continuous, and let $[x,y]\subset\dom$. Then
\begin{equation*}
  F(y)-F(x) \in \op{conv} \big(\pa F([x,y])\big) \,(y-x) \,
\end{equation*}
and $\partial F([x,y])$ is compact.
\end{theorem}

\begin{proof}
By Theorem~\ref{app-s30} there is some $\mu\in\Dens_{[x,y]}$ such that
\begin{equation*}
  F(y)-F(x) = \sI{[x,y]}{DF}{\mu}\: (y-x)\,.
\end{equation*}
By Corollary~\ref{ep-s1a}, Theorem~\ref{ep-s3}, and Theorem~\ref{ep-s2}, 
\begin{eqnarray*}
  \Dens_{[x,y]} 
&=& 
  \ol{\op{conv}} \bigcup_{z\in[x,y]} \Dens_z^{\tx{\zo}} 
\;=\;
   \ol{\op{conv}} \bigcup_{z\in[x,y]} \Dens_z \,.
\end{eqnarray*}
Since $\partial F([x,y])$ is compact by Proposition~\ref{gd-s0} (3),
linearity implies
\begin{equation*}
  \df{\Dens_{[x,y]}}{DF} =  
  \ol{\op{conv}} \bigcup_{z\in[x,y]} \df{\Dens_z}{DF} =
  \op{conv} \: \partial F([x,y]) \,  
\end{equation*}
and the assertion follows.
\end{proof}

With the representation from Theorem~\ref{gd-s2} we can give new short 
proofs for some elementary calculus rules (cf. also
\cite[p. 444]{pales}). 

\begin{proposition} \label{gd-s5}
Let $F,G:\dom\to\R^m$ be locally Lipschitz continuous.
\bgl
\item
{\rm (Scalar multiple)} For $x\in\dom$ we have
\begin{equation*}
  \partial(sF)(x) = s\partial F(x)
  \qmq{for all} s\in\R\,.
\end{equation*}
\item
{\rm (Sum rule)} For $x\in\dom$ we have
\begin{equation*}
  \partial\big(F(x) + G(x)\big)(x) \:\subset\:
  \partial F(x) + \partial G(x) \,.
\end{equation*}
If $F$ or $G$ is strictly differentiable at $x$, then we have equality. 

\item
{\rm (Product rule)} For $x\in\dom$ we have
\begin{equation*}
  \partial\big(F\cdot G\big)(x) 
  \subset F(x)^T\partial G(x) + G(x)^T\partial F(x) \,
\end{equation*}
(where $^T$ denotes the transpose).
If $F$ or $G$ is strictly differentiable at $x$, then we have equality. 

\item
{\rm (Quotient rule)} For $x\in\dom$ and $g:\dom\to\R$ locally Lipschitz
continuous with $g(x)\ne 0$, we have
\begin{equation*}
  \partial\big(\tfrac{1}{g}F\big)(x) \subset 
  \frac{g(x)\partial F(x)-F(x)\otimes\partial g(x)^T}{g(x)^2} \,
\end{equation*}
(where $\otimes$ denotes the tensor product).
If $F$ or $g$ is strictly differentiable at $x$, then we have equality. 

\el
\end{proposition}

\begin{proof}
For (1) we notice that $D(sF)=sDF$ on $D_F$, i.e. 
$\lem$-a.e. on $\dom$. Thus
\begin{equation*}
  \big\langle \Dens_x, D(sF) \big\rangle =
  s\big\langle \Dens_x, DF \big\rangle
\end{equation*}
and, by \reff{gd-e1}, the statement follows. 

For (2) we have $D(F+G)=DF+DG$ on $D_F\cap D_G$,
i.e. $\lem$-a.e. on $\dom$. Then 
\begin{eqnarray*}
  \big\langle \Dens_x,D\big(F+G\big) \big\rangle 
&=&
  \big\langle \Dens_x, DF + DG \big\rangle \\
&\subset&
  \langle\Dens_x,DF\rangle + \langle\Dens_x,DG\rangle  \,
\end{eqnarray*}
and the stated inclusion follows. If e.g. $F$ is strictly
differentiable at $x$, then $\df{\Dens_x}{DF}$ is a singleton by 
Proposition~\ref{gd-s3} and the inclusion becomes an equality.

For (3) we have by continuity that $F-F(x)=0$ i.m. $\!\mu$ 
for any $\mu\in\Dens_x$. Then, 
\begin{equation*} \label{gd-s5-10}
  \sI{x}{F^TDG}{\mu} = 
  \sI{x}{\big(F^T-F(x)^T\big)DG}{\mu} + F(x)^T\sI{x}{DG}{\mu}
  = F(x)^T\sI{x}{DG}{\mu}\,.
\end{equation*}
Moreover, $D(F\cdot G)=F^TDG+G^TDF$ on $D_F\cap D_G$, i.e. 
$\lem$-a.e.on $\dom$. Consequently,
\begin{eqnarray*}
  \big\langle \Dens_x,D(F\cdot G) \big\rangle 
&=&
  \big\langle \Dens_x,F^TDG+G^TDF \big\rangle \\
&\subset&
  F(x)^T \langle \Dens_x,DG \rangle + G(x)^T \langle \Dens_x,DF \rangle 
\end{eqnarray*}
and the stated inclusion follows. If e.g. $F$ is strictly
differentiable at $x$, then 
\begin{equation*}
  \df{\Dens_x}{G^TDF} = G(x)^T\df{\Dens_x}{DF} 
\end{equation*}
is a singleton by Proposition~\ref{gd-s3} and we get equality.

For (4) we use
\begin{equation*}
  D\big(\tfrac 1g F\big) = \tfrac 1g DF + F\otimes D\big(\tfrac 1g\big) =
  \frac{gDF-F\otimes Dg^T}{g^2} \qmq{on} D_g\cap D_F
\end{equation*}
and argue analogously as in (3).
\end{proof}

Instead of using $F-F(x)=0$ i.m. $\!\mu$ in the proof of (3), 
we can alternatively apply \reff{ep-s22-4} to the product
$F^TDG$. 
Next we provide some chain rule, which extends the scalar version  
$m=1$ contained in \cite[p.~75]{clarke}. 
We say that $H:\R^n\to\R^k$ is {\it surjective near} $x$ if for any $\ep>0$
there is some $\de>0$ such that
\begin{equation*}
  H:B_{\ep}(x)\to B_{\de}(H(x)) 
\end{equation*}
is surjective.

\begin{proposition} \label{gd-s7}
{\rm (Chain rule)} Let $\dom\subs\R^n$ be open, let $H:\dom\to\R^k$ and 
$G:\R^k\to\R^m$ be locally Lipschitz continuous, and let $x\in\dom$. 
Then $F:\dom\to\R^m$ with $F=G\circ H$ is locally Lipschitz continuous and
\begin{equation*}
  \partial F(x) \subset 
  \op{conv} \big(\,\partial G(H(x))\,\partial H(x)\,\big). 
\end{equation*}
We have the following special cases:
\bgl
\item
If $G$ is strictly differentiable at $H(x)$, then
\begin{equation*}
  \partial F(x) = 
  D_sG(H(x))\,\partial H(x)   \,. 
\end{equation*}

\item
If $H$ is strictly differentiable at $x$, then
\begin{equation*}
  \partial F(x) \subs \pa G(H(x))\, D_sH(x)   \,. 
\end{equation*}

\item
If $H$ is strictly differentiable at $x$ and surjective near $x$, then
\begin{equation*}
  \partial F(x) = \partial G(H(x))\, D_sH(x) \,.
\end{equation*}
\el
\end{proposition} 
\noi
Notice that $\partial G(H(x))\,\partial H(x)$ is compact and, thus, also its
convex hull. The assumption on $H$ in (3) to be surjective near $x$ can be
ensured by the regularity of the matrix $D_sH(x)$ (i.e. if the rank of
$D_sH(x)$ is $k$). 
Let us furthermore mention that the requirement on $H$ 
in \cite[p.~45]{clarke}, to map any neighborhood of $x$ to a set which is 
dense in a neighborhood of $H(x)$, readily implies that $H$ is
surjective near $x$ by continuity. 

\begin{proof}
The local Lipschitz continuity of $F$ follows easily.
Let $y\in\dom$, let $v\in\R^n\setminus\{0\}$, 
and let $t_k\downarrow 0$. By Theorem~\ref{app-s30} there are 
$\mu_k\in\Dens_{[H(y),H(y+t_kv)]}$ such that
\begin{eqnarray}
  \frac{F(y+t_kv)-F(y)}{t_k}
&=& \label{gd-s7-7}
  \frac{G(H(y+t_kv)) - G(H(y))}{t_k} \\[1mm]
&=&  \nonumber
  \sI{[H(y),H(y+t_kv)]}{DG}{\mu_k}\; 
  \frac{H(y+t_kv)-H(y)}{t_k} \,. 
\end{eqnarray}
By Proposition~\ref{dm-s4f}, 
the $\mu_k$ have a weak$^*$ accumulation point $\mu\in\Dens_{H(y)}$.
Then, for $y$ in $D_F\cap D_H$, we can go to the limit, at least for a 
suitable subsequence denoted the same way, and obtain
\begin{equation*}
  DF(y) v = \sI{H(y)}{DG}{\mu} \; DH(y) v \,. 
\end{equation*}
The arbitrariness of $v$ and the definition of $\pa G(H(y))$ imply 
for any $y\in D_F\cap D_H$ that
\begin{equation} \label{gd-s7-5}
  DF(y) = \sfG_y \: DH(y) \qmq{for some}  \sfG_y\in\pa G(H(y))\,. 
\end{equation}
Now we use \reff{gd-s2-1} and Proposition~\ref{gd-s0} (3) to study possible
limits in \reff{gd-s7-5} for $y\to x$. For any sequence 
$y_k\to x$ there is a subsequence, denoted the same way, such that
\begin{equation*}
  DF(y_k)\to:\sfF\in\pa F(x)\,, \quad DH(y_k)\to:\sfH\in\pa H(x)\,, \quad
  \sfG_{y_k}\to:\sfG\in\pa G(H(x))\,.
\end{equation*}
Taking into account the convex hull in \reff{gd-s2-1}, we obtain the
first assertion. For the special case (1) we use that always 
$\sfG_{y_k}\to D_sG(H(x))$ by Proposition~\ref{gd-s3}. 
For (2) we always have
$DH(y_k)\to D_sH(x)$, while the limits of the $\sfG_{y_k}$ 
might not catch all elements in $\pa G(H(x))$. 

For (3) we use (2) and it remains to show the opposite inclusion. 
For that we consider $z_k\in D_G$ with $z_k\to H(x)$.
Since $H$ is surjective near~$x$, possibly for a subsequence, 
there are  $y_k\to x$ with $z_k=H(y_k)$. Then, again up to a subsequence, 
\begin{equation} \label{gd-s7-9}
  DG(z_k)\to:\sfG\in\pa G(H(x)) \,.
\end{equation}
For $L=D_sH(x)$ and fixed $v\in\R^n$, we can choose $t_k\dto 0$ such that
\begin{equation*} 
  \Big|\frac{G(z_k+t_k Lv)-G(z_k)}{t_k} - DG(z_k)\,Lv\Big| 
  \le \tfrac1k \,.
\end{equation*}
Since $H$ is strictly differentiable, we have, up to a subsequence and
for a Lipschitz constant $c>0$, that
\begin{equation*}
  \Big| \frac{G(H(y_k+t_kv)) - G(z_k + t_kLv)}{t_k} \Big| \le
  c\, \Big| \frac{H(y_k+t_kv) - H(y_k)-t_k Lv}{t_k}\Big| \le \tfrac1k \,.
\end{equation*}
Let us now consider \reff{gd-s7-7} with $y=y_k$. 
On the left hand side we apply Theorem~\ref{app-s30} with 
$\hat\mu_k\in\LDens_{[y_k,y_k+t_kv]}$, that 
have a weak$^*$ accumulation point $\hat\mu\in\Dens_x$. On the right hand side
we replace $G(H(y_k+t_kv))$ with $G(z_k + t_kLv)$. Then we take the
limit for a suitable subsequence, denoted the same way, and we use our 
estimates to get
\begin{eqnarray*}
  \sI{x}{DF}{\hat\mu} \: v 
&=&
  \lim_{k\to\infty}\sI{[y_k,y_k+t_kv]}{DF}{\hat\mu_k}\:v 
\:=\: 
  \lim_{k\to\infty}  \frac{G(z_k + t_kLv) - G(z_k)}{t_k} \\[2mm]
&=&
  \lim_{k\to\infty} DG(z_k)\,Lv = \sfG D_sH(x)\, v\,
\end{eqnarray*}
for all $v\in\R^n$. Hence
\begin{equation*}
  \sfG D_sH(x) = \sI{x}{DF}{\hat\mu} \in\pa F(x) \,.
\end{equation*}
Using \reff{gd-s2-1}, \reff{gd-s7-9}, and the arbitrariness of the sequence
$z_k\to H(x)$, we can take the convex hull to obtain the desired inclusion,
which verifies case (3).  
\end{proof}

\bibliographystyle{plain}

\vspace{5mm}

\noi
Friedemann Schuricht \\ 
TU Dresden \\
Fakultät Mathematik \\
01062 Dresden \\
Germany\\
email: {\it friedemann.schuricht@tu-dresden.de}

\end{document}